\documentclass{amsart}
\title[Balanced alg.~unknotting, linking forms, surfaces in 3D \& 4D]{Balanced algebraic unknotting, linking forms, and surfaces in three- and four-space}
\author{Peter Feller}
\address{ETH Zurich, R\"amistrasse 101, 8092 Zurich, Switzerland}
\email{\myemail{peter.feller@math.ch}}
\author{Lukas Lewark}
\address{Faculty of Mathematics, University of Regensburg, 93053 Regensburg, Germany}
\email{\myemail{lukas@lewark.de}}
\keywords{Blanchfield pairing, linking pairings, branched coverings, unknotting number, slice genus}
\usepackage[latin1]{inputenc}
\usepackage[final]{microtype}
\usepackage{amssymb,amsmath,caption,array,multirow,pifont,amsfonts,amsthm,bbm,graphicx,psfrag,pstricks,mathrsfs,pifont,enumerate,enumitem,pdfpages,xcolor,thmtools,thm-restate,picinpar,mathtools}
\usepackage[pagebackref]{hyperref}
\usepackage{hyperref}

\renewcommand*{\backrefalt}[4]{%
\tiny
\ifcase #1 %
No citations.%
\or
Cited on page~#2.%
\else
Cited on pages~#2.%
\fi
}

\usepackage[nameinlink,nosort]{cleveref}

\Crefname{subsection}{Section}{Sections}
\Crefname{prop}{Proposition}{Propositions}
\crefname{equation}{}{}
\let\cref\Cref
\newcommand{\myemail}[1]{\href{mailto:#1}{#1}}

\hypersetup{colorlinks={true},linkcolor={black},citecolor={black},filecolor={black},urlcolor={black},pdfauthor={\authors},pdftitle={\shorttitle}}

\newcommand{\qua}{\hskip 0.4em \ignorespaces}
\def\arxiv#1{\relax\ifhmode\unskip\qua\fi
\href{http://arxiv.org/abs/#1}%
{\tt arXiv:\penalty -100\unskip#1}}

\def\MR#1{\relax\ifhmode\unskip\qua\fi
\href{https://mathscinet.ams.org/mathscinet-getitem?mr=#1}{\tt MR#1}}
\def\ZB#1{\relax\ifhmode\unskip\qua\fi
\href{https://zbmath.org/?q=an:#1}{\tt Zbl #1}}
\def\xox#1{\csname xx#1\endcsname}

\newcolumntype{x}[1]{>{\centering\arraybackslash\hspace{0pt}}p{#1}}
\makeatletter
\newcommand{\addresseshere}{%
  \enddoc@text\let\enddoc@text\relax
}
\makeatother

\declaretheorem[numberwithin=section,name=Theorem]{thm}
\newtheorem{corollary}[thm]{Corollary}
\newtheorem{lemma}[thm]{Lemma}
\newtheorem{definition}[thm]{Definition}
\newtheorem{prop}[thm]{Proposition}
\newtheorem{claim}[thm]{Claim}

\newtheorem{observation}[thm]{Observation}

\theoremstyle{remark}
\newtheorem{example}[thm]{Example}
\newtheorem{rmk}[thm]{Remark}

\newenvironment{Rmk}{\begin{rmk}\rm}{\end{rmk}}

\def\N{\mathbb{N}}
\def\Q{\mathbb{Q}}

\def\Z{\mathbb{Z}}

\def\epsilon{\varepsilon}

\def\gtop{{g_{\rm{top}}}}
\def\gs{{g_{\rm{smooth}}}}
\def\Ann{{\rm{Ann}}}
\def\Adj{{\rm{Adj}}}
\def\Ord{{\rm{Ord}}}
\def\ga{{g_{\rm{alg}}}}

\def\gst{{g_{\Z}}}

\DeclareMathOperator\snZ{sn_{\Z}}
\DeclareMathOperator\sn{sn}
\DeclareMathOperator{\Arf}{Arf}

\DeclareMathOperator{\Hom}{Hom}
\DeclareMathOperator{\coker}{coker}
\DeclareMathOperator{\Mat}{Mat}

\DeclareMathOperator{\Bl}{Bl}

\DeclareMathOperator{\im}{im}
\DeclareMathOperator{\cyc}{cyc}

\DeclareMathOperator{\lkh}{lkh}
\DeclareMathOperator{\lke}{lke}
\DeclareMathOperator{\PD}{PD}

\DeclareMathOperator{\id}{id}
\newcommand{\legendre}[2]{\ensuremath{\Bigl(\frac{#1}{#2}\Bigr)}}

\newcommand{\subsectionpdfbookmark}[2]{%
\stepcounter{subsection}%
\pdfbookmark[2]{\thesubsection. #2}{pdfbookmarksubsec:\thesubsection}%
\addtocounter{subsection}{-1}%
\hypersetup{bookmarksdepth=-2}%
\subsection{#1}%
\mbox{}%
\hypersetup{bookmarksdepth=2}%
}

\newcommand{\sectionpdfbookmark}[2]{%
\stepcounter{section}%
\pdfbookmark[1]{\thesection. #2}{pdfbookmarksec:\thesection}%
\addtocounter{section}{-1}%
\hypersetup{bookmarksdepth=-2}%
\section{#1}%
\mbox{}%
\hypersetup{bookmarksdepth=2}%
}

\usepackage{tikz-cd}
\graphicspath{{figs/}}
\begin{document}
\subjclass{57K10, 57K14, 57M12, 11E39}
\maketitle
\begin{abstract}
We provide three $3$--dimensional characterizations
of the \emph{$\Z$--slice genus of a knot}, the minimal genus of a locally-flat surface in 4--space cobounding the knot whose complement has cyclic fundamental group:
in terms of balanced algebraic unknotting, in terms of Seifert surfaces, and in terms of presentation matrices
of the Blanchfield pairing.
This result generalizes to a knot in an integer homology 3--sphere and surfaces in certain simply connected signature zero 4--manifolds cobounding this homology sphere.
Using the Blanchfield characterization, we obtain effective lower bounds for the $\Z$--slice genus from the linking pairing
of the double branched covering of the knot.
In contrast, we show that for odd primes $p$,
the linking pairing on the first homology of the $p$--fold branched covering is determined up to isometry by the action of the deck transformation group on said first homology.
As an application of the new upper and lower bounds, we complete the calculation of the $\Z$--slice genus for all prime knots with crossing number up to~12.
\end{abstract}

\section{Introduction}
The main result of this paper is the following.
\begin{thm}\label{thm:main}
For a knot $K$---a smooth, oriented, non-empty, and connected 1--submanifold of $S^3$---and a non-negative integer $g$, the following are equivalent.
\begin{enumerate}[itemsep=1ex]
  \item\label{item:Zslicesurface} There exists an oriented compact connected surface $F$ of genus~$g$
properly embedded and locally flat in $B^4$ with boundary $K\subseteq S^3=\partial B^4$
such that $\pi_1(B^4\setminus F)\cong \Z$.
  \item\label{item:Seifertsurface} There exists a smooth oriented compact connected surface of genus~$g$ in $S^3$ with two boundary components, %
one of which is the knot~$K$ and the other %
a knot with Alexander polynomial~$1$.
  \item \label{item:unknotting}The knot $K$ can be turned into a knot with Alexander polynomial $1$ by changing $g$ positive and $g$ negative crossings.
  \item \label{item:Blanchfield}The Blanchfield pairing of $K$ can be presented by a Hermitian matrix $A(t)$ of size
 $2g$ over $\Z[t^{\pm1}]$ such that the integral symmetric matrix $A(1)$ has signature zero.
\end{enumerate}
\end{thm}

Here, the \emph{Alexander polynomial} and the \emph{Blanchfield pairing} are the classical knot invariants introduced by their respective eponyms~\cite{Alexander_28_TopInvsOfKnotsAndLinks,blanchfield}.
The Alexander polynomial of a knot $K$ is most quickly defined as the order of the \emph{Alexander module} $H_1(S^3\setminus K; \Z[t^{\pm 1}])$ of $K$,
which is the $\Z[t^{\pm 1}]$--module given as the first integral homology group of the infinite cyclic covering of $S^3\setminus K$ with $\Z[t^{\pm 1}]$--module structure given by $t$ acting as the group isomorphism induced by a generator of the deck transformation group. The Blanchfield pairing is a Hermitian pairing on $H_1(S^3\setminus K;\Z[t^{\pm 1}])$ taking values in $\Q(t)/\Z[t^{\pm 1}]$.
We refer the reader to \Cref{sec:prel} for more detailed definitions.
By \emph{changing a positive (negative) crossing} of a knot~$K$, we mean the $-1$--framed ($1$--framed) Dehn surgery on the boundary of a \emph{crossing disk}, i.e.~a smooth closed 2--disk $D\subset S^3$ that intersects $K$ exactly twice, only in the interior, transversely, and such that the two intersection points have opposite induced orientations.

\begin{figure}[h]
\centering
\includegraphics{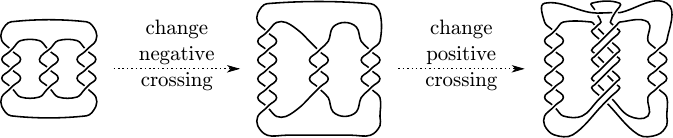}
\caption{The genus-one pretzel knot $(3,3,3)$ (on the left) can be turned into a knot with Alexander\break polynomial $1$ (on the right) by changing one negative  and one positive crossing.}
\label{fig:333}
\end{figure}
Before discussing context, applications, an outline of the proof, and a generalization to more general ambient 3-- and 4--manifolds of \Cref{thm:main}, we note that
this paper naturally splits into two, essentially independent, parts. A first part contains the proof of \cref{thm:main}. %
A second part is concerned with providing obstructions for knots to satisfy~\eqref{item:Zslicesurface} of \Cref{thm:main} for small $g$ using~\eqref{item:Blanchfield}. Concretely, we provide obstructions by specializing the Blanchfield form to the linking form on the first homology of the $p$--fold branched covering of the knot for $p$ a prime. For $p=2$, we provide an easily applicable criterion (see \cref{prop:jacobi}), which turns out to be effective for knots in the knot tables.
The same strategy does not give interesting obstructions for odd primes $p$. This is explained by our second result. %
\begin{restatable}{thm}{thmalliso}
\label{thm:allisometric}
Let $p$ be an odd prime, $\Sigma_p(K)$ the $p$--fold branched covering of a knot~$K$, and $N_p(K)\coloneqq H_1(\Sigma_p(K); \Z[t]/(t^p-1))$
the first homology group of $\Sigma_p(K)$, on which~$t$ acts as a generator of the deck transformation group.
Then, any two non-degenerate Hermitian pairings on $N_p(K)$ are isometric.
In particular, if for two knots $K, K'$, the $\Z[t]/(t^p-1)$--modules $N_p(K)$ and $N_p(K')$ are isomorphic,
then the linking pairings on these modules are isometric.
\end{restatable}
The techniques for this second part involve elementary number theory, which are rather different from the low-dimensional topology arguments employed in the rest of the text: we use quadratic reciprocity and Dirichlet's prime number theorem, and, for our result when $p$ is odd, we generalize parts of the proof of Wall's classification of symmetric pairings on finite abelian groups with odd order \cite{MR0156890} to modules over Dedekind domains of a certain order.
This second part is contained in \Cref{sec:Comp}, and outlined in more detail in \Cref{subsecintro:lkf}.
The first part can be read independently of it.

\subsection{The genus zero case}
For $g=0$, the conditions (2), (3) and (4) of \Cref{thm:main} are all immediately seen to be equivalent to $K$ having Alexander polynomial 1. In other words, for $g=0$, \Cref{thm:main} is simply restating the following celebrated application of Freedman's disk embedding theorem \cite{Freedman_82_TheTopOfFour-dimensionalManifolds}.
\begin{thm}[{\cite[11.7B]{FreedmanQuinn_90_TopOf4Manifolds}; see also~\cite[Appendix]{GaroufalidisTeichner_04_OnKnotswithtrivialAlex}}]\label{thm:Alex1<=>Zslice} A knot $K$ has Alexander polynomial $1$ if and only if there exists a properly embedded locally-flat disk $D$ in $B^4$ with boundary $K\subseteq S^3=\partial B^4$
such that $\pi_1(B^4\setminus D)\cong \Z$.\qed\end{thm}
We do not claim to reprove \Cref{thm:Alex1<=>Zslice}; in fact, its only-if-direction (the part based on the disk embedding theorem) is the main input for the proof of $\eqref{item:Seifertsurface}\Rightarrow\eqref{item:Zslicesurface}$.
We understand \Cref{thm:main} as a quantitative version of \Cref{thm:Alex1<=>Zslice} in that it characterizes the existence of a genus $g$ surface in $B^4$ (rather than a disk) in terms of classical, $3$--dimensional knot invariants.
It does so in three a priori different ways:
via the 3D--cobordism distance \eqref{item:Seifertsurface},
balanced algebraic unknotting distance \eqref{item:unknotting},
and a condition on presentation matrices of the Blanchfield pairing \eqref{item:Blanchfield}.

\subsection{Three-dimensional consequences}
While the principal motivation for our main result is the 3D--characterization of an a priori 4D--quantity,
the different 3D--characteri\-zations \eqref{item:Seifertsurface}, \eqref{item:unknotting}, and~\eqref{item:Blanchfield} of \eqref{item:Zslicesurface} in \Cref{thm:main} yield 3D--consequences such as the following.
\begin{corollary}\label{cor:genusoneknots}
A knot of genus one, i.e.~a knot arising as the boundary of a once-punctured torus embedded in~$S^3$,
can be turned into a knot with Alexander polynomial~$1$ by changing one positive and one negative crossing.

More generally, this can be achieved
for any knot that arises as one of the boundary components of a twice-punctured torus embedded in~$S^3$,
whose other boundary component has Alexander polynomial $1$.
\end{corollary}
\begin{proof}
The first claim is indeed a special case of the second claim:
puncturing a genus one surface with boundary knot $K$
yields a twice punctured torus in $S^3$ with boundary consisting of $K$ and an unknot (which of course has Alexander polynomial~1).
The second claim follows from \Cref{thm:main} \eqref{item:Seifertsurface} $\Rightarrow$ \eqref{item:unknotting} for $g=1$.
\end{proof}
Compare this to related results in \cite{MR1297516,livingston_unknotting}.
\cref{cor:genusoneknots} stands in %
contrast with the existence of knots with genus one that cannot be unknotted by changing one positive
and one negative crossing, such as $P(p,q,r)$ pretzel knots with $p\geq 1, q\geq 3, r\geq 3$ odd \cite{owens}.
See \Cref{fig:333} for an example, where we explicitly provide the two crossing changes that must exist by \Cref{cor:genusoneknots} for the pretzel knot $P(3,3,3)$. We note that $P(3,3,3)$ is known to have unknotting number $u(P(3,3,3))=3>2$~\cite{owens}.

\subsection{Four-dimensional consequences}\label{subsec:galg}
To facilitate the discussion, let us define the \emph{$\Z$--slice genus} $\gst(K)$ of a knot $K$ as the smallest integer arising as the genus of a
\emph{$\Z$--slice surface} for $K$, i.e.\ an oriented compact surface $F$,
properly embedded and locally flat in $B^4$ with boundary $K\subseteq S^3$ and $\pi_1(B^4\setminus F)\cong \Z$.
In other words, $\gst(K)$ is the minimal $g$ such that~\eqref{item:Zslicesurface} holds.%

In previous work~\cite{FellerLewark_16}, the authors defined the \emph{algebraic genus} $\ga(K)$ of a knot $K$ as the smallest non-negative integer $g$ satisfying the following condition:\medskip
\emph{
\begin{enumerate}[itemsep=1ex]
\setcounter{enumi}{4}
\item\label{item:ga}
The knot $K$ admits a $2(g+h)$--dimensional Seifert matrix $M$ for some $h \geq 0$ such that the
upper-left square $2h$--dimensional submatrix $N$ of $M$ satisfies $\det(t\cdot N - N^{\top}) = t^{h}$.
\end{enumerate}}\medskip
One motivation for this definition was to obtain easily calculable upper bounds for $\gst$ and the topological slice genus $\gtop$.
Indeed, it was shown in \cite{FellerLewark_16} that $\eqref{item:ga} \Leftrightarrow \eqref{item:Seifertsurface} \Rightarrow \eqref{item:Zslicesurface}$, i.e.\ $\ga \geq \gst$.
Now, as a consequence of \Cref{thm:main}, \eqref{item:ga} is equivalent to (1)\,--\,(4), and we have the following
\begin{corollary}\label{cor:galg=gZ}
The algebraic genus and the $\Z$--slice genus agree for all knots.\qed
\end{corollary}

The algebraic genus %
and the Seifert matrix techniques related to \eqref{item:ga} will not be used here and
the paper at hand can be read without any knowledge of~\cite{FellerLewark_16}.
On the subject of Seifert matrices, we point out the following consequence of our main result. It follows from~\Cref{cor:galg=gZ} and the fact (see~\cite[Proposition~10]{FellerLewark_16}) that the algebraic genus is a classical knot invariant. A knot invariant is a \emph{classical knot invariant}, if for every knot $K$ it is determined by the isometry class of the Blanchfield pairing of $K$
(or equivalently, by the $S$-equivalence class of the Seifert forms of $K$).

\begin{corollary}\label{cor:gZisclassical}
The $\Z$--slice genus of a knot $K$ is a \emph{classical knot invariant}.\qed %
\end{corollary}
This distinguishes $\gst$ from other knot genera such as the three-dimensional knot genus, the smooth slice genus, or the topological slice genus, none of which are classical.

A topological slice surface $F\subset B^4$ is a (topological) \emph{superslice surface} if its double $S\subset S^4$ (the surface $S$ in $S^4$ given by gluing together $(B^4,F)$ with an oppositely oriented copy of itself) is unknotted, i.e.~the boundary of an embedded locally flat handlebody. Combining \Cref{thm:main} with a recent argument by Chen~\cite{Chen_18}, we find that whenever a knot $K$ has a $\Z$--slice surface of genus $g$, then it has a superslice surface of genus $g$. Put differently, for the \emph{topological superslice genus} of a knot $K$, the smallest integer that arises as the genus of a superslice surface of $K$, we have the following.
\begin{corollary}\label{cor:gZ=gsuperslice}
For all knots $K$, the $\Z$--slice genus of $K$ equals the topological superslice genus of $K$.
\end{corollary}
\begin{proof}
By the Seifert-van Kampen theorem, superslice surfaces are $\Z$--slice surfaces. Hence, the superslice genus of a knot is greater than or equal to its $\Z$--slice genus. We discuss the other inequality.

If a knot $K$ has a $\Z$--slice surface of genus $g$, then by \Cref{thm:main}$(1)\Rightarrow (2)$, we find a
smooth connected oriented surface $C$ of genus $g$ with two boundary components, one of which is $K$ and the other is a knot $J$ with Alexander polynomial~$1$. One may further arrange $C$ to be $H_1$--null (see \cref{lem:nice3dcobo}) and hence one finds a Seifert surface $F_J$ for $J$, i.e.~a connected smooth oriented surface in $S^3$ with boundary $J$,
such that the interior of $F_J$ is disjoint from $C$.
Since the knot $J$ has Alexander polynomial 1, it has a $\Z$--slice disk $D$.

We now follow Chen to see that $C\cup D$ becomes a superslice surface when properly pushed into $B^4$.
For details on the following argument, we refer to his text.
Let $B_J$ be a $3$--ball in $S^4$ that intersects $S^3$ transversely with boundary the $2$--sphere given as the double of $D$ (known to exist by \cite[11.7A]{FreedmanQuinn_90_TopOf4Manifolds}).
Crucially, Chen shows that $B_J$ can be isotoped (fixing its boundary) such that the transverse intersection of $B_J$ with $S^3$ is a surface $F\coloneqq B_J\cap S^3$ consisting of a Seifert surface $F'_J$ for $J$ that is a stabilization of $F_J$ and several closed components such that $C\cap F=J$~\cite[Lemma~3.8]{Chen_18}.
With the isotoped $B_J$, denoted again $B_J$ by abuse of notation, at hand, one can check that $C\cup D$ is (after being properly pushed into $B^4$) a superslice surface of genus $g$. Indeed, let $H$ be the handlebody in $S^4$ given as $\mathrm{thick}(C)\cup B_J$, where $\mathrm{thick}(C)$ is the image of an embedding of $C\times[-\epsilon,\epsilon]\to S^4$ mapping $C\times\{0\}$ to $C$ and $J\times[-\epsilon,\epsilon]$ to $\partial B_J$. The boundary $\partial H$ is the double of (a properly pushed into $B^4$ copy of) $C\cup D$; compare \cite[Proof of Theorem~1.5]{Chen_18}.
\end{proof}
We note the following subtlety in the above proof. One is tempted to show the stronger statement that every $\Z$--slice surface is a superslice surface. Promisingly,
a Seifert-van Kampen calculation reveals that the double $S\subset S^4$ of a $\Z$--slice surface satisfies $\pi_1(S^4\setminus S)\cong\Z$.
However, as far as the authors know, it remains open whether closed surfaces $S$ with $\pi_1(S^4\setminus S)\cong\Z$ are always unknotted.
Of course, if $S$ is a sphere this is known to be a consequence of the disk embedding theorem~\cite[Theorem~11.7A]{FreedmanQuinn_90_TopOf4Manifolds}), and, while this paper was under consideration for publication, Conway-Powell resolved the case of $S$ of genus 3 or larger~\cite{ConwayPowell_20}.
The use of our main theorem (switching to a $\Z$--slice surface given as the union of a $\Z$--slice disk and a surface as in (2)) avoids having to deal with the non-zero genus case.

Next, we discuss applications of \eqref{item:Zslicesurface} $\Leftrightarrow$ \eqref{item:Blanchfield} of \Cref{thm:main}.
As mentioned above, we use \eqref{item:Zslicesurface} $\Rightarrow$ \eqref{item:Blanchfield} and then specializations of \eqref{item:Blanchfield} to give lower bounds for $\gst$; see \Cref{subsecintro:lkf}.
In a different direction, let us make a note of the following heuristic:
if some knot theoretic construction has a well-understood effect on the Blanchfield pairing, then one can use \Cref{thm:main} to better understand its effect on $\gst$. As a specific example, the following inequality for satellite knots, the main result of~\cite{FellerMillerPinzon}, has a rather pleasing and fast proof using \eqref{item:Zslicesurface} $\Leftrightarrow$ \eqref{item:Blanchfield}.
\begin{thm}[{\cite[Theorem~1.2]{FellerMillerPinzon}; see also \cite[Theorem~1.4]{zbMATH07442350}}]\label{thm:FMP}
Every satellite knot $P(K)$ satisfies \[\gst(P(K))\leq \gst(P(U))+\gst(K),\] where $P$ and $K$ denote the pattern knot and companion knot of $P(K)$, respectively.
\end{thm}
We omit precise definitions of the satellite operation and also keep the following proof short (but complete), since \cref{thm:FMP} is not %
 an original result of this text.
\begin{proof}
By applying~\eqref{item:Zslicesurface} $\Rightarrow$ \eqref{item:Blanchfield} to the knots $K$ and $P(U)$, we know that there exist presentation matrices $A_K(t)$ and $A_P(t)$ of $K$ and $P(U)$, respectively, satisfying~\eqref{item:Blanchfield}, i.e.\ $A_K(t)$ and $A_P(t)$ are of size $2\gst(K)$ and $2\gst(P(U))$, respectively, and $\sigma(A_K(1)) = \sigma(A_P(1)) = 0$.
By a result of Livingston and Melvin~\cite[Theorem~2]{LivingstonMelvin85},
the Blanchfield pairing of $P(K)$ is presented by the matrix $A_{P(K)}(t) \coloneqq A_K(t)\oplus A_P(t^w)$, where $w$ denotes the algebraic winding number of the pattern knot $P$. However, $A_{P(K)}(t)$ is a presentation matrix of size $2\gst(K)+2\gst(P(U))$
satisfying~\eqref{item:Blanchfield}, since $\sigma(A_{P(K)}(1)) = \sigma(A_K(1)) + \sigma(A_P(1)) = 0$.
Thus, $\gst(P(K))\leq \gst(P(U))+\gst(K)$ by~\eqref{item:Blanchfield}~$\Rightarrow$~\eqref{item:Zslicesurface}.
\end{proof}

Finally, we discuss two four-dimensional knot invariants, the slice genus and the stabilizing number, showing that their natural $\Z$--analogs coincide.
In~\cite{nagel_conway}, Conway and Nagel introduce the \emph{stabilizing number} $\sn(K)$ of a knot~$K$ with $\Arf$ invariant zero, the minimal $c$ such that a null-homologous locally flat slice disk for $K$ exists in $B^4\# (S^2\times S^2)^{\# c}$, and they show the following. \[\text{\emph{For all knots $K$ with $\Arf(K) = 0$, $\sn(K)\leq\gtop(K)$ \cite{nagel_conway}.}}\]
While the classical Levine-Tristram signature lower bounds on $\gtop$ also hold for~$\sn$, 
Casson-Gordon invariants (which are not classical in the sense that they do not only depend on the $S$-equivalence class of knots) may be employed to show that $\gtop$ and $\sn$ differ substantially \cite{nagel_conway}. As a consequence of a generalization of our main result (stated below as~\Cref{thm:mains2s2}), we find that the corresponding notion of {$\Z$--stabilizing number} is equal to the {$\Z$--slice genus for all knots with $\Arf$ invariant zero. Here, we define the \emph{$\Z$--stabilizing number}, denoted $\snZ(K)$, like $\sn$ with the additional assumption that the slice disk $D$ is a $\Z$--slice disk in $B^4\# (S^2\times S^2)^{\# c}$. The condition that the disk is null-homologous in $B^4\# (S^2\times S^2)^{\# c}$ turns out to be superfluous: it is implied by its complement having fundamental group $\Z$; see \Cref{lem:Zslicesurfacesarenullhomo}.
\begin{corollary}\label{cor:snZ=gZ}
For all knots $K$ with $\Arf(K) = 0$, $\snZ(K)=\gst(K)$.\qed
\end{corollary}
It was surprising to the authors that two different 4-dimensional knot invariants ($\gtop$ and $\sn$) with very subtle behavior, become the same if a regularity assumption (complements have fundamental group $\Z$) is added.
We also note 
that $\gtop(K)=0$ if and only if $\sn(K)=0$, by definition. So, the difference between $\sn$ and $\gtop$ manifests itself only for the ``non-genus zero case''. %

\subsection{Perspective: The $\Z$--slice genus is a balanced algebraic unknotting number}
By \Cref{thm:main},
the $\Z$--slice genus of a knot $K$ can be seen as a variation of the \emph{algebraic unknotting number} $u_a(K)$
introduced by Murakami \cite{Murakami} and studied amongst others by Fogel \cite{fogel}, Saeki \cite{saeki} and by Borodzik and Friedl
\cite{BorodzikFriedl_15_TheUnknottingnumberAndClassInv1,BorodzikFriedl_14_OnTheAlgUnknottingNr}.
Indeed, compare to \eqref{item:unknotting} that
$u_a(K)$ equals the minimal number of crossing changes necessary to convert $K$ into a knot with Alexander polynomial~1.
One could say $u_a(K)$ and $\gst(K)$ are respectively the \emph{unsigned} and \emph{balanced} Gordian distance between $K$ and Alexander polynomial 1 knots.
From this perspective, the inequalities $\gst(K) \leq u_a(K) \leq 2\gst(K)\leq \deg (\Delta_K)$ first proven in \cite{FellerLewark_16} are now quite evident.

Of course, one may also more generally consider \emph{signed} algebraic unknotting, and ask whether for given $p, n \geq 0$, the knot $K$ can be turned into an Alexander polynomial 1 knot by changing $p$ positive and $n$ negative crossings.
This question admits an answer in terms of presentation matrices of the Blanchfield pairing
 \cite{BorodzikFriedl_14_OnTheAlgUnknottingNr} generalizing \eqref{item:Blanchfield},
which we cite and use in this text as \Cref{thm:BoroFriedl}.
However, the balanced setting is of special interest, because there seems to be no analog in the unsigned or signed setting of the
characterizations \eqref{item:Zslicesurface} and \eqref{item:Seifertsurface} of the $\mathbb{Z}$--slice genus
in terms of surfaces in 3-- and 4--space.

\subsectionpdfbookmark{Proof of \Cref{thm:main}}{Proof of Theorem~\ref{thm:main}}\label{subsec:proofofmain}
We provide the structure of the proof of \Cref{thm:main}. The details of the argument will be given in \Cref{sec:4D,sec:3D}.
\begin{proof}
[Proof of \Cref{thm:main}]
$\eqref{item:Zslicesurface} \Rightarrow \eqref{item:Blanchfield}$: This is the content of \Cref{sec:4D}.
This part of the proof is
the generalization of the if-part of \Cref{thm:Alex1<=>Zslice}
from $g = 0$ to arbitrary $g \geq 0$. However, while the case $g = 0$
is a quite straight-forward homology calculation (see \cite[last paragraph of Sec.~1.2]{Freedman_82_ASurgerySequenceInDimFour}),
the case $g\geq 0$ requires a little more work. We give a sketch of the argument.
Given a surface $F$ as in~$\eqref{item:Zslicesurface}$, the Blanchfield pairing appears as
sesquilinear intersection form of the universal covering of a slight modification of $B^4 \setminus F$.
The second homology of that covering is a free $\Z[t^{\pm}]$--module of rank $2g$, which gives a Hermitian presentation matrix
$A(t)$ of the desired size for the Blanchfield pairing. The signature
of $A(1)$ may be calculated by the Novikov-Wall non-additivity theorem.

$\eqref{item:Blanchfield} \Rightarrow \eqref{item:unknotting}$:
A result of Borodzik and Friedl implies $\eqref{item:Blanchfield}\Rightarrow \eqref{item:unknotting}$ under the added hypothesis
that $A(1)$ be congruent to a diagonal matrix \cite[Thm.~5.1]{BorodzikFriedl_14_OnTheAlgUnknottingNr}.
We show that this hypothesis is not necessary. The details are provided in \Cref{subsec:(4)=>(3)}.

$\eqref{item:unknotting} \Rightarrow \eqref{item:Seifertsurface}$: The Seifert surface may
be constructed from the crossing changes in an explicit and geometric way; see \Cref{subsec:(3)=>(2)}.

$\eqref{item:Seifertsurface}\Rightarrow\eqref{item:Zslicesurface}$: This is known to be a consequence of the only-if-part of \Cref{thm:Alex1<=>Zslice}. %
We recall the brief argument. The surface $F$ as in \eqref{item:Zslicesurface} is found by taking the union of the Seifert surface in $S^3$ with two boundary components as in~\eqref{item:Seifertsurface} with a locally flat disk $D$ with boundary the Alexander polynomial 1 component as described in \Cref{thm:Alex1<=>Zslice}. After pushing the interior of $F$ into the interior of $B^4$, it remains to check that $\pi_1(B^4\setminus F)\cong \Z$; see \cite[Proof of Claim~20]{FellerLewark_16} for details of how this can be done.

Alternatively, using the setup up from \Cref{subsec:galg}, we note that $\eqref{item:Seifertsurface}{\Leftrightarrow}\eqref{item:ga}$ by \cite[Proposition~17]{FellerLewark_16} and $\eqref{item:ga}{\Rightarrow}\eqref{item:Zslicesurface}$ by \cite[Theorem~1]{FellerLewark_16}; however, we prefer the above direct argument since it makes it clear that the Seifert matrix arguments from \cite{FellerLewark_16} are not needed.
\end{proof}

\subsection{Attempts at variations of the main theorem}
The reader might be tempted to vary the main theorem
by replacing the condition of `Alexander polynomial 1'.
Such attempts will probably succeed to produce variations (1'), (2'), (3') of conditions (1), (2), (3),
and prove the implications (3') $\Rightarrow$ (2') $\Rightarrow$ (1'),
but will probably fail to produce a fitting analog to condition (4), or prove the equivalence of (1'), (2'), (3').

For example, consider the balanced Gordian distance between $K$ and the set of smoothly slice knots;
this knot invariant was introduced by Livingston~\cite{MR1936979}, who denoted it by $U_s(K)$.
As in the proof of \Cref{thm:main}, one may prove that $U_s(K)$ is an upper bound for
the 3D--cobordism distance of $K$ to a smoothly slice knot, which is in turn an upper bound
for the smooth slice genus $\gs$ of $K$. Indeed, the resulting inequality $U_s(K) \geq \gs(K)$ was one
of Livingston's motivations to study $U_s(K)$.
However, as shown by Owens~\cite{MR2592947}, there exist knots $K$ with $U_s(K) \neq \gs(K)$.

We invite the reader to try further variations, e.g.~replacing the condition of `Alexander polynomial 1' by
`topologically slice', `smoothly $\Z$--slice' or `trivial'.
We have found no case where the equivalence of the resulting conditions (1'), (2'), (3') seems plausible for $g > 0$.

\subsection{A generalization of the main theorem to integral homology spheres and $\mathbb{Z}$--slice surfaces in other four-manifolds
}\label{subsec:genmainthm}

So far, we have dealt with knots in $S^3$. However, the only property of $S^3$ we ever use is its homology,
and so \cref{thm:main} and its proof generalize with minimal changes to all integer homology spheres $M$. By the work of Freedman
\cite{Freedman_82_TheTopOfFour-dimensionalManifolds,FreedmanQuinn_90_TopOf4Manifolds}, there is a unique contractible four-manifold~$B$ with boundary $M$, which is the natural ambient space in which to consider $\mathbb{Z}$--slice surfaces.
Perhaps more exciting than the transition from $S^3$ to an integral homology sphere $M$ is to consider $\mathbb{Z}$--slice surfaces in four-manifolds $V$ with boundary $M$ that are different from $B$. This was motivated by~\Cref{cor:snZ=gZ}, which one gets as an application. We note that it turns out that $\mathbb{Z}$--slice surfaces in the four-manifolds we consider are automatically null-homologous; see \Cref{lem:Zslicesurfacesarenullhomo}. This is why we do not add a null-homologous assumption below (in contrast to the definition of $\sn$ above). 

We describe a version of \Cref{thm:main} that incorporates all of this. 
Since this will be a bit lengthy to express, we invite the reader to first parse the statements that follow for the cases $M=S^3$, $h=0$, and either $c_1=0$ or $c_2=0$.

\begin{thm}\label{thm:mains2s2}
Let $K$ be a knot in an integral homology sphere $M$.
Let $h$ and~$c_2$ be non-negative integers and $c_1\in\{0,1\}$, such that $h+c_1\geq1$ if $\Arf(K) = 1$.
Set $g =h + c_1+c_2$. Then the following are equivalent:
\begin{enumerate}[itemsep=1ex]
  \item[(1')] There exists an oriented compact surface $F$ of genus $h$
properly embedded and locally flat in $V = B \# (\mathbb{C}P^2\#\overline{\mathbb{C}P^2})^{\#c_1}\# (S^2\times S^2)^{\#c_2}$
with boundary $K\subseteq M=\partial V$ such that $\pi_1(V\setminus F)\cong \Z$.
Here, $B$ denotes the unique contractible topological four-manifold with $\partial B = M$.
  \item[(2')] There exists a smooth oriented compact genus $g$ surface in $M$ with two boundary components, %
one of which is the knot $K$ and the other %
a knot with Alexander polynomial~$1$.
  \item[(3')] The knot $K$ can be turned into a knot with Alexander polynomial $1$ by changing $g$ positive and $g$ negative crossings.
  \item[(4')] The Blanchfield pairing of $K$ can be presented by a Hermitian matrix $A(t)$ of size
 $2g$ over $\Z[t^{\pm1}]$ such that the integral symmetric matrix $A(1)$ has signature zero.
\end{enumerate}

\end{thm}
Note that (2'), (3'), (4') of \Cref{thm:mains2s2} are identical to (2), (3), (4) of \cref{thm:main} (with $S^3$ replaced by $M$).
Also recall that $\mathbb{C}P^2\#S^2\times S^2=\mathbb{C}P^2\#\mathbb{C}P^2\#\overline{\mathbb{C}P^2}$,
which is why $c_1$ is restricted to $\{0,1\}$.
The proof of \Cref{thm:mains2s2} will be given in detail in \Cref{sec:s2s2}.
It follows the same outline as the proof of \Cref{thm:main} given in \Cref{subsec:proofofmain}.

By now, the eager reader has spotted that the special case $M = S^3$, $h = 0$, $c_1 = 0$ of \Cref{thm:mains2s2} yields \Cref{cor:snZ=gZ}, %
while the special case $M = S^3$, $h = 0$, $c_2 = 0$ takes the following form:
\begin{corollary}\label{cor:snZ=gZandCP2}
Let $K\subset S^3$ be a knot. Then the $\mathbb{Z}$--slice genus $\gst(K)$ of $K$ equals the smallest integer $c\geq 0$ such that a $\mathbb{Z}$--slice disk for $K$ exists in $B^4\# (\mathbb{C}P^2\#\overline{\mathbb{C}P^2})^{\# c}$.\qed
\end{corollary}

\subsection{Linking forms of cyclic branched coverings%
}\label{subsecintro:lkf}

In a second part of this paper (\Cref{sec:Comp}), we use our new characterization of $\gst$ given in \Cref{thm:main}\eqref{item:Blanchfield} to provide a criterion to obstruct knots from having $\gst\leq1$. We summarize what we obtain.

The Blanchfield specializes (essentially by setting $t=-1$) to the linking pairing $\lkh\colon  H_1(\Sigma_2(K); \Z)\times H_1(\Sigma_2(K); \Z)\to \Q/\Z$
on the first integral homology group of the double branched covering $\Sigma_2(K)$ of $S^3$ along $K$.
Using \Cref{thm:main}\eqref{item:Blanchfield}, we show that $\gst(K)\leq 1$ implies that $\lkh$ admits a $2\times 2$ presentation
matrix with determinant equal to $-1$ modulo $4$; see \Cref{prop:doublepresmat}.
Since $H_1(\Sigma_2(K);\Z)$ is of odd order, the following proposition about pairings on abelian groups of odd order provides a testable criterion whether $\lkh$ has such a $2\times2$ presentation matrix.

\begin{restatable}{prop}{propjacobi}
\label{prop:jacobi}
Let $A$ be an abelian group of odd order with at most two generators,
equipped with a non-degenerate symmetric pairing $\ell\colon A \times A \to \Q/\Z$.
Then $(A,\ell)$ decomposes as an orthogonal sum of two subgroups generated cyclically by $g_1, g_2$,
which are of respective order $q_1$ and $q_2$ with $q_1 | q_2$ with odd $q_1, q_2\geq 1$.
Let $\frac{a_i}{q_i} = \ell(g_i, g_i)$.
Then $a_i$ and $q_i$ are coprime, and for a given $u \in \{-1,1\}$, the following statements (A) and (B) are equivalent:
\begin{enumerate}[leftmargin=2em]
\item[(A)]\label{item:A} $\ell$ can be presented by an odd symmetric $2\times 2$ integer matrix~$M$ with $\det M \equiv u \pmod{4}$.
\item[(B)]\label{item:B} $a_1, a_2, q_1, q_2, u$ satisfy both of the following two conditions.\smallskip
\begin{enumerate}[leftmargin=2em]
\item[(B1)] $(-1)^{(q_1q_2 - u)/2}a_1a_2$ is a quadratic residue modulo $q_1$,
 \label{cond:2}
\item[(B2)] $u = 1$, or $q_1q_2 \equiv 3\pmod{4}$, or the Jacobi symbol $\legendre{a_2}{q_2/q_1}$ equals $1$.\label{cond:3}%
\end{enumerate}%
\end{enumerate}%
Note that $a_1, a_2 \in \Z$ are not uniquely determined. However, whether statement (B) holds does not depend on the choice of $a_1, a_2$.
\end{restatable}

Note that criterion (B) is easy to check for a given pairing.
While we find this to be of theoretical interest, \Cref{prop:jacobi} also allows to complete the calculation of the $\Z$--slice genus for all prime knots of crossing number up to~12; see \Cref{subsec:comp}.
\Cref{prop:jacobi} can also be applied to simplify \cite[Lemma~5.2]{BorodzikFriedl_15_TheUnknottingnumberAndClassInv1}, which can be used to show that a knot has algebraic unknotting number at least 3, and which was our inspiration to implement obstructions for $\Z$--slice genus (the balanced version of the algebraic unknotting number).
The proof of \Cref{prop:jacobi} uses Wall's classification \cite{MR0156890}, quadratic reciprocity and (for the proof of (B)$\Rightarrow$(A)) Dirichlet's prime number theorem, but is elementary apart from that.

In contrast to the effectiveness in obstructing $\Z$--slice genus using the linking pairing of the double-branched covering, we have \Cref{thm:allisometric}: for odd primes $p$, the linking pairing of the $p$--fold branched covering does not provide any additional information to the $\Z[\Z/p]$--module structure of its first homology.
\Cref{thm:allisometric} partially explains a disappointing finding of Borodzik and Friedl in \cite{BorodzikFriedl_15_TheUnknottingnumberAndClassInv1}: the implementation of their obstruction \cite[Lemma~5.1(2)]{BorodzikFriedl_15_TheUnknottingnumberAndClassInv1} for $p>2$ failed to give bounds for the algebraic unknotting number of small knots that were better than the bounds given by the Nakanishi index and Levine-Tristram signatures. %

We establish \Cref{thm:allisometric} by proving a more general statement about linking forms on modules over Dedekind domains with an involution; see \Cref{subsec:higher}.

\subsection{Structure of the paper}
\Cref{sec:prel} gives the essential definitions and fixes notation concerning the Alexander module and the Blanchfield pairing.
\Cref{sec:4D,sec:3D} contain the 4--dimensional and 3--dimensional part of the proof of \Cref{thm:main}, respectively.
In \Cref{sec:s2s2}, \Cref{thm:mains2s2} is proven.
\Cref{sec:Comp} is devoted to linking pairings of branched coverings and calculations of $\gst$.
It contains the proof of \cref{thm:allisometric}.
\Cref{app} provides background on Hermitian pairings.

\subsection*{Acknowledgments} The first author thanks Matthias Nagel, Patrick Orson, and Mark Powell for a fun Blanchfield pairing calculation session. %
Thanks to Filip Misev for drawing \cref{fig:333} together with the second author. %
Thanks to Maciej Borodzik, Anthony Conway, Jim Davis, Chuck Livingston, Duncan McCoy, and Matthias Nagel for comments on a first version of this paper.
The first author gratefully acknowledges support by the Swiss National Science Foundation Grant 181199.
The second author was supported by the Emmy Noether Programme of the DFG, project no.~412851057.
The authors thank the anonymous referee for their careful reading of the text.
\section{Preliminaries%
}\label{sec:prel}
In this section, we collect definitions of and known facts about links and Seifert surfaces, Hermitian pairings and their presentation matrices,
the Alexander module, and the Blanchfield pairing.

\subsection{Matrix presentations of pairings on torsion modules}\label{subsec:pairings}
Let $R$ denote a commutative unital ring with an involution, denoted by $r\mapsto\overline{r}$.
Some examples are the integers $\Z$ with the identity as involution, the ring $\Lambda\coloneqq \Z[t^{\pm1}]$ of Laurent polynomials with integer coefficients and involution given by $f(t)\mapsto f(t^{-1})$, and, for all integers $n\geq 2$, the rings $\Lambda/(t^n-1)$ and $\Lambda/(\Phi_n)$ (where $\Phi_n$ denotes the $n$--th cyclotomic polynomial) with the involution induced from $\Lambda$. We denote by $Q(R)$ the \emph{total quotient ring}---the localization $S^{-1}R$ of $R$ with respect to $S\coloneqq R\setminus \{\text{zero-divisors of } R\}$. Of course, for integral domains%
, $Q(R)$ is simply the field of fractions; e.g.~$\Q$ and $\Q(t)$ for $\Z$ and $\Lambda$, respectively.

An $R$--module $M$ is called \emph{torsion} if it is annihilated by some non-zero-divisor of~$R$.
A \emph{Hermitian pairing} on such an $M$ is an
$R$--sesquilinear map (anti-linear with respect to $\overline{\,\cdot\,}$ in the first factor) $\ell\colon M\times M\to Q(R)/R$ such that $\ell(y,x)=\overline{\ell(x,y)}$ for all $x$ and $y$ in $M$.
Such a pairing is called \emph{non-degenerate} if for all $x\in M$ there is a $y\in M$ with $\ell(x,y) \neq 0$.

A Hermitian square matrix $A\in \Mat_{n\times n}(R)$
whose determinant is not a zero divisor
defines a non-degenerate Hermitian pairing on the cokernel of $A$ as follows:
\[\ell_A\colon R^n/AR^n\times R^n/AR^n\to Q(R)/R,\quad (x,y)\mapsto \overline{x}^{\top}A^{-1}y + R.\]

A Hermitian square matrix $A\in \Mat_{n\times n}(R)$ with non-zero-divisor determinant is said to \emph{present} a Hermitian pairing $\ell\colon M\times M\to Q(R)/R$ if $\ell$ is isometric to $\ell_A$, i.e.~there exists an $R$--module isomorphism $\phi\colon M\to R^n/AR^n$ such that $\ell_A(\phi(x),\phi(y))=\ell(x,y)$ for all $x$ and $y$ in $M$.

See the appendix for a different perspective on Hermitian pairings, and a base change proposition.

\subsection{Links, Seifert surfaces, and 3D--cobordism} Fix a $3$--manifold $M$.
A \emph{link} is a smooth oriented non-empty closed $1$--submanifold of $M$; a connected link is a \emph{knot}. A \emph{Seifert surface} is a smooth oriented connected compact $2$--submanifold with non-empty boundary. A Seifert surface for a link $L$ is a Seifert surface with boundary $L$.
A \emph{3D--cobordism} between two knots $K$ and $J$ is a Seifert surface with boundary a $2$--component link, one component of which is isotopic to $K$, and the other to $J$ with reversed orientation. In previous work, the authors considered 3D--cobordisms between multi-component links \cite{FellerLewark_16}, but this will not be needed here.

\subsection{Twisted homology}\label{subs:twist}
Let $X$ be a space admitting a universal covering.
A surjective group homomorphism $\varphi\colon \pi_1(X)\to G$ to some group $G$ (we consider only abelian $G$)
yields a notion of twisted homology.
For this, take the $\ker(\varphi)$--covering $Y$ of~$X$ and use the deck transformation group action by $G$ to endow the singular chain complex $C_*(Y; \Z)$ with a $\Z[G]$--module structure, resulting in the twisted chain complex~$C_*(X; \Z[G])$.
Define the twisted homology $H_*(X;\Z[G])$ to be the homology of $C_*(X; \Z[G])$, and the twisted cohomology $H^*(X; \Z[G])$ to be homology of the cochain complex given by \[C^*(X; \Z[G]) = \Hom_{\Z[G]}\left(\overline{C_*(X; \Z[G])}, \Z[G]\right).\]
More details may be found e.g.~in~\cite{FNOP}.

\subsection{Alexander module}

Let $K$ be a knot in $S^3$.
The abelianization $\pi_1(S^3\setminus K) \to \Z$ induces a covering space,
the \emph{infinite cyclic covering}, which we denote by~${S^3\setminus K}^{\cyc}$. The first homology $H_*({S^3\setminus K}^{\cyc};\Z)$ becomes a $\Z[\Z]$-module using the deck transformation group action and can be canonically identified with the twisted homology  with respect to the abelianization $H_*({S^3\setminus K};\Lambda)$
(where we identify the group ring $\Z[\Z]$ with $\Lambda$).
The \emph{Alexander polynomial} $\Delta_K\in \Lambda$ of $K$ is usually defined (as we also did in the introduction)
as the order of $H_1({S^3\setminus K};\Lambda)$, which is well-defined up to multiplication with a unit in~$\Lambda$.
Here, the \emph{order ideal} of a finitely presented torsion-module is the ideal generated by the determinants of $n\times n$ minors of an $n\times m$ presentation matrix with $n\leq m$. Since the Alexander module can be presented by a square matrix, the order ideal is a principal ideal. The \emph{order} is a generator of the order ideal.

Alternative (and equivalent) definitions of the Alexander module and polynomial
use the zero-framed Dehn surgery $M_K$ of $K$ instead of the complement $S^3\setminus K$.
Recall that the map $\pi_1(S^3\setminus K)\to\pi_1(M_K)$ induced by inclusion is surjective,
and its kernel---normally generated by the class of a zero-framed longitude---is contained in the second derived commutator subgroup.
Since $H_1(\,\cdot\,; \Z)$ is canonically isomorphic to the abelianization of the commutator subgroup of $\pi_1(\,\cdot\,)$,
the inclusion also induces a $\Lambda$--module isomorphisms between $H_1(S^3\setminus K;\Lambda)$ and $H_1(M_K;\Lambda)$.
Thus, it is consistent with the usual definitions to see the Alexander module as $H_1(M_K;\Lambda)$, and the Alexander polynomial
as its order.

Also, we normalize the Alexander polynomial $\Delta_K$ to be symmetric and satisfy $\Delta_K(1)=1$.
\subsection{The Blanchfield pairing}\label{subsec:defBl}

The Alexander module being a torsion module, one can define the \emph{Blanchfield pairing} as the following non-degenerate Hermitian pairing on the Alexander module:
\[\Bl(K)\colon H_1(S^3\setminus K;\Lambda)\times H_1(S^3\setminus K;\Lambda)\to \Q(t)/\Lambda,\quad (x,y)\mapsto (\Psi(x))(y),\]
where $\Psi$ is the composition of the following maps
\begin{multline*}
H_1(S^3\setminus K;\Lambda)\to H_1(S^3\setminus K,\partial(S^3\setminus K);\Lambda)\xrightarrow{\cong}
H^2(S^3\setminus K;\Lambda)\\\xrightarrow{\cong}
H^1(S^3\setminus K;\Q(t)/\Lambda)
\xrightarrow{\mathrm{ev}}
\overline{\Hom_\Lambda(H_1(S^3\setminus K;\Lambda),\Q(t)/\Lambda)}.
\end{multline*}

The first map is given by canonical inclusion on the chain complex level, the second map is the inverse of Poincar\'e duality for the twisted homology of the 3--manifold $S^3\setminus K$ (see \cite[Section~2]{Wall_book} or \cite[Section~A.3]{FNOP}), the third map is the inverse of the \emph{Bockstein} map---the connecting homomorphism in the long exact sequence of cohomology induced by the short exact sequence of coefficients
\[0\to\Lambda\to\Q(t)\to\Q(t)/\Lambda\to0,\]
and the fourth map is the so-called Kronecker evaluation map.
We only give this brief treatment since we will not make use of the definition of the Blanchfield pairing, and refer the reader to~\cite{FriedlPowell_15} for a detailed treatment.

\subsection{The Blanchfield pairing via twisted intersection forms on $4$--manifolds}
Above we recalled the definition of the Blanchfield pairing using (twisted) Poincar{\'e} duality of 3--manifolds. Much like linking numbers in 3--manifolds can be calculated by intersecting surfaces in 4--manifolds with boundary the 3--manifold in question (most classically, the linking number of two disjoint oriented curves in $S^3$ equals the oriented intersection of generic surfaces bounding them in $B^4$), the Blanchfield pairing has a presentation via the twisted homology of 4--manifolds $W$ with boundary~$M_K$. Borodzik and Friedl established the following rather general statement, which only asks for natural homological assumptions on $W$.
We also note that the result holds in the topological category.
\begin{thm}[{\cite[Theorem~2.6]{BorodzikFriedl_15_TheUnknottingnumberAndClassInv1}}]\label{thm:BF:W}
Let $K$ be a knot and $W$ a connected compact oriented topological 4--manifold with infinite cyclic fundamental group
and boundary $M_K$ such that
the inclusion of $M_K$ into $W$ descends to an isomorphism on $H_1(\,\cdot\,;\Z)$.
Then the twisted homology $H_2(W; \Z[\pi_1(W)])$ is free of rank $b_2(W)$.
Furthermore, if $B$ is an integral matrix for the ordinary intersection pairing of $W$, then
there exists a basis $\mathfrak{B}$ for $H_2(W; \Z[\pi_1(W)])$ such that the matrix $A(t)$ of the twisted
intersection pairing with respect to $\mathfrak{B}$ presents the Blanchfield pairing $\Bl(K)$ of $K$, and $A(1) = B$.
\end{thm}
Although we only apply \Cref{thm:BF:W} to a rather special manifold $W$ in the next section,
we do not know of a faster proof that the above holds for this manifold $W$ than the one by Borodzik and Friedl,
which goes through a quite general argument employing the universal coefficient spectral sequence.

\sectionpdfbookmark{The four-dimensional part of the proof---\eqref{item:Zslicesurface}~$\Rightarrow$~\eqref{item:Blanchfield}}{The four-dimensional part of the proof---\eqref{item:Zslicesurface}~=>~\eqref{item:Blanchfield}}\label{sec:4D}
\newcommand{\tub}[1]{\ensuremath{\nu\hspace{-0.1em}#1}}
Let $F$ be a $\Z$--slice surface of genus $g$ in $B^4$ with boundary $K$. We calculate the Blanchfield pairing of $K$ by using $F$ to define a $4$--manifold $W$ with boundary the zero-framed Dehn surgery of $K$, denoted by $M_K$, such that the intersection pairing twisted by $\pi_1(W)$ on $W$ is a Hermitian presentation of the Blanchfield pairing.
For this purpose, we construct $W$ such that $\pi_1(W) \cong\Z, b_2(W) = 2g, \sigma(W) = 0$ and the inclusion
of $M_K$ into $W$ descends to an isomorphism on integral first homology groups.

Given such a $4$--manifold $W$, \Cref{thm:BF:W} yields a Hermitian $2g\times2g$--matrix $A(t)$ over the ring $\Lambda$ that presents the Blanchfield pairing of $K$ such that $A(1)$ is a unimodular matrix with signature 0. Unimodularity follows since $A(1)$ presents the ordinary intersection form on $W$, which is unimodular because the fact that the inclusion of $M_K=\partial W$ into $W$ induces an isomorphism on $H_1(\,\cdot\,;\Z)$ implies that the long exact sequence of the pair $(W,\partial W)$ induces an isomorphism on $H^2(W,\partial W;\Z)\to H^2(W;\Z)$. And, by definition, the signature of $W$ is the signature the ordinary intersection form on $W$, thus $\sigma(A(1))=\sigma(W)=0$. %

Therefore, to prove \eqref{item:Zslicesurface}~$\Rightarrow$~\eqref{item:Blanchfield},
it only remains to actually construct the 4--manifold~$W$ with the desired properties.
Our construction is modeled on what one often sees in the literature when $F$ is a pushed-in Seifert surface; see e.g.~\cite[Proof of Lemma~5.4]{COT2}.
See also \cite{MR3604490}, where  this construction is considered for strong slice surfaces of links.
We build $W$ in two steps.

\subsection{Step I} We set $W'\coloneqq B^4\setminus\tub{F}$, where $\tub{F}$ denotes an open tubular neighborhood.
Concretely, $\tub{F}$ may be taken as open disk subbundle of the normal bundle of $F$ in the sense of~\cite[Section~9.3]{FreedmanQuinn_90_TopOf4Manifolds}. In particular, the boundary  $\partial \tub{F}$ of $\tub{F}$ as a subspace of $B^4$ is a locally flat 3--manifold with boundary,
properly embedded in $B^4$ and homeomorphic to $F\times S^1$. We note that $\partial W'=\partial \tub{F} \cup (S^3\setminus\tub{K})$, where $\tub{K}$ denotes an open tubular neighborhood of $K$ in $S^3$.  The two pieces, $\partial \tub{F}$ and $S^3\setminus\tub{K}$, intersect in a torus, which we denote by $\Sigma$. It can be understood as the unit normal bundle of $K$ in $S^3$.

\begin{claim}\label{claim1}
We have (i) $\pi_1(W')\cong \Z$, (ii) $b_2(W')=2g$, and (iii) $\sigma(W')=0$.
\end{claim}
Before proving the claim, let us consider the following special case of Novikov-Wall non-additivity,
which will be needed for the proof of (iii).
\begin{lemma}\label{lemma:NW-add}
Let $Z$ be a closed surface, let $X_{\pm}, X_0$ be 3--manifolds with boundary~$Z$,
let $Y_{\pm}$ be topological 4--manifolds with boundaries $\partial Y_{\pm} = X_{\pm} \cup_Z X_0$,
and let $Y$ be the topological 4--manifold given as $Y_+ \cup_{X_0} Y_-$
(see \Cref{fig:NW-add}(i)).
Consider the three maps on $H_1(Z;\Q)$ induced by the inclusions of $Z$ into $X_{\pm}$ and $X_0$.
If the kernels of two of these maps agree, then
$\sigma(Y) = \sigma(Y_+) + \sigma(Y_-)$.
\end{lemma}
\begin{figure}[tbh]%
\begin{tabular}{@{}c*2{@{\,}c}}%
\fbox{
\begingroup%
  \makeatletter%
  \providecommand\color[2][]{%
    \errmessage{(Inkscape) Color is used for the text in Inkscape, but the package 'color.sty' is not loaded}%
    \renewcommand\color[2][]{}%
  }%
  \providecommand\transparent[1]{%
    \errmessage{(Inkscape) Transparency is used (non-zero) for the text in Inkscape, but the package 'transparent.sty' is not loaded}%
    \renewcommand\transparent[1]{}%
  }%
  \providecommand\rotatebox[2]{#2}%
  \newcommand*\fsize{\dimexpr\f@size pt\relax}%
  \newcommand*\lineheight[1]{\fontsize{\fsize}{#1\fsize}\selectfont}%
  \ifx\svgwidth\undefined%
    \setlength{\unitlength}{92.28337428bp}%
    \ifx\svgscale\undefined%
      \relax%
    \else%
      \setlength{\unitlength}{\unitlength * \real{\svgscale}}%
    \fi%
  \else%
    \setlength{\unitlength}{\svgwidth}%
  \fi%
  \global\let\svgwidth\undefined%
  \global\let\svgscale\undefined%
  \makeatother%
  \begin{picture}(1,0.85949316)%
    \lineheight{1}%
    \setlength\tabcolsep{0pt}%
    \put(0,0){\includegraphics[width=\unitlength,page=1]{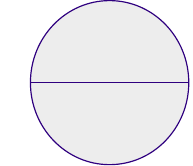}}%
    \put(0.57006456,0.17983557){\color[rgb]{0,0,0}\makebox(0,0)[t]{\lineheight{1.25}\smash{\begin{tabular}[t]{c}$Y_-$\end{tabular}}}}%
    \put(0.57006456,0.61463922){\color[rgb]{0,0,0}\makebox(0,0)[t]{\lineheight{1.25}\smash{\begin{tabular}[t]{c}$Y_+$\end{tabular}}}}%
    \put(0.407522,0.46469344){\color[rgb]{0.18431373,0,0.49803922}\makebox(0,0)[t]{\lineheight{1.25}\smash{\begin{tabular}[t]{c}$X_0$\end{tabular}}}}%
    \put(0.23026097,0.7284195){\color[rgb]{0.18431373,0,0.49803922}\makebox(0,0)[rt]{\lineheight{1.25}\smash{\begin{tabular}[t]{r}$X_+$\end{tabular}}}}%
    \put(0.22399209,0.07705626){\color[rgb]{0.18431373,0,0.49803922}\makebox(0,0)[rt]{\lineheight{1.25}\smash{\begin{tabular}[t]{r}$X_-$\end{tabular}}}}%
    \put(0,0){\includegraphics[width=\unitlength,page=2]{walls_3.pdf}}%
    \put(0.07667567,0.40069103){\color[rgb]{0.64313725,0.22745098,0.07843137}\makebox(0,0)[rt]{\lineheight{1.25}\smash{\begin{tabular}[t]{r}$Z$\end{tabular}}}}%
  \end{picture}%
\endgroup%
}
&
\fbox{\hspace{-1.2em}
\begingroup%
  \makeatletter%
  \providecommand\color[2][]{%
    \errmessage{(Inkscape) Color is used for the text in Inkscape, but the package 'color.sty' is not loaded}%
    \renewcommand\color[2][]{}%
  }%
  \providecommand\transparent[1]{%
    \errmessage{(Inkscape) Transparency is used (non-zero) for the text in Inkscape, but the package 'transparent.sty' is not loaded}%
    \renewcommand\transparent[1]{}%
  }%
  \providecommand\rotatebox[2]{#2}%
  \newcommand*\fsize{\dimexpr\f@size pt\relax}%
  \newcommand*\lineheight[1]{\fontsize{\fsize}{#1\fsize}\selectfont}%
  \ifx\svgwidth\undefined%
    \setlength{\unitlength}{123.36641955bp}%
    \ifx\svgscale\undefined%
      \relax%
    \else%
      \setlength{\unitlength}{\unitlength * \real{\svgscale}}%
    \fi%
  \else%
    \setlength{\unitlength}{\svgwidth}%
  \fi%
  \global\let\svgwidth\undefined%
  \global\let\svgscale\undefined%
  \makeatother%
  \begin{picture}(1,0.64293776)%
    \lineheight{1}%
    \setlength\tabcolsep{0pt}%
    \put(0,0){\includegraphics[width=\unitlength,page=1]{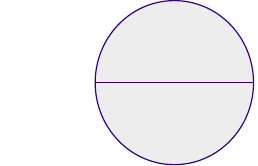}}%
    \put(0.6783894,0.13452472){\color[rgb]{0,0,0}\makebox(0,0)[t]{\lineheight{1.25}\smash{\begin{tabular}[t]{c}$\overline{\tub{F}}$\end{tabular}}}}%
    \put(0.6783894,0.45977651){\color[rgb]{0,0,0}\makebox(0,0)[t]{\lineheight{1.25}\smash{\begin{tabular}[t]{c}$W'$\end{tabular}}}}%
    \put(0.55680007,0.34761063){\color[rgb]{0.18431373,0,0.49803922}\makebox(0,0)[t]{\lineheight{1.25}\smash{\begin{tabular}[t]{c}$\partial\tub{F}$\end{tabular}}}}%
    \put(0.42420051,0.54488904){\color[rgb]{0.18431373,0,0.49803922}\makebox(0,0)[rt]{\lineheight{1.25}\smash{\begin{tabular}[t]{r}$S^3\setminus\tub{K}$\end{tabular}}}}%
    \put(0.40431333,0.05764139){\color[rgb]{0.18431373,0,0.49803922}\makebox(0,0)[rt]{\lineheight{1.25}\smash{\begin{tabular}[t]{r}$\overline{\tub{K}}$\end{tabular}}}}%
    \put(0,0){\includegraphics[width=\unitlength,page=2]{walls.pdf}}%
    \put(0.30933445,0.29973408){\color[rgb]{0.64313725,0.22745098,0.07843137}\makebox(0,0)[rt]{\lineheight{1.25}\smash{\begin{tabular}[t]{r}$\Sigma$\end{tabular}}}}%
  \end{picture}%
\endgroup%
}
&
\fbox{\hspace{0.5em}
\begingroup%
  \makeatletter%
  \providecommand\color[2][]{%
    \errmessage{(Inkscape) Color is used for the text in Inkscape, but the package 'color.sty' is not loaded}%
    \renewcommand\color[2][]{}%
  }%
  \providecommand\transparent[1]{%
    \errmessage{(Inkscape) Transparency is used (non-zero) for the text in Inkscape, but the package 'transparent.sty' is not loaded}%
    \renewcommand\transparent[1]{}%
  }%
  \providecommand\rotatebox[2]{#2}%
  \newcommand*\fsize{\dimexpr\f@size pt\relax}%
  \newcommand*\lineheight[1]{\fontsize{\fsize}{#1\fsize}\selectfont}%
  \ifx\svgwidth\undefined%
    \setlength{\unitlength}{127.59762464bp}%
    \ifx\svgscale\undefined%
      \relax%
    \else%
      \setlength{\unitlength}{\unitlength * \real{\svgscale}}%
    \fi%
  \else%
    \setlength{\unitlength}{\svgwidth}%
  \fi%
  \global\let\svgwidth\undefined%
  \global\let\svgscale\undefined%
  \makeatother%
  \begin{picture}(1,0.6216176)%
    \lineheight{1}%
    \setlength\tabcolsep{0pt}%
    \put(0,0){\includegraphics[width=\unitlength,page=1]{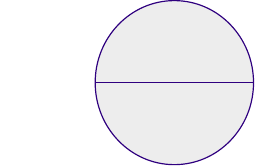}}%
    \put(0.65589263,0.13006363){\color[rgb]{0,0,0}\makebox(0,0)[t]{\lineheight{1.25}\smash{\begin{tabular}[t]{c}$H\times S^1$\end{tabular}}}}%
    \put(0.65589263,0.44453006){\color[rgb]{0,0,0}\makebox(0,0)[t]{\lineheight{1.25}\smash{\begin{tabular}[t]{c}$W'$\end{tabular}}}}%
    \put(0.65883157,0.33608366){\color[rgb]{0.18431373,0,0.49803922}\makebox(0,0)[t]{\lineheight{1.25}\smash{\begin{tabular}[t]{c}$\partial\tub{F}\cong F'\times S^1$\end{tabular}}}}%
    \put(0.4101338,0.52682022){\color[rgb]{0.18431373,0,0.49803922}\makebox(0,0)[rt]{\lineheight{1.25}\smash{\begin{tabular}[t]{r}$S^3\setminus\tub{K}$\end{tabular}}}}%
    \put(0.40559991,0.05572996){\color[rgb]{0.18431373,0,0.49803922}\makebox(0,0)[rt]{\lineheight{1.25}\smash{\begin{tabular}[t]{r}$\overline{H\setminus F'}\times S^1$\end{tabular}}}}%
    \put(0,0){\includegraphics[width=\unitlength,page=2]{walls2.pdf}}%
    \put(0.29905517,0.28979473){\color[rgb]{0.64313725,0.22745098,0.07843137}\makebox(0,0)[rt]{\lineheight{1.25}\smash{\begin{tabular}[t]{r}$\Sigma$\end{tabular}}}}%
  \end{picture}%
\endgroup%
\hspace{-.3em}}\\[1ex]
(i) $Y$ & (ii) $B^4$ & (iii) $W$
\end{tabular}
\caption{Three 4--manifolds, each as union of two others.}
\label{fig:NW-add}
\end{figure}
\begin{proof}
Novikov-Wall non-additivity~\cite{Wall} holds in the topological category~\cite{kirbybook}.
It states that $\sigma(Y) = \sigma(Y_+) + \sigma(Y_-) - \sigma(N)$,
where $N$ is a certain space with a bilinear form.
To prove the lemma, it suffices to prove that $\dim N = 0$, which we will do; in particular, we will not need to explain the definition of the bilinear form on $N$.

To see $\dim N = 0$, we recall from~\cite{Wall} that, if $A, B, C$ are the three kernels
mentioned in the statement of the lemma, then $N$ can be defined as
\[
N = \frac{A \cap (B + C)}{(A \cap B) + (A\cap C)} \subseteq H_1(Z; \Q).
\]
(We remark that the dimension of $N$ and the signature of the bilinear form do not change when permuting $A, B, C$~\cite{Wall}.)
If there are two among $A, B, C$ that agree, then clearly $\dim N = 0$.
\end{proof}
\begin{proof}[Proof of \Cref{claim1}]
(i) is satisfied by the definition of $\Z$--slice surface.

(ii) follows quickly from the Mayer-Vietoris sequence of $B^4 = W' \cup_{\partial\tub{F}} \overline{\tub{F}}$
(see \Cref{fig:NW-add}(ii)), which implies that $H_2(\partial\tub{F};\Q) \cong H_2(W';\Q) \oplus H_2(\overline{\tub{F}};\Q)$.
Alternatively, one could use an appropriate variant of Alexander duality in the ball.

For (iii), we wish to apply the \Cref{lemma:NW-add} to $B^4=W'\cup \overline{\tub{F}}$ (see \Cref{fig:NW-add}(ii)).
The hypothesis is satisfied
because the two maps induced by the inclusions of $\Sigma$ into
$S^3\setminus\tub{K}$ and $\partial\tub{F}$, respectively, have the same kernel, which is generated
by the class of a zero-framed longitude. Thus,
\[
\sigma(B^4)=\sigma(W')+\sigma(\overline{\tub{F}}).
\]
Clearly, we have $\sigma(B^4)=0$ (since $B^4$ has no $H_2$) and $\sigma(\overline{\tub{F}})=0$ (since $\overline{\tub{F}}$ is homeomorphic the product of $F$ with a closed 2--disk, so it has no $H_2$ either). It follows that $\sigma(W') = 0$.
\end{proof}
\subsection{Step II}
The $4$--manifold $W'$ satisfies the desired properties, with the exception that $\partial W'$ is not $M_K$.
In this step, we fix that.

Let $H$ be a genus $g$ handlebody.
In the boundary of $H$, a closed surface of genus $g$, we pick a subsurface $F'$ of genus $g$ with one boundary component (i.e.~the complement of a small open neighborhood of a point in~$\partial H$).
Let us glue $H\times S^1$ to $W'$ to obtain $W$ via a certain homeomorphism
\[
\begin{tikzcd}
\phi\colon\hspace*{-4.5em} & F'\times S^1\ar{r} \ar[draw=none]{d}[sloped,auto=false]{\subseteq} & \partial \tub{F} \ar[draw=none]{d}[sloped,auto=false]{\subseteq}. \\[-3ex]
& H\times S^1 & W'
\end{tikzcd}
\]
It turns out that the choice of homeomorphism matters.
Before we explain our choice, let us discuss properties that hold for any such homeomorphism~$\phi$.

The homeomorphism $\phi$ restricts on the boundary to a homeomorphism $(\partial F') \times S^1 \to \Sigma$ of tori.
Let us write $\ell, m$ respectively for a zero-framed longitude and a meridian of $K$ on $\Sigma$.
The curve $\phi(\partial F' \times \{1\})$ is homotopic to $\ell$,
since this curve and $K$ bound the disjoint surfaces $\phi(F'\times\{1\})$ and $F$ in~$B^4$.
Next, fix a base point $f_0 \in \partial F'$ and consider the curve $\phi(\{f_0\}\times S^1)$.
Its homology class and the class of $\ell$ form a basis of $H_1(\Sigma;\Z)$.
Therefore, its homology class is equal to $\pm [m] + \lambda [\ell]$ for some $\lambda\in\Z$.
Since $H_1(W';\Z)$ is generated by $[m]$, and $\ell$ is null-homologous in $W'$, it follows that
the homology class of $\phi(\{f_0\}\times S^1)$ generates $H_1(W';\Z)\cong\pi_1(W')$.

This can be rephrased as follows.
Let us denote by $i_1$ and $i_2$ and $j$ the homomorphisms of fundamental groups induced by the inclusions
$F' \to F'\times S^1, f\mapsto (f,1)$ and $S^1\to F'\times S^1, z\mapsto (f_0, z)$ and $\partial\tub{F} \hookrightarrow W'$.
Then the composition $j \circ \phi_* \circ i_2$
is an isomorphism $\pi_1(S^1) \to \pi_1(W')$, which we denote by $k$.

Let us now discuss our choice of $\phi$. We wish to choose $\phi$ such that the composition of the following two maps
is the zero map:
\[
\begin{tikzcd}
\pi_1(F') \ar{r}{\phi_*\circ i_1}  &
\pi_1(\partial\tub{F}) \ar{r}{j} &
\pi_1(W') %
\end{tikzcd}
\]
To construct $\phi$, start by picking any homeomorphism $\phi'\colon F'\times S^1\to \partial \tub{F}$.
Let $g\colon \pi_1(F')\to \pi_1(S^1)$ denote the composition $k^{-1}\circ j\circ \phi'_*\circ i_1$.
Let $\psi\colon F'\to S^1$  be a continuous map such that $\psi_*$ is $-g$ (we write the group action additively in $\pi_1(S^1)$,
and multiplicatively in $S^1$ and $\pi_1(F'\times S^1)$).
For this, recall that all group homomorphisms from $\pi_1(F')$ to $\Z$ are induced by a continuous map as a consequence of the following chain of canonical identifications
\[\mathrm{Hom}(\pi_1(F'),\Z)\cong \mathrm{Hom}(H_1(F';\Z),\Z)\cong H^1(F';\Z)\cong [F',S^1].\]
We define a homeomorphism $\omega\colon F'\times S^1\to F'\times S^1$ by $(f,z)\mapsto (f,z\cdot \psi(f))$ and set $\phi\coloneqq\phi'\circ \omega$.

Let us check that the composition $j\circ(\phi_*\circ i_1)$ is indeed zero for this choice of $\phi$.
It is sufficient that $k^{-1}\circ j\circ\phi_*\circ i_1 = 0$, which one finds as follows:
\begin{align*}
  k^{-1}\circ j\circ\phi_*\circ i_1 & = k^{-1}\circ j\circ\phi'_* \circ \omega_*\circ i_1 \\
                                    & = k^{-1}\circ j\circ\phi'_* \circ (i_1 \cdot (i_2\circ \psi_*)) \\
                                    & = (k^{-1}\circ j\circ\phi'_* \circ i_1) + (k^{-1}\circ j\circ\phi'_* \circ i_2\circ \psi_*) \\
                                    & = g + (k^{-1} \circ k \circ (-g)) = 0.%
\end{align*}%
\begin{Rmk}%
The careful reader will notice
that we in fact will only use that the kernel of $j \circ\phi_*$
contains the kernel of
the map induced by
the inclusion of $F'\times S^1$ into $H\times S^1$.
If $g > 0$, not all homeomorphisms $\phi\colon  F'\times S^1 \to \partial\tub{F}$ satisfy this,
and choosing one which does not would make the fundamental group of $W$ a finite instead of infinite cyclic group.
\end{Rmk}

We define $W\coloneqq W'\cup_{\phi}H\times S^1$ and note that $\partial W=\partial M_K$ (see \cref{fig:NW-add}(iii)).
We will establish the following properties $W$, thereby concluding the proof of~\eqref{item:Zslicesurface}~$\Rightarrow$~\eqref{item:Blanchfield}.

\begin{claim}\label{claim:PropOfW}
We have that (i) $\pi_1(W)\cong \Z$, (ii) $b_2(W)=2g$, (iii) $\sigma(W)=0$ and (iv) the inclusion of $M_K$ into $W$ descends to an isomorphism on~$H_1(\,\cdot\,;\Z)$.
\end{claim}
\begin{proof}[Proof of \cref{claim:PropOfW}]
(i) A simple application of the Seifert-van Kampen theorem implies that the inclusion of $W'$ into $W$ induces an isomorphism on fundamental groups.
For this, recall that $W$ is obtained by gluing $W'$ and $H\times S^1$ along $F'\times S^1$ via $\phi$. Consider the following commutative diagram of groups:
\[\begin{tikzcd}
    \pi_1(F'\times S^1) \arrow{r}{\subseteq_*} \arrow[swap]{d}{j\circ \phi_*} & \pi_1(H\times S^1) \arrow{d}{\subseteq_*} \\
    \pi_1(W')\arrow{r}{\subseteq_*} & \pi_1(W)
  \end{tikzcd}\]
By our assumption on $\phi$, $j \circ \phi_*$ factors through $\pi_1(F'\times S^1) \xrightarrow{\subseteq_*} \pi_1(H\times S^1)$
(that is, the map $\pi_1(H\times S^1) \to \pi_1(W')$ given as composition of $k$ with the canonical projection $\pi_1(H\times S^1) \to \pi_1(S^1)$
commutes with the other maps).
This implies that the bottom arrow is a group isomorphism since the diagram is a push-out diagram by the Seifert-van Kampen theorem.

(iv)
This follows because
the inclusions $S^3\setminus\tub{K}$ into $W'$ and $M_K$ and the inclusion of $W'$ into $W$ descend to isomorphism on $H_1(\,\cdot\,;\Z)$
(all of these first homology groups are generated by the class of a meridian of $K$ on $\Sigma$).

(ii) To show that $b_2(W)=2g$, we consider the long exact Mayer-Vietoris sequence for homology with rational coefficients:
\[
\begin{tikzcd}[column sep=1.3em]
H_3(W) &[-2em] \ar{r}{\partial} & H_2(F'\times S^1)\ar{r} & H_2(H\times S^1)\oplus H_2(W') \ar{r}    & H_2(W)           & \\[-4ex]
         & \ar{r}{\partial} & H_1(F'\times S^1)\ar{r} & H_1(H\times S^1)\oplus H_1(W') \ar{r}    & H_1(W) \ar{r}{0} & \ldots
\end{tikzcd}
\]
We note that $H_3(W;\Q)=0$. Indeed, $H_3(W;\Q)\cong H^1(W,M_K;\Q)\cong H_1(W,M_K;\Q)$ by Poincar\'{e} duality and the universal coefficient theorem, and $H_1(W,M_K;\Q)=0$ by the long exact sequence for the pair $(W,M_K)$ and the fact that $H_1(M_K;\Q)\to H_1(W;\Q)$ is an isomorphism (see (iv)).
Therefore, the dimensions of the homology spaces in the Mayer-Vietoris sequence are (from left to right)
$0,2g,3g,b_2(W),2g+1,g+2,1$. Since the alternating sum of dimensions is zero, we have $b_2(W)=2g$.

(iii)
We prove $\sigma(W)=0$ by applying \Cref{lemma:NW-add} to $W=W'\cup H\times S^1$;
compare \Cref{fig:NW-add}(iii).
The hypothesis of the lemma is satisfied since
the inclusion of $\Sigma$ in $S^3\setminus\tub{K}$ and $\partial\tub{F}$ have the same kernel (compare the proof of \Cref{claim1}(iii)).
We have $\sigma(W')=0$ (see \Cref{claim1}(iii)) and $\sigma(H\times S^1)=0$ (since, for example, any pair of classes can be represented by two disjoint closed surfaces). It follows that $\sigma(W) = 0$.
\end{proof}

\begin{rmk}\label{rmk:ko}
If the $\mathbb{Z}$--slice surface $F$ of $K$ is obtained as a pushed-in 3D--cobordism
between $K$ and a knot with Alexander polynomial $1$, then a presentation matrix of the Blanchfield pairing
can be given more explicitly using Ko's formula \cite{ko} (see also \cite[(2.3) and (2.4)]{BorodzikFriedl_15_TheUnknottingnumberAndClassInv1}, \cite[Sec.~5]{FellerLewark_16}).
This yields a direct proof of \eqref{item:Seifertsurface}~$\Rightarrow$~\eqref{item:Blanchfield},
the details of which we omit.
\end{rmk}

\section{The three-dimensional part of the proof}\label{sec:3D}
This section completes the proof of \Cref{thm:main}.
We will show~\eqref{item:Blanchfield}~$\Rightarrow$~\eqref{item:unknotting} and
\eqref{item:unknotting}~$\Rightarrow$~\eqref{item:Seifertsurface}
in the next two subsections, respectively.%

\subsectionpdfbookmark{Unknotting information from the Blanchfield pairing---\eqref{item:Blanchfield}~$\Rightarrow$~\eqref{item:unknotting}}{Unknotting information from the Blanchfield pairing---\eqref{item:Blanchfield}~=>~\eqref{item:unknotting}}\label{subsec:(4)=>(3)}
The main result of~\cite{BorodzikFriedl_14_OnTheAlgUnknottingNr} can be phrased as follows:
\begin{thm}[{\cite[Thm.~5.1]{BorodzikFriedl_14_OnTheAlgUnknottingNr}}]\label{thm:BoroFriedl}
Let $A(t)$ be a Hermitian presentation matrix of the Blanchfield pairing of a knot $K$.
Assume that the symmetric bilinear form $A(1)$ is diagonalizable,
and denote the number of its positive and negative eigenvalues
counted with multiplicity by $p, n \in \Z^+_0$, respectively.
Then $K$ can be turned into a knot with Alexander polynomial $1$
by changing $p$ positive and $n$ negative crossings.
\qed\end{thm}

We show that in case $A(1)$ is indefinite, the diagonalization assumption is unnecessary.
More precisely, we prove the following proposition, which might be of independent interest.

\begin{prop}\label{prop:Bl->unknotting}
Let $A(t)$ be a Hermitian presentation matrix of the Blanchfield pairing of a knot $K$.
Assume that the symmetric bilinear form $A(1)$ is indefinite,
and denote the number of its positive and negative eigenvalues
counted with multiplicity by $p, n \in \N$, respectively.
Then $K$ can be turned into a knot with Alexander polynomial~$1$
by changing $p$ positive and $n$ negative crossings.
\end{prop}

The case $p = n$ of this proposition is exactly the desired implication \eqref{item:Blanchfield}~$\Rightarrow$~\eqref{item:unknotting}.
The remainder of this section is devoted to the proof of the proposition.
Rather than unpacking Borodzik and Friedl's intricate 3--dimensional argument turning algebraic information into unknotting information, we use a purely algebraic argument about Hermitian pairings over $\Lambda\coloneqq\Z[t,t^{-1}]$ and $\Lambda_0\coloneqq\Z[t,t^{-1},(t-1)^{-1}]$ to reduce~\Cref{prop:Bl->unknotting} to~\Cref{thm:BoroFriedl}.
As a first step, we recall that multiplication by $(t-1)$ is an isomorphism of the Alexander module of a knot \cite{MR0461518}.
This implies that the Blanchfield pairing can be dealt with over $\Lambda_0$ rather than~$\Lambda$.
More precisely, we have the following.
\begin{lemma}\label{lem:lambda0}
Let $A(t)$ be a Hermitian $\Lambda$--matrix presenting the Blanchfield pairing of $K$.
If $T(t)$ is a $\Lambda_0$--matrix such that $\det T(t)$ is a unit in~$\Lambda$,
and
\[
B(t) = \overline{T(t)}^{\top}\! A(t) T(t)
\]
is a $\Lambda$--matrix,
then $B(t)$ also presents the Blanchfield pairing of $K$.\qed
\end{lemma}
This statement is implicit in~\cite[Proof of Prop.~2.1]{BorodzikFriedl_15_TheUnknottingnumberAndClassInv1} and~\cite[Proof of Lemma~5.4]{COT2}, but we thought it beneficial to make the statement explicit.
We use the occasion to formulate a general principle for arbitrary rings, which we prove in detail; see appendix.
The above lemma follows as a special case of \Cref{cor:BC}. %
Now, \Cref{prop:Bl->unknotting} is a direct consequence of \Cref{thm:BoroFriedl} and \cref{lem:nodiag} below.
The proof of the latter will rely on the following well-known fact.
\begin{thm}[see e.g.~\cite{huse}] \label{thm:classificationindefinite}
Let $M$ be a finitely generated free abelian group and $\theta\colon M\times M \to \Z$
an indefinite unimodular symmetric bilinear form. If $\theta$ is \emph{odd}, i.e.~there exists $v\in M$
such that $\theta(v,v)$ is odd, then $\theta$ admits a diagonal matrix; in other words,
$\theta$ is isometric to an orthogonal sum of copies of the forms $(1)$ and $(-1)$.
If $\theta$ is \emph{even}, i.e.~$\theta(v,v)$ is even for all $v\in M$, then $\theta$ is isometric
to an orthogonal sum of copies of the so-called hyperbolic plane $H = \small \begin{pmatrix}0&1\\1&0\end{pmatrix}$ and copies of the positive definite form~$E_8$, with at least one copy of $H$.
\end{thm}
\begin{lemma}\label{lem:nodiag}
Let $A(t)$ be a Hermitian matrix over $\Lambda$ such that $A(1)$ is unimodular and indefinite.
Then there is a transformation matrix over $\Lambda_0$ with determinant a unit in $\Lambda$,
transforming $A(t)$ into a Hermitian matrix $B(t)$ over $\Lambda$
such that $B(1)$ is diagonal with $\pm 1$ diagonal entries and $\sigma(B(1)) = \sigma(A(1))$.
\end{lemma}
\begin{proof}
If the form $A(1)$ is odd, then by \cref{thm:classificationindefinite}
there is a base change over $\Z$ that diagonalizes $A(1)$.
The same base change transforms $A(t)$ into a matrix~$B(t)$ that satisfies the desired properties.
So let us consider the case that $A(1)$ is even.
Then there is a base change over $\Z$ transforming $A(1)$ into $H \oplus R$ for some $R$.
Assume w.l.o.g.\ that $A(1)$ is already of this form.
Then all the entries of the first two rows of $A(1) - (H \oplus N)$ (where $N$ is a zero matrix of size $2g-2$) evaluate to $0$
at $t = 1$ and are thus divisible by $(t - 1)$. Because $A(t)$ is Hermitian, we see that its top-left $2\times 2$ submatrix is
of the form
\[
\begin{pmatrix}
x b_{11} & 1 + (1-t) b_{12} \\
1 + (1-t^{-1})\overline{b_{12}} & xb_{22}
\end{pmatrix},
\]
where we write $x = (1-t) + (1-t^{-1}) = (1-t)(1-t^{-1})$.
Furthermore
\[
A_{ij}  = (1-t) b_{ij}, \qquad
A_{ji}  = (1-t^{-1}) \overline{b_{ij}}
\]
for $i \in \{1,2\}$ and $j > 2$ and some polynomials $b_{ij} \in \Lambda$.
We will now consider the parity of $b_{11}(1)$, and in each case give a transformation matrix over $\Lambda_0$
with determinant a unit in $\Lambda$, which transforms $A(t)$ into a matrix $C(t)$ with odd $C(1)$. In this way the case
that $A(1)$ is even is reduced to the case that $A(1)$ is odd, which has already been discussed.

If $b_{11}(1)$ is even, add $1/(1-t)$ times the first row to the second, and then $1/(1-t^{-1})$ times the first column to the second.
This is a base change over $\Lambda_0$ coming from a transformation matrix with determinant~$1$.
It yields a Hermitian matrix $C(t)$ over $\Lambda$ with top-left $2\times 2$ submatrix
\[
\begin{pmatrix}
x b_{11} & 1 + (1-t) (b_{12} + b_{11}) \\
1 + (1-t^{-1})(b_{11} + \overline{b_{12}}) & xb_{22} + 1 + b_{11} + b_{12} + \overline{b_{12}}
\end{pmatrix}.
\]
One finds $C_{22}(1) = 0 + 1 + b_{11}(1) + 2b_{12}(1)$ to be odd.
This concludes the case that $b_{11}(1)$ is even.

If $b_{11}(1)$ is odd, one proceeds similarly:
one may divide the first row by $(1-t)$, and multiply the second row by $(1-t^{-1})$,
and apply the corresponding changes to the columns.
This is a base change over $\Lambda_0$ coming from a transformation matrix with determinant~$-t^{-1}$.
It yields a Hermitian matrix $C(t)$ over $\Lambda$ with top-left $2\times 2$ submatrix
\[
\begin{pmatrix}
b_{11} & 1 + (1-t) b_{12} \\
1 + (1-t^{-1})\overline{b_{12}} & x^2b_{22},
\end{pmatrix}.
\]
Clearly $C_{11}(1) = b_{11}(1)$ is odd, which concludes the case of odd $b_{11}(1)$.
\end{proof}
We have thus completed the proof of the implication \eqref{item:Blanchfield}~$\Rightarrow$~\eqref{item:unknotting},
and turn to the next part of the proof.
\subsectionpdfbookmark{3D--cobordisms from crossing changes---\eqref{item:unknotting}~$\Rightarrow$~\eqref{item:Seifertsurface}}{3D--cobordisms from crossing changes---\eqref{item:unknotting}~=>~\eqref{item:Seifertsurface}}\label{subsec:(3)=>(2)}
This section is devoted to the proof of the following proposition, from which \eqref{item:unknotting}~$\Rightarrow$~\eqref{item:Seifertsurface} follows since it corresponds to the special case that $K'$ is a knot with Alexander polynomial 1.
\begin{prop}\label{prop:crchangestoSS}
Let $g$ be a non-negative integer.
If a knot $K'$ can be obtained from a knot $K$ by changing $g$ positive and $g$ negative crossings (in any order),
then there exists an oriented connected surface $\Sigma$ of genus $g$ in $S^3$ with oriented boundary a two-component link whose components are isotopic to $K'$ with reversed orientation and $K$, respectively.
\end{prop}
\begin{figure}[h]%
(a) \includegraphics{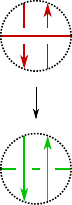}\hfill
(b) \raisebox{0pt}{\includegraphics[width=.38\textwidth]{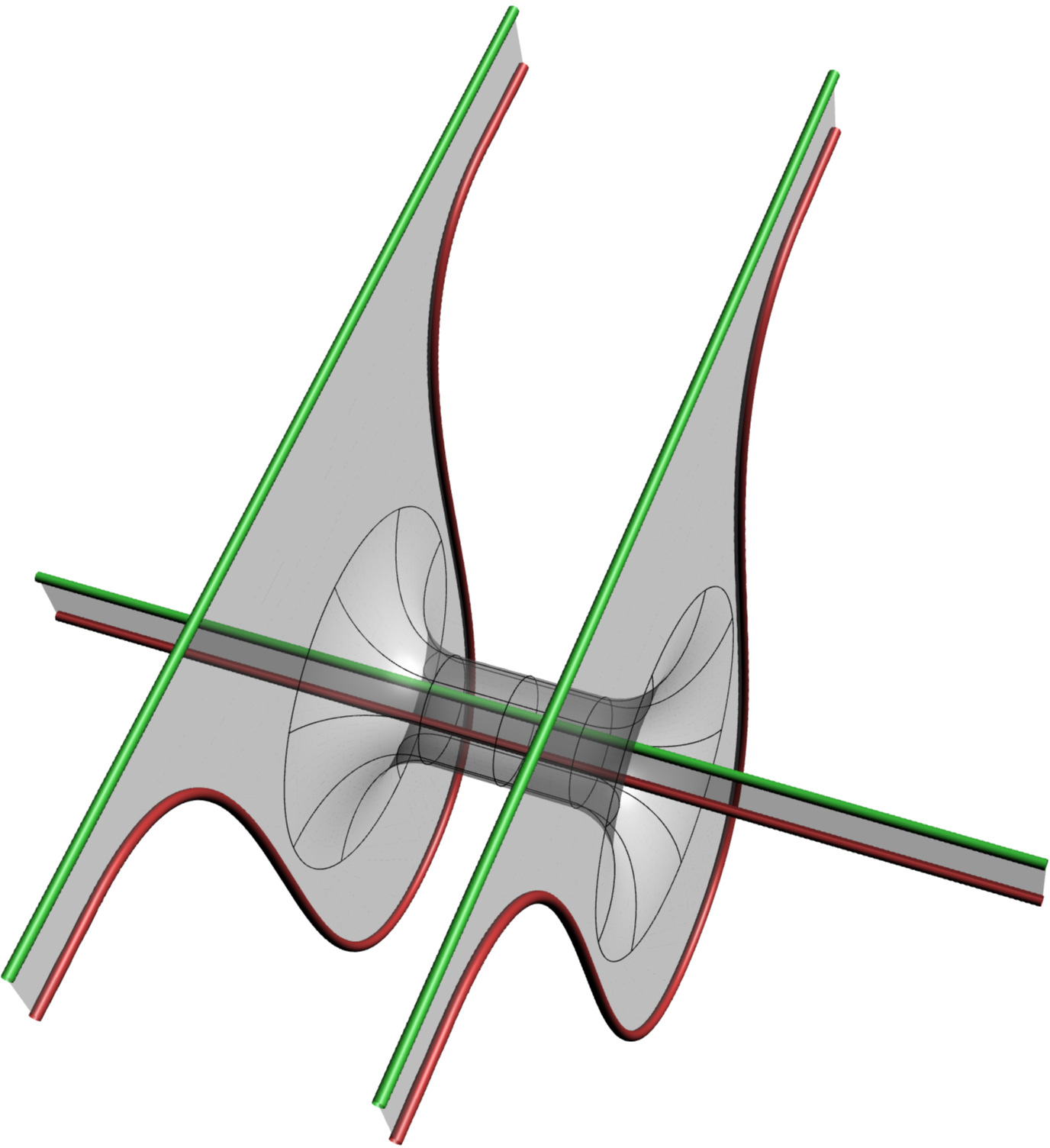}}\hfill
\raisebox{0pt}{\includegraphics[width=.38\textwidth]{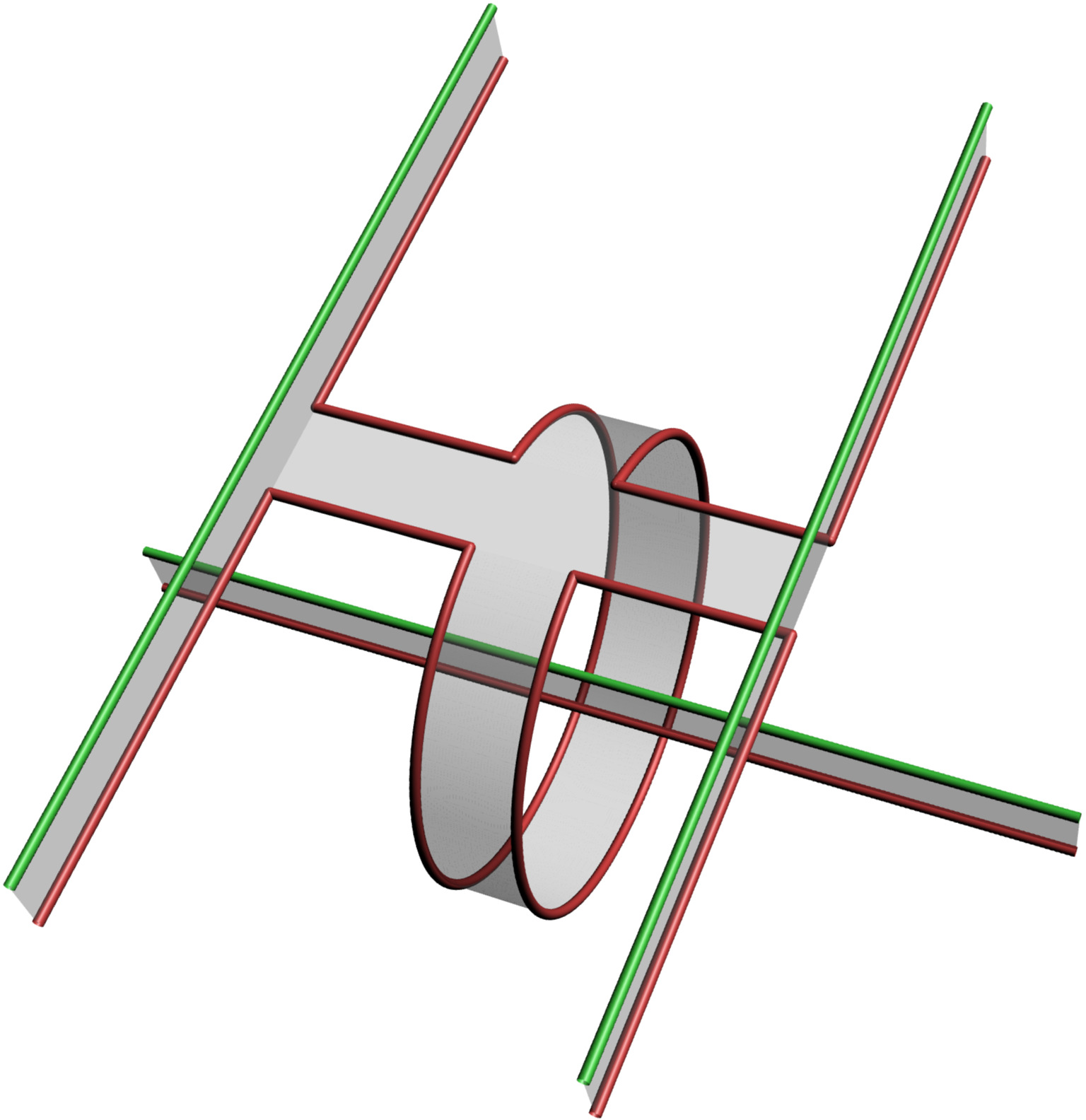}}
\caption{(a) A balanced crossing change, diagrammatically.
(b) Two isotopic drawings of the intersection of ball with a genus 1 Seifert surface realizing a balanced crossing change.}
\label{fig:3d}
\end{figure}
\begin{figure}[h]
\begin{center}
\begingroup%
  \makeatletter%
  \providecommand\color[2][]{%
    \errmessage{(Inkscape) Color is used for the text in Inkscape, but the package 'color.sty' is not loaded}%
    \renewcommand\color[2][]{}%
  }%
  \providecommand\transparent[1]{%
    \errmessage{(Inkscape) Transparency is used (non-zero) for the text in Inkscape, but the package 'transparent.sty' is not loaded}%
    \renewcommand\transparent[1]{}%
  }%
  \providecommand\rotatebox[2]{#2}%
  \newcommand*\fsize{\dimexpr\f@size pt\relax}%
  \newcommand*\lineheight[1]{\fontsize{\fsize}{#1\fsize}\selectfont}%
  \ifx\svgwidth\undefined%
    \setlength{\unitlength}{346.37349286bp}%
    \ifx\svgscale\undefined%
      \relax%
    \else%
      \setlength{\unitlength}{\unitlength * \real{\svgscale}}%
    \fi%
  \else%
    \setlength{\unitlength}{\svgwidth}%
  \fi%
  \global\let\svgwidth\undefined%
  \global\let\svgscale\undefined%
  \makeatother%
  \begin{picture}(1,0.87768131)%
    \lineheight{1}%
    \setlength\tabcolsep{0pt}%
    \put(0,0){\includegraphics[width=\unitlength,page=1]{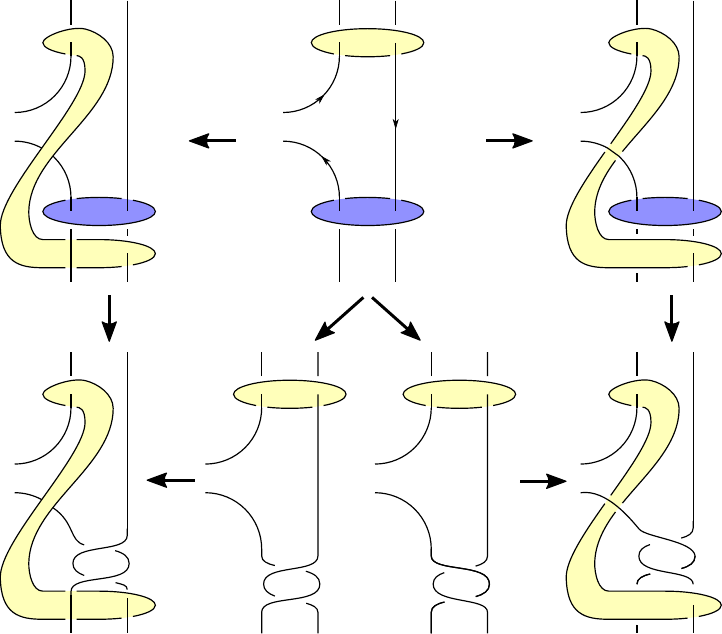}}%
    \put(0.23824942,0.18601715){\color[rgb]{0,0,0}\makebox(0,0)[t]{\lineheight{1.25}\smash{\begin{tabular}[t]{c}{\footnotesize isotopy}\end{tabular}}}}%
    \put(0.08882142,0.44850906){\color[rgb]{0,0,0}\makebox(0,0)[t]{\lineheight{1.25}\smash{\begin{tabular}[t]{c}\footnotesize crossing\\\footnotesize change\end{tabular}}}}%
    \put(0.75132558,0.18601715){\color[rgb]{0,0,0}\makebox(0,0)[t]{\lineheight{1.25}\smash{\begin{tabular}[t]{c}{\footnotesize isotopy}\end{tabular}}}}%
    \put(0.3892565,0.44850906){\color[rgb]{0,0,0}\makebox(0,0)[t]{\lineheight{1.25}\smash{\begin{tabular}[t]{c}\footnotesize crossing\\\footnotesize change\end{tabular}}}}%
    \put(0.6247312,0.44850906){\color[rgb]{0,0,0}\makebox(0,0)[t]{\lineheight{1.25}\smash{\begin{tabular}[t]{c}\footnotesize crossing\\\footnotesize change\end{tabular}}}}%
    \put(0.86832547,0.44850906){\color[rgb]{0,0,0}\makebox(0,0)[t]{\lineheight{1.25}\smash{\begin{tabular}[t]{c}\footnotesize crossing\\\footnotesize change\end{tabular}}}}%
  \end{picture}%
\endgroup%

\end{center}
\vspace{5ex}
\caption{Top: the result of performing the two crossing changes in the middle is the same as the result of performing those to the left or right, where left or right depends on the sign of the Dehn surgery along the boundary of the blue (gray-scale: dark) disk. The sign of the Dehn surgery along the boundary of the yellow (gray-scale: light) disk is not relevant.
\newline
Bottom: the result of performing one of the crossing changes---performing $-1$-framed Dehn surgery (left) and $1$-framed Dehn surgery (right) on the boundary of the blue disk---yields isotopic yellow crossing disks; hence isotopic knots when both crossing changes are performed.}
\label{fig:sliding}
\end{figure}
\begin{proof}
Let us first consider two simple cases.
If $g = 0$, just take $\Sigma$ as a knotted band.
Next, consider the case that $g = 1$ and that the two crossing changes happen inside of a small ball as shown in \Cref{fig:3d}(a).
Then, one may construct $\Sigma$ by gluing a knotted band (as in the case $g=0$) outside of the ball,
and the surface shown in \Cref{fig:3d}(b) inside of the ball.
In fact, a similar construction---gluing $g$ copies of the surface shown in \Cref{fig:3d}(b) inside of $g$ balls, and a knotted band in the complement of the $g$ balls---gives the desired surface in the general case, due to the following lemma.
\end{proof}
\begin{lemma}
Let $K$ and $K'$ be knots as in \Cref{prop:crchangestoSS}.
Then there exist pairwise disjoint balls $B_1, \ldots, B_g$,
such that $K$ and $K'$ agree outside of $\bigcup_i B_i$, and for all~$i$, $B_i \cap K$ and $B_i \cap K'$
look like the top and bottom of \Cref{fig:3d}(a), respectively.
\end{lemma}
\begin{proof}
By the assumption on $K$ and~$K'$,
one may choose $2g$ crossing disks
such that the corresponding surgeries (in the right order) transform $K$ into~$K'$,
and there are $g$ surgeries of each sign.
Since the crossings are not changed simultaneously, the crossing disks
may a priori intersect. However, by a general position argument the disks
may be chosen to be disjoint (see e.g.~\cite[Prop.~1.5]{MR1628751} for details); in other words, crossing changes
may be assumed to happen simultaneously.

Finally, 
we use that one may arbitrarily modify the order in which
the $4g$ intersection points of crossing disks with $K$ occur on $K$. This follows from the fact that modifying two crossing changes (i.e.~modifying a choice of two $\pm1$-framed unknots that arise as the boundary of crossing disks) in a $3$--ball intersecting the knot and crossing disks as depicted in the upper half of \Cref{fig:sliding}, does not alter the outcome of the crossing changes; see bottom half of \Cref{fig:sliding}.

We modify this order such that the $2g$ disks can be arranged in $g$ pairs $D, D'$
with the following properties.
The surgeries corresponding to $D$ and $D'$ are of opposite sign;
and there is a closed interval $I \subseteq K$ with endpoints on $D$ and $D'$, such that $I^{\circ}$ does
not intersect any other crossing disks.
Now, for each such pair take a ball that is a neighborhood of $D\cup D'\cup I$, and make these $g$ balls
small enough so that no two of them intersect. These balls form the desired collection $B_1, \ldots, B_g$.
\end{proof}

\sectionpdfbookmark{More general ambient 3-- and 4--manifolds (proof of \Cref{thm:mains2s2})}{Changing the ambient 3-- and 4--manifold (proof of Theorem~\ref{thm:mains2s2})}
\label{sec:s2s2}
The proof of \Cref{thm:mains2s2} follows the same structure as the proof of \Cref{thm:main}.
We note some preliminary lemmas.
The first one explains why $\Z$--slice surfaces are automatically null-homologous.
\begin{lemma}\label{lem:Zslicesurfacesarenullhomo}
Let $F$ be a properly locally-flatly embedded surface in a four-manifold
\[
V = B \# (\mathbb{C}P^2\#\overline{\mathbb{C}P^2})^{c_1}\# (S^2\times S^2)^{c_2},
\]
where $c_1, c_2\geq 0$ and $B$ is a compact orientable contractible topological
four-manifold with boundary an integral homology sphere $M$.
If $H_1(V\setminus F;\Z)\cong \Z$, then $F$ is null-homologous in  $V$.
\end{lemma}

\begin{proof}

For $F$ to be null-homologous in $V$ means that the inclusion $i\colon (F,K)\to (V,M = \partial V)$
induces the zero map on second homology groups, i.e.~$i_*[F] = 0\in H_2(V,M; \Z)$
for $[F] \in H_2(F,K; \Z)$ the fundamental class of $F$.
Since $H_1(\partial V; \Z) \cong H_2(\partial V; \Z) \cong 0$, the inclusion $j\colon (V,\varnothing)\to (V,M)$
induces an isomorphism on second homology groups.
So to establish the lemma it will be sufficient to show that
\[(j_*)^{-1}(i_*([F])) \eqqcolon e = 0 \in H_2(V;\Z).
\]

Since $F$ and $V\setminus F$ are disjoint, the intersection form $H_2(V; \Z) \times H_2(V; \Z) \to \Z$
of $V$ evaluates to $0$ on $(e, e')$ for all $e'\in \im k_* \subset H_2(V; \Z)$ for $k$ the inclusion map $V\setminus F \to V$ and $k_*$ the induced map $H_2(V\setminus F;\Z)\to H_2(V;\Z)$.
Since $M$ is an integral homology sphere, the intersection form is non-degenerate.
So to show $e = 0$ and thus conclude the proof, it only remains to show that $k_*$ is surjective.

For this, let $\nu F\subset V$ be a tubular neighborhood of $F$, and consider the Mayer-Vietoris sequence of $\nu F\cup (V\setminus F)$
of integral homology groups:
\[
0\to H_2(V\setminus F) \xrightarrow{k_*} H_2(V) \to H_1(\nu F\setminus F) \xrightarrow{\alpha} H_1(\nu F) \oplus H_1(V\setminus F) \to 0.
\]
Both domain and target of $\alpha$ are free abelian of rank $2g(F)+1$, and $\alpha$ is surjective.
It follows that $\alpha$ is an isomorphism, and hence so is $k_*$.
\end{proof}

The next three lemmas are preparation for framing considerations for surgeries in $B^4$ in the proof of \Cref{thm:mains2s2} depending on the Arf invariant of $K$.

Let $F$ be a Seifert surface of genus $g$ of a link in some $\Z HS^3$.
We denote the intersection form on $H_1(F;\Z)$ by $\langle\,\cdot\,,\cdot\,\rangle$.
A \emph{symplectic basis of $F$} is a tuple 
$(a_1,b_1,a_2,b_2,\ldots,a_g,b_g)$ of homology classes in $H_1(F;\Z)$
such that $\langle a_i,b_k\rangle=\delta_{ik}$
and $\langle a_i, a_k\rangle = \langle b_i, b_k\rangle = 0$ for all $i, k$.
Such a symplectic basis generates a summand of rank $2g$ of $H_1(\Sigma;\Z)$,
and descends to a basis of $H_1(\Sigma;\Z) / \iota_*H_1(\partial\Sigma;\Z)$, where $\iota$ denotes the inclusion $\partial\Sigma\to\Sigma$.
A \emph{half basis of $F$} is a tuple $(a_1, \ldots, a_g)$ of homology classes in $H_1(F;\Z)$
that may be extended to a symplectic basis $(a_1, b_1, \ldots, a_g, b_g)$ of $F$.

A link $L$ is called \emph{proper} if each of its components $K$ has even linking number with $L\setminus K$
(for example, all knots are proper links).
The Arf invariant of a proper link $L$ may be defined as follows (see~\cite[Chapter~10]{Lickorish_97}).
Let $F$ be a genus $g$ Seifert surface of~$L$, and pick a symplectic basis $(a_1,b_1,a_2,b_2,\ldots,a_g,b_g)$ for $F$.
For $v\in H_1(F;\Z)$, let $q(v)\in\Z/2\Z$ denote the reduction mod 2 of the Seifert form evaluated at $(v,v)$. Then
\[\Arf(L)\coloneqq\sum_{i=1}^gq(a_i)q(b_i).\]
This definition is independent of the choice of $F$ and of the symplectic basis.
\begin{lemma}\label{lem:curves}
Let $F$ be Seifert surface of a proper link $L$ in some $\Z HS^3$.
\begin{enumerate}[label=(\roman*)]
\item There exists a half basis $a_1, \ldots, a_g$ of $F$
such that the framing induced by $F$ on $a_i$ is even for $2 \leq i \leq g$ and odd for $a_1$.
\item There exists a half basis $a_1, \ldots, a_g$ of $F$
such that the framings induced by~$F$ of $a_i$ are even for $1 \leq i \leq g$ if and only if $\Arf(L) = 0$.
\end{enumerate}
\end{lemma}
\begin{proof}
(i)
Let us first show that there exists a symplectic basis
$(a_1, b_1, \ldots, a_g, b_g)$ of $F$ at least one element of which has odd framing.
Indeed, if all basis elements have even framing, then replacing $a_1$ by $a_1 + b_1$ yields such a symplectic basis.
Thus, there exists a half basis $a_1, \ldots, a_g$ of $F$ and $k\in\{1,\ldots,g\}$
such that $a_i$ has odd framing if and only if $i \leq k$.
Then, $(a_1, a_2 + a_1, \ldots, a_k + a_1, a_{k+1}, \ldots, a_g)$ is a half basis in which only the first element has odd framing.

(ii) The `only if' direction follows directly from the definition of the Arf invariant.
For the `if' direction, assume that $\Arf(L) = 0$ and let $a_1, \ldots, a_g$ be a half basis as in (i).
We have $0 = \Arf(L) = q(a_1)q(b_1)$, so $q(b_1) = 0$. Thus $a_1 + b_1, a_2, \ldots, a_g$ is a desired half basis.
\end{proof}

Let us say a Seifert surface $F$ is \emph{$H_1$--null} (or \emph{$H_1$--null mod 2}) if
every simple closed curve on $F$ has zero (or even) linking number with every boundary component of $F$.
Note that the boundary of an $H_1$--null mod 2 surface is a proper link.
Equivalently, $F$ is $H_1$--null (or {$H_1$--null mod 2}) if the inclusion
of $F^{\circ}$ into the complement of $\partial F$ induces the zero map on the first homology group with integer coefficients (or with $\Z/2$--coefficients). Denoting by $\beta^+\subset S^3\setminus F$ a positive normal push-off of a curve $\beta$ in $F$, we have the following technical lemma.
\begin{lemma}\label{lem:nice3dcobo}
Let $F$ be a Seifert surface with boundary components $L_1, \ldots, L_n$ in some $\Z HS^3$.
\begin{enumerate}[label=(\roman*)]
\item There exists an $H_1$--null Seifert surface $F'$ with an  orientation preserving homeomorphism ${f\colon F\to F'}$ such that for all ${j\in\{1,\dots, n\}}$, $f(L_j)$ is isotopic to $L_j$.
\item If $F$ is $H_1$--null mod 2, then there exist $f$ and $F'$ as in (i)
such that the linking number of $\alpha$ and $\alpha^+$ has the same parity as the linking number of $f(\alpha)$ and $f(\alpha)^+$ for all curves $\alpha$ in $F$.
\end{enumerate}
\end{lemma}
In other words, \Cref{lem:nice3dcobo}~(i) says that every link $L$ can be changed to a link $L'$ with isotopic components that has a $H_1$--null Seifert surface $F'$ with the same genus as a given Seifert surface of $L$. \Cref{lem:nice3dcobo}~(ii) says that furthermore, in case $L$ admits an $H_1$--null mod 2 Seifert surface $F$, $F'$ can be chosen with an identification with $F$ that preserves the Seifert form mod $2$.
\begin{proof}
We only sketch the proof, and refer the reader to \cite[Proof of Lemma~18]{FellerLewark_16} for details.
The Seifert surface $F$ has a handle decomposition into one 0--handle, 1--handles $h_1, \ldots, h_{2g(F)}$
whose belt spheres lie in $L_n$, and 1--handles $h_1', \ldots, h_{n-1}'$, such that the belt sphere of $h_j'$
has one point in $L_j$ and one point in $L_n$. The cores of the $1$--handles can be closed off (by two intervals to the core of the $0$--handle) to simple closed curves $\alpha_1,\dots, \alpha_{2g},\alpha'_1,\dots, \alpha'_{n-1}$ that form a basis of $H_1(F;\Z)$. The curve $\alpha'_j$ is parallel to the boundary $L_j$.

For (i), we construct $F'$ in three steps, by modifying $F$ by twisting along $1$--handles (i.e. integer Dehn surgery on the boundary of a disk that intersects $F$ transversely in cocores of $1$-handles). Each step yields a new surface in $S^3$, which comes with a canonical abstract identification homomorphism to the former one, and we use the same notation for the canonically identified $\alpha_i$, $L_j$, and $\alpha'_j$.
\begin{enumerate}[label=\arabic*)]
\item For all $1\leq j<k\leq n-1$, twist the 1--handles $h_{j}'$ and $h_k'$ around each other to make the linking number of $\alpha'_j$ and $\alpha'_k$ zero. 

\item For all pairs $1 \leq i\leq 2g$ and $1\leq j\leq n-1$, twist the 1--handle $h_i$ around $h_j'$ to make the linking number of $\alpha_i$ and $\alpha'_j$ zero. This can be done without 
changing the isotopy classes of the boundary components or the linking numbers arranged to be zero in step~1). 

\item Insert twists into the 1--handles $h_j'$ for all $1\leq j\leq n-1$ to make the framing of $\alpha'_j$ (which is the linking number between $\alpha'_j$ and $L_j$ since $\alpha'_j$ is parallel to $L_j$) zero.

\end{enumerate}
In conclusion, we have built an $F'$ in which all $\alpha_i$ and all $\alpha'_j$ have linking number zero with all $L_k$.
We set $f$ to be the iterative composition of the abstract identification homeomorphisms between a surface and the result of twisting along a properly embedded interval in a surface (the cocores of the 1--handles).

To prove (ii), observe that if $F$ is $H_1$--null mod 2, then all twisting done in the above steps was by introducing an \emph{even} number of twists. (Indeed, the linking numbers of $\alpha'_j$ and $\alpha'_k$, and of $\alpha_i$ and $\alpha'_j$ are even, so the twistings in 1) and 2), respectively, have to be even. Similarly, since every $\alpha_j$ has even framing, which is not affected by 1)~and 2), 3)~necessarily has to be done with an even number of twists.) In other words, $F'$ is obtained from $F$ by even surgeries along curves in the complement of $F$; hence, the canonical identification of $F$ and $F'$ via $f$ preserves the Seifert form mod 2.
\end{proof}
\begin{lemma}\label{lem:arfadditive}
Let $F$ be a Seifert surface of a link $L$ in some $\Z HS^3$. Assume $F$ is $H_1$--null mod 2,
and has two boundary components $L_1, L_2$. Then
$\Arf(L) = \Arf(L_1) + \Arf(L_2)$.
\end{lemma}
\begin{proof}
By \cref{lem:nice3dcobo}(ii), there exists an $H_1$--null Seifert surface $F'$
with a homeomorphism $f\colon F\to F'$ such that $f(L_j)$ is isotopic to $L_j$ for $j\in\{1,2\}$,
and $f$ preserves the framing of curves in the surfaces mod 2.
Pick a Seifert surface $\Sigma$ of~$L_2$. There is a stabilization $\Sigma'$ of $\Sigma$ such that
$\Sigma' \cap F' = L_2$ (see \cite[Proof of Lemma~18]{FellerLewark_16} for details).
It is clear from the definition of the Arf invariant that
\begin{align*}
\Arf(\partial (\Sigma' \cup F')) & = \Arf(\partial \Sigma') + \Arf(\partial F') \\
\Rightarrow
\Arf(f(L_1)) & = \Arf(f(L_2)) + \Arf(f(L)),
\end{align*}
and this implies the statement of the lemma since $f$ preserves framing mod 2.
\end{proof}

\begin{proof}[Proof of \Cref{thm:mains2s2}]
(1') $\Rightarrow$ (4'):
The proof of (1) $\Rightarrow$ (4) in \Cref{sec:4D} can be adapted with minimal changes,
using that $\sigma(V) = 0$ and $\pi_1(V) = 1$ and that Borodzik-Friedl's \cref{thm:BF:W}
holds verbatim for knots in integer homology three-spheres.

(4') $\Rightarrow$ (3'):
This part of the proof is purely 3--dimensional and makes no reference to the four-manifold $V$.
So all we need for our proof of (4) $\Rightarrow$ (3) in \cref{subsec:(4)=>(3)} to adapt is that
Borodzik-Friedl's \cref{thm:BoroFriedl} holds for knots in any $\Z HS^3$; see \cite[Rmk.~5.2]{BorodzikFriedl_14_OnTheAlgUnknottingNr}.

(3') $\Rightarrow$ (2'):
The geometric construction of the surface in \Cref{subsec:(3)=>(2)}, which proves (3) $\Rightarrow$ (2) works in any $\Z HS^3$, without changes. In fact, the statement of \Cref{prop:crchangestoSS} and its proof carry over to knots $K'$ and $K$ in any 3--manifold $M$ rather than $S^3$ (even without any assumptions about the homology of $M$).

(2') $\Rightarrow$ (1'):
This is the only step that goes beyond a straight-forward generalization of the corresponding step in \cref{thm:main}.
Let $\Sigma$ be a genus $g=h+c=h+c_1+c_2$ 3D--cobordism between $K$ and a knot $J$ with Alexander polynomial 1. Note that $\Arf(J)=0$.
By \cref{lem:nice3dcobo}, we may and do choose $\Sigma$ to be $H_1$--null.

\begin{figure}[ht]
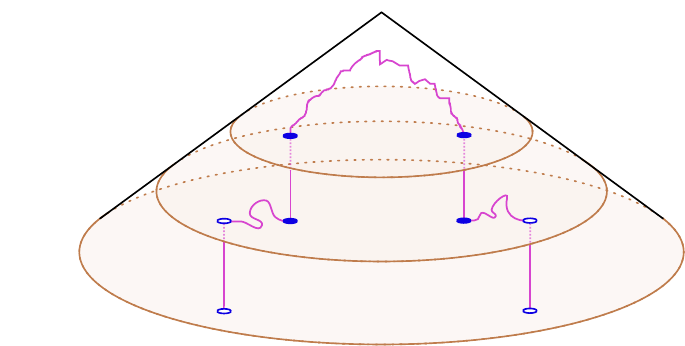
\caption{Schematic drawing of $F'$ (dimensions reduced by one). 1--, 2--, 3-- and 4--manifolds are drawn in blue (small ellipses), purple (lines), brown (cone levels) and black (cone over large ellipse), respectively.}
\label{fig:s2s2schema}
\end{figure}
We start as in (2) $\Rightarrow$ (1). See \cref{fig:s2s2schema} for the following construction.
Construct a $\Z$--slice surface $F'$ by taking the union of $\Sigma\subset M=\partial B$ with a $\Z$--slice disc in $B$ for $J$, using that \Cref{thm:Alex1<=>Zslice} holds in any $\Z HS^3$, and then pushing the interior of this union into $B$. To be more precise, we identify $B$ with a copy of $B$ (we denote it by $B_\frac{1}{2}$) union $M\times[\frac{1}{2},1]$ obtained by canonically gluing $\partial B$ to $M\times\{\frac{1}{2}\}$.
Now we construct $F'$ as follows. We take a $\Z$--slice disk $D$ in
\[B_{\tfrac{1}{2}}\subset B=B_{\frac{1}{2}}\cup M\times[\tfrac{1}{2},1],\]
with boundary $J\times\{\frac{1}{2}\}\subset M\times[\frac{1}{2},1]$ and
the surface $\Sigma'\subset M\times[\frac{1}{2},1]$ given as
\[\Sigma\times\{\tfrac{3}{4}\}\cup J\times[\tfrac{1}{2},\tfrac{3}{4}]\cup K\times[\tfrac{3}{4},1],\] and set $F'\coloneqq D\cup\Sigma'$.

The argument to show $\pi_1(B\setminus F')\cong\Z$ is similar as in the case of $M=S^3$ and $B=B^4$, which is needed in the proof of (2) $\Rightarrow$ (1) and for which a detailed argument is available in~\cite[Proof of Claim~20]{FellerLewark_16}. Here is a brief version. By the `rising water principle' (see \cite[Proof of Proposition~6.2.1]{GompfStipsicz_99}), $(M\times[\tfrac{1}{2},1]\setminus\nu(\Sigma'),M\times\{\tfrac{1}{2}\}\setminus \nu(J\times\{\frac{1}{2}\})$ has a relative handle decomposition with only $2$--handles (here we either think in the category of smooth manifolds with corners or we smooth the corners). As a consequence, the inclusion
\[M\times\{\tfrac{1}{2}\}\setminus (J\times\{\tfrac{1}{2}\})\subset M\times[\tfrac{1}{2},1]\setminus\Sigma'\]
induces a surjection on $\pi_1$, and, thus, $F'=D\cup\Sigma'$ is a $\Z$--slice surface in $B$.

Next, we want to specify $c$ many simple closed framed curves $\gamma_i'$ on $\Sigma'\subset F'$ such that surgery along them corresponds to connected sum with $S^2\times S^2$ or $\mathbb{C}P^2\# \overline{\mathbb{C}P^2}$ and such that $\Sigma'$ may be be compressed $c$ times in the resulting manifold. 

One picks disjoint simple closed curves $\gamma_1, \ldots, \gamma_c$ in $\Sigma^{\circ}$
with the property that their homology classes may be extended to a symplectic basis of $\Sigma$. If $c_1=1$, arrange for exactly one of them to have odd framing induced by $\Sigma$. If instead $c_1=0$, choose all of them to have even framing. This is possible by \Cref{lem:curves}. Note that we need the assumption $\Arf(K)=0(=\Arf(J))$ exactly in case $h=c_1=0$. Indeed, in this case the homology classes of the $\gamma_i$ form a half basis of $\Sigma$, so choosing all of them to have even framing is possible since $\Arf(K\cup J) = \Arf(K) + \Arf(J) = 0$ by \Cref{lem:arfadditive}

We denote by $\gamma_i'$ the curves in $\Sigma'$ that map to $\gamma_i$ under the projection to the first factor
$M\times[\frac{1}{2},1]\to M$. Note that $\gamma_i' \subset M \times \{\tfrac34\}$.
Let $I$ be the interval $[-\epsilon,\epsilon]$, with $\epsilon$ chosen in $ (0,\tfrac{1}{4})$ for the explicit formulas below to work. Let
\[A=A(\gamma_1)\cup\cdots \cup A(\gamma_c)\] denote a closed regular neighborhood of the union of the $\gamma_i$, where $A(\gamma_i)\cong S^1\times I$ denotes a regular neighborhood of $\gamma_i$ in $\Sigma$. We fix a trivialization $\Phi\colon A\times I\hookrightarrow M$ of a regular neighborhood of $A\subset M$ with $\Phi(a,0)=a$ for all $a\in A$ (e.g.~take the restriction of a trivialization of an embedded normal bundle of $\Sigma$ in $M$). We further let \[\Psi\colon A\times I\times I\hookrightarrow M\times[\tfrac{1}{2},1],\quad(s,t,h)\mapsto (\Phi(s,t),\tfrac{3}{4}+h),\] which gives a trivialization of a regular neighborhood of $A'\coloneqq A\times\{\tfrac{3}{4}\}$ in\break ${M\times[\tfrac{1}{2},1]\subset B}$. %

For each $\gamma_i'$, we use $\Psi$ to identify the regular neighborhood $N_i=\Psi(A(\gamma_i)\times I\times I)$ with $A(\gamma_i)\times I\times I\cong S^1\times I\times I\times I$. %
We now perform surgery on $M\times[\tfrac{1}{2},1]$ and thus $B=B_{\tfrac{1}{2}}\cup M\times[\tfrac{1}{2},1]$ by removing the interior of $N_i$ and gluing $D^2\times \partial (I\times I\times I)$ along $S^1\times \partial (I\times I\times I)$ to the boundary of $N_i$ via $\Psi$. %
We modify $\Sigma'$ by removing $\Sigma'\cap N_i^\circ$ and adding $D^2\times \partial I \times \{0\}\times \{0\}\subset D^2\times \partial (I\times I\times I)$. In other, words we have ambiently compressed $\gamma_i'$ in $\Sigma'$.
Performing such a surgery for all the $\gamma_i'\subset B$ together with the corresponding modification on $\Sigma'\subset F'$ yields a $4$--manifold $V$ and an embedded surface $F$ of genus $g-c=h$. 

It remains to check that $\pi_1(V\setminus F) \cong \Z$ and that $V$ is homeomorphic to
\[B \# (\mathbb{C}P^2\#\overline{\mathbb{C}P^2})^{\#c_1}\# (S^2\times S^2)^{\#c_2}.\]
First, we calculate $\pi_1(V\setminus F)$.
Writing
\[V\setminus F=\left(B\setminus (F'\cup \bigcup_{i=1}^b N_i)\right) \cup \left(\bigcup_{i=1}^b R_i\right),\]
where $R_i\cong D^2\times (S^2\setminus \{N,S\})$ is the copy of
$D^2\times \partial (I\times I\times I)\setminus (D^2\times \partial I \times \{0\}\times \{0\}))$ in $V$ corresponding to $\gamma_i$,
the Seifert-van Kampen theorem shows that $\pi_1(V\setminus F)$ is a quotient of $\pi_1(B\setminus (F'\cup \bigcup_{i=1}^b N_i))\cong \pi_1(B\setminus F')\cong\Z$.
In fact, this quotient is given by killing the free conjugation classes of $\Psi(\gamma_i\times\{-\epsilon\}\times\{-\epsilon\})$ in $B\setminus F'$, which is homotopy equivalent to $B\setminus \left(F'\cup \bigcup_{i=1}^b N_i^\circ)\right)$.
However, $\Psi(\gamma_i\times\{-\epsilon\}\times\{-\epsilon\})$ is null-homologous in $B\setminus F'$ (and thus null-homotopic) for all $i$. For this, note that $\Psi(\gamma_i\times\{-\epsilon\}\times\{-\epsilon\})$ is homotopic to
$$
\Psi(\gamma_i\times\{0\}\times\{-\epsilon\})
\subset \Sigma^{\circ} \times \{\tfrac{3}{4}-\epsilon\} \subset (M\setminus J)\times\{\tfrac{3}{4}-\epsilon\}\subset B\setminus \Sigma'\subset B\setminus F',$$
and all the $\gamma_i\subset\Sigma^{\circ}$ are null-homologous in $M\setminus J$ by our choice of $\Sigma$ being $H_1$--null.

Secondly, we describe $V$. Recall that embeddings of a finite union of circles in a simply-connected $4$--manifold are unknotted, i.e.~there exist disjoint embedded 4--balls, one for each component, such that each component lies in a standard position (say as the boundary of some standard 2-disc) in one of the 4--balls.
Therefore, surgery along each component corresponds to changing the manifold by connect summing with the result of performing surgery along a simple closed curve in standard position in $S^4$. Surgery along such a simple closed curve in $S^4$ yields either $S^2\times S^2$ or $\mathbb{C}P^2\# \overline{\mathbb{C}P^2}$, depending on the framing.
Hence, we already know that $V$ is $B \# (\mathbb{C}P^2\#\overline{\mathbb{C}P^2})^{\#c_1'}\# (S^2\times S^2)^{\#c_2'}$ for some non-negative integers $c_1'$ and $c_2'$ with $c_1'+c_2'=g-h$.
To determine  that $c_1'=c_1$ and $c_2'=c_2$, one can check that the surgeries along $\gamma_i'$ we performed to find $V$ yielded connected summing with $S^2\times S^2$ when the framing of $\gamma_i$ induced by $\Sigma$ is even and connected summing with $\mathbb{C}P^2\# \overline{\mathbb{C}P^2}$ when the framing induced by $\Sigma$ is odd. We provide a way to check this using the intersection form in the following, last two paragraphs of the proof.

For each $i$, the second homology classes $a_i$ and $b_i$ in $H_2(V;\Z)$ represented by
\begin{align*}
S_i &\coloneqq\{0\}\times \partial(I\times I\times I) \\
T_i &\coloneqq(D^2\times\{0\}\times\{0\}\times\{\epsilon\})
\cup (\Sigma_i\times\{\tfrac{3}{4}+\epsilon\}),
\end{align*}
respectively. Here, $S_i$ and $(D^2\times\{0\}\times\{0\}\times\{\epsilon\})$ are viewed as subsets of the $D^2\times \partial(I\times I\times I)$ subset of $V$ corresponding to $i$, and $\Sigma_i\subset M$ is a Seifert surface for $\gamma_i$. We have that $(a_1,b_1,\ldots, a_c,b_c)$ is a basis of $H_2(V;\Z)\cong \Z^{2c}$. %

One checks that %
$a_i$ is also represented by a sphere disjoint of $S_i$, $S_i$ and $T_i$ intersect transversely in one point, and $b_i$ can be represented by a surface $T_i'$ that algebraically intersects $T_i$ $n_i$ times, where $n_i$ denotes the framing of $\gamma_i$ induced by~$\Sigma$. Hence, the intersection form on $H_2(V;\Z)$ restricted to $\langle a_i,b_i\rangle\subset H_2(V;\Z)$ is given by
$\begin{pmatrix}
0 & \pm 1\\
\pm 1 & n_i
\end{pmatrix}$. This in turn is equivalent to $\begin{pmatrix}
0 & \pm 1\\
\pm 1 & 0
\end{pmatrix}$ if $n_i$ is even, and to $\begin{pmatrix}
0 & \pm 1\\
\pm 1 & 1
\end{pmatrix}$ and thus $\begin{pmatrix}
-1 & 0\\
0 & 1
\end{pmatrix}$ if $n_i$ is odd. The former is the intersection form of $S^2\times S^2$ while the later is the intersection form of $\mathbb{C}P^2\# \overline{\mathbb{C}P^2}$, which proves that the surgery corresponding to $\gamma_i$ is as claimed.
\end{proof}

\section{Linking pairings of cyclic branched coverings of prime order}\label{sec:Comp}
{\Cref{thm:main} \eqref{item:Blanchfield}}
expresses the $\mathbb{Z}$--slice genus in terms of the existence of certain presentation matrices for the Blanchfield pairing.
In spite of the algebraic nature of this characterization,
no algorithm to compute $\gst$ of a given knot is evident from it.
One faces the same hurdle as when trying to compute the Nakanishi index (the minimal number of generators of
the Alexander module), which bounds $\gst$ from below,
namely the complexity of the ring $\Lambda = \mathbb{Z}[t^{\pm 1}]$
underlying the Blanchfield pairing. In particular, one lacks a classification
of finitely generated modules over $\Lambda$, as is available over rings
such as PIDs or Dedekind domains.

To obtain lower bounds for $\gst$, however, one may consider the Hermitian pairing induced
by $\Bl$ when taking the quotient of $\Lambda$ by a suitable ideal.
In this section, we pursue this idea for the principal ideals generated by the $n$--th cyclotomic
polynomials~$\Phi_n$ for $n$ prime.
There are two reasons for this choice.
Firstly, taking the quotient by $\Phi_n$ yields the sesquilinear linking pairing on the first homology
group of the $n$--fold cyclic branched covering of $S^3$ along the knot, which is of geometric interest.
The details of this relationship are well known (see e.g.~\cite{MR1324523}), and will be explained in \cref{subsec:bisesqui}.

Secondly, the algebraic situation is particularly simple.
The conjugation on $\Lambda / (\Phi_n)$ is just complex conjugation,
the ring $\Lambda / (\Phi_n)$ is Dedekind, and isometry classes
of Hermitian linking pairings can be fully classified.
For $n = 2$, one has $\Phi_2 = t + 1$ and $\Lambda / (\Phi_2) \cong \Z$,
hence the Hermitian pairings are in fact simply the symmetric ones.
Isometry classes of symmetric integral pairings on finite abelian groups $A$ with odd order have been classified by
Wall~\cite{MR0156890}.
We obtain an obstruction for $\gst(K) \leq 1$ from this pairing in \cref{subsec:double}
and use it to complete the calculation
of the $\mathbb{Z}$--slice genus
for knots  in the table of prime knots with crossing number~12 and fewer (see \Cref{subsec:comp}).

One might expect the cases $n \geq 3$ to yield similarly efficient lower bounds.
That appears to be not so. Indeed, in \cref{subsec:higher}, we prove that for an odd prime $n$, the isometry class of the linking pairing
of the $n$--fold branched covering is already determined by the action of the deck transformation group.

This section is inspired by Borodzik and Friedl's pursuit of the analogous lower bounds for the algebraic unknotting number \cite{BorodzikFriedl_15_TheUnknottingnumberAndClassInv1}.
They show that the problem of finding minimal presentation matrices after quotienting by $\Phi_n$ is essentially a finite
problem, and can thus be solved by a computer using brute force.
In contrast, the lower bounds we obtain from the double branched covering in \cref{subsec:double} can for a given knot be checked
by hand. Moreover, our \cref{thm:allisometric} partially answers the question implicitly asked in
\cite{BorodzikFriedl_15_TheUnknottingnumberAndClassInv1}:
why $n$--fold branched coverings for prime odd $n$ do not supply efficient lower bounds.

\subsection{Bilinear and sesquilinear linking pairings of finite branched coverings}
\label{subsec:bisesqui}
\renewcommand{\cyc}{\infty}%
Fix a prime power $n\geq 2$ and let $\Sigma_n(K)$ denote the $n$--fold cyclic branched covering of $S^3$ along $K$.
Then $H_1(\Sigma_n(K);\Z)$ is an abelian group of finite order $d$,
carrying a bilinear symmetric non-degenerate \emph{linking pairing}
\[\lke\colon H_1(\Sigma_n(K);\Z)\times H_1(\Sigma_n(K);\Z)\to\Q/\Z,\]
which can (similarly to the Blanchfield pairing)
be defined as $\lke(x,y) = (\Psi_n(x))(y)$ for $\Psi_n$ the composition of the maps
\begin{multline*}
H_1(\Sigma_n(K);\Z) \to H^2(\Sigma_n(K);\Z) \to H^1(\Sigma_n(K);\Q/\Z)\\
\to \Hom_{\Z}(H_1(\Sigma_n(K);\Z), \Q/\Z),
\end{multline*}
where the first map is the inverse of Poincar\'e duality, the second map is the inverse of the Bockstein map (the connecting homomorphism in the long exact sequence of cohomology induced by the short exact sequence of coefficient $0\to\Z\to\Q\to\Q/\Z\to 0$), and the third map is the Kronecker evaluation map.
Here, we are using that $\Sigma_n(K)$ is a rational homology sphere, which follows from $n$ being a prime power.
For non-prime powers $n$,
$\Sigma_n(K)$ need not be a rational homology sphere,
but $\lke$ may be defined on the torsion part of the first homology.
But since the results we prove in the following sections apply only to prime $n$ anyway,
we do not follow that more general setup.

The deck transformation group of the branched covering $\Sigma_n(K) \to S^3$ is cyclic of order $n$.
Identifying it with the multiplicative group of units $\langle [t]\rangle\subset \Lambda_n\coloneqq \Lambda/(t^n-1)$
endows $H_1(\Sigma_n(K);\Z)$ with the structure of a $\Lambda_n$--module.
Clearly, $\lke$ is \emph{equivariant} with respect to the action of the unit group of $\Lambda_n$, i.e.\ $\lke(tx,ty)=\lke(x,y)$ for all $x,y\in H_1(\Sigma_n(K); \Z)$.

One may now define another linking pairing $\lkh'$ on the $\Lambda_n$-module $H_1(\Sigma_n(K);\Z)$ that is sesquilinear (anti-linear in the first variable), Hermitian (with respect to
the conjugation on $\Lambda_n$ given by $t^k\mapsto t^{-k}$) and also non-degenerate:
\begin{align}\tag{$\dagger$}\label{deflambda}
\lkh'\colon & H_1(\Sigma_n(K);\Z)\times H_1(\Sigma_n(K);\Z)\to %
Q(\Lambda_n) /\Lambda_n,\\\notag
&  (x,y)\mapsto \sum_{k=1}^n [t]^k\lke(t^kx,y).
\end{align}

Now, we will need that $Q(\Lambda_n)$ and $\Q[t^{\pm 1}]/(t^n - 1)$ are isomorphic,
which is implied by the following more general statement.
\begin{lemma}
For all monic polynomials $f\in \Lambda$, the function
$\Q[t^{\pm 1}]/(f) \to Q(\Lambda/(f))$ given by $[g/k] \mapsto [g]/[k]$
for $g \in \Lambda$ and $k\in \Z\setminus\{0\}$ is a ring isomorphism.
\end{lemma}
\begin{proof}
The rings $Q(\Lambda/(f))$ and $\Q[t^{\pm 1}]/(f)$ are the localizations of $\Lambda/(f)$ with respect to the multiplicative sets $S$ and $T$, respectively, where $S$ is the set of non-zero divisors, and $T = \Z\setminus\{0\}$.

We claim that $[g]$ is a zero divisor of $\Lambda / (f)$ if and only if $f$ and $g$ have a non-unital common divisor $h$.
For the `if' direction, note that $g \cdot (f / h)$ is divisible by $f$ in $\Lambda$, and so $[g]\cdot [f/h] = 0 \in \Lambda/(f)$.
Since $f/h$ is not divisible by $f$, $[f/h] \neq 0 \in \Lambda/(f)$, and so $[g]$ is a zero divisor.
For the `only if' direction, recall that $[g]$ being a zero divisor means that there exists a non-zero $[h] \in \Lambda / (f)$ such that $[g][h] = 0$. Equivalently, $f$ divides $g\cdot h$ in $\Lambda$. Since $\Lambda$ is a UFD, and $f$ does not divide $h$ since $[h] \neq 0$, it follows that $f$ and $g$ have a non-unital common divisor.

Since $f$ is monic, the gcd of its coefficients is trivial, and so there is no $k\in T$ having a non-unital common divisor with $f$. It follows that $T\subset S$, and so $Q(\Lambda/(f))$ is in fact
a localization of $\Q[t^{\pm 1}]/(f)$. This yields an injective homomorphism $\Q[t^{\pm 1}]/(f) \to Q(\Lambda(f))$,
which equals the function given in the statement of the lemma.
We claim it is in fact an isomorphism. We just need to check surjectivity, and for that it is sufficient
to prove that every non-zero divisor $[g]\in \Lambda/(f)$ has a multiplicative inverse in $\Q[t^{\pm 1}]/(f)$.
As shown above, $f$ and $g$ have no non-unital common divisor over $\Lambda$.
Gauss's Lemma implies that $f$ and $g$ do not have a non-unital common divisor over $\Q[t^{\pm 1}]$, either.
Since $\Q[t^{\pm 1}]$ is a PID, by B{\'e}zout's lemma
there are $u, v\in \Q[t^{\pm 1}]$ with $uf + vg = 1$. Now $[v] \in \Q[t^{\pm 1}]/(f)$ is the desired multiplicative inverse.
\end{proof}

For the rest of of the text we fix an identification isomorphism $Q(\Lambda_n)\cong\Q[t^{\pm 1}]/(t^n - 1)$ as provided by the above lemma.

Let $\theta\colon Q(\Lambda_n) \cong \Q[t^{\pm 1}]/(t^n - 1) \to \Q$ be the $\Q$--linear map sending
$t^k$ to $1$ if $n|k$ and to $0$ otherwise.
Then $\lke$ may be recovered from $\lkh'$ simply by
\begin{equation}\tag{$\ddagger$}\label{recoverlk}
\lke(x,y)  = \theta(\lkh'(x,y)),
\end{equation}

In fact, we have the following (a variation of \cite[Prop.~1.3]{MR1324523}).
\begin{prop}\label{prop:bisesqui}
On a fixed 
torsion $\Lambda_n$--module,
the bilinear symmetric equivariant pairings
and sesquilinear Hermitian pairings are in one-to-one correspondence
via \cref{deflambda} and \cref{recoverlk}.
This correspondence preserves isometry type and non-degeneracy.
\end{prop}\begin{proof} For a fixed torsion $\Lambda_n$--module, \cref{deflambda} defines a map from sesquilinear Hermitian pairings to bilinear symmetric equivariant pairings that has the map from bilinear symmetric equivariant pairings to sesquilinear Hermitian pairings defined by \cref{recoverlk} as its inverse. Indeed, if $\ell_h$ is a sesquilinear Hermitian pairing and $\ell_e(x,y)\coloneqq \theta(\ell_h(x,y))$,
then we have (using $\sum_{k=1}^n [t^k]\theta([t^{-k}]g)=g$ for $g\in \Q[t^{\pm 1}]/(t^n - 1)\cong Q(\Lambda_n)$)
\[\sum_{k=1}^n[t]^k\ell_e(t^kx,y)=\sum_{k=1}^n[t]^k\theta(\ell_h(t^kx,y))=\sum_{k=1}^n[t]^k\theta([t^{-k}]\ell_h(x,y))=\ell_h(x,y).\]
If instead, $\ell_e$ is a symmetric equivariant pairing and \[\ell_h(x,y)\coloneqq\sum_{k=1}^n[t]^k\ell_e(t^kx,y),\] then we have $\theta(\ell_h(x,y))=\theta\left(\sum_{k=1}^n[t]^k\ell_e(t^kx,y)\right)=\ell_e(t^nx,y)=\ell_e(x,y)$.
\end{proof}

Next, let us change the ground ring from $\Lambda_n$ to its quotient $\Lambda/\rho_n$,
where $\rho_n = (t^n - 1)/(t - 1) = 1 + t + \cdots + t^{n-1}$.
For this, consider the covering map $\pi\colon{S^3\setminus K}^{\cyc}\to S^3\setminus K^n$ and the inclusion $\iota\colon S^3\setminus K^n \to \Sigma_n(K)$, where ${S^3\setminus K}^{n}$ denotes the $n$--fold cyclic covering of $S^3\setminus K$. The deck transformation group endows $H_1({S^3\setminus K}^{n};\Z)$ with a $\Lambda_n$--module structure and as a $\Lambda_n$--module $H_1({S^3\setminus K}^{n};\Z)$ is canonically isomorphic to the twisted homology  ${H_1(S^3\setminus K; \Lambda_n)}$
(twisted with respect to the abelianization $\pi_1(S^3\setminus K)\to \Z$ composed with $\Z\to \Z/n\Z$).
Here,  we identify the group ring $\Z[\Z/n\Z]$ with $\Lambda_n$.

The \emph{transfer homomorphism} $\tau\colon C_*(S^3\setminus K; \Z) \to C_*(S^3\setminus K^n; \Z)$
is defined by sending a singular simplex $\sigma$ of $S^3\setminus K$ to the sum of its $n$ lifts to $S^3\setminus K^n$.
In other words, if $\widetilde{\sigma}$ is an arbitrary lift of $\sigma$, then $\tau(\sigma) = \rho_n\cdot\widetilde{\sigma}$.
Thus it follows that $\im \tau \subset \rho_n C_*(S^3\setminus K^n; \Z)$.
Similarly, if $\widetilde{\sigma}$ is any singular simplex of $S^3\setminus K^n$, and $\sigma$ its image under the covering map, then $\rho_n\cdot \widetilde{\sigma} = \tau(\sigma)$. Hence we have 
$\im \tau = \rho_n {C_*(S^3\setminus K^n; \Z)}$.
The transfer homomorphism is a chain map and thus induces a map $\tau_*\colon 
{H_*(S^3\setminus K; \Z)} \to {H_*(S^3\setminus K^n ; \Z)}$.
The image of $\tau_*$ equals ${\rho_n H_*({S^3\setminus K}^n ; \Z)}$, since this is already true on the chain level.
Because ${H_1(S^3\setminus K; \Z)}$ 
is generated by the class of a meridian $m$ of the boundary torus of $S^3\setminus K$,
it follows that ${\rho_n H_1({S^3\setminus K}^n ; \Z)}$ is generated by $\tau_*([m])$,
which equals the class of a meridian of the boundary torus of $S^3\setminus K^n$.
But since this class is killed by $\iota_*$ we find that $\rho_n$ annihilates~$H_1(\Sigma_n; \Z)$.
So $H_1(\Sigma_n; \Z)$ has the structure of a $\Lambda/\rho_n$--module, 
on which $\lkh'$ induces a sesquilinear Hermitian non-degenerate pairing
\begin{align*}
\lkh\colon & H_1(\Sigma_n(K);\Z)\times H_1(\Sigma_n(K);\Z)\to %
Q(\Lambda/\rho_n) / (\Lambda / \rho_n),\\\notag
&  \lkh(x,y) = p(\lkh'(x,y)) = \sum_{k=1}^n [t]^k\lke(t^kx,y),
\end{align*}
for $p$ the canonical projection
\[
Q(\Lambda_n)/\Lambda_n \to
Q(\Lambda / \rho_n) / (\Lambda / \rho_n).
\]
By \cref{cor:pfold} the isometry classes of $\lkh'$ and $\lkh$
determine one another.

In summary, we have seen three variations of linking pairings on the $n$--fold branched covering,
all of which contain the same information and can thus be used interchangeably.
In the following, we will stick to $\lkh$ and refer to it as
the \emph{sesquilinear linking pairing of $\Sigma_n$}.
Our reason to favor $\lkh$ is its close relationship to the Blanchfield form:
in a certain sense, which will be made more precise in \cref{sec:ApP:BlandLk},
$\lkh$ is induced by $\Bl$. In particular, we have the following statement,
which we will also prove in \cref{sec:ApP:BlandLk}.

\begin{lemma}\label{prop:inducedpresmatrix}
If the Blanchfield pairing of a knot $K$ is presented by a Hermitian matrix $A(t)$ over $\Z[t^{\pm 1}]$,
then, for $n$ a prime power, the sesquilinear linking pairing $\lkh$ of the $n$--fold branched covering $\Sigma_n$ of $S^3$ along $K$
is presented by the matrix obtained from $A(t)$ by applying the quotient ring homomorphism $\Z[t^{\pm 1}] \to  \Z[t^{\pm 1}]/(1 + t + \ldots + t^{n-1})$ to each entry.
\end{lemma}

\subsection{The double branched covering}\label{subsec:double}
For $n = 2$, we have $\rho_2 = \Phi_2 = 1 + t$, and $\Lambda / (\rho_2)$ is isomorphic to $\Z$
via $t\mapsto -1$. Under this identification, the sesquilinear pairing $\lkh$ equals \emph{twice}
the bilinear pairing $\lke$, i.e.\ $\lkh(x,y) = 2\lke(x,y)$ for all $x, y\in H_1(\Sigma_2;\Z)$, since
\begin{align*}
\lkh(x,y)& =\sum_{k=1}^2 t^k\lke(t^kx,y)=t\lke(tx,y)+t^2\lke(t^2x,y)\\
  & =(-1)\lke(-x,y)+(-1)^2\lke(x,y) = 2\lke(x,y).
\end{align*}
Although $\lke$ is the pairing that is usually considered in the literature, we prefer to stick with $\lkh$ for consistency's sake.
We ask the reader to keep this subtlety in mind when working with results from this section (see also \cref{rmk:ko2} below).
Let us start by proving the following.

\begin{prop}\label{prop:doublepresmat}
There is a symmetric integral presentation matrix $C$ of size $2\gst(K)$ for the pairing $\lkh$ on $H_1(\Sigma_2;\Z)$
with $\det C \equiv (-1)^{\gst(K)} \pmod{4}$.
\end{prop}
\begin{proof}
By \Cref{thm:main} \eqref{item:Blanchfield}, there exists a presentation matrix
$A(t)$ of the Blanchfield pairing of size $2\gst(K)$ such that $\sigma(A(1)) = 0$.
When taking the quotient by~$\Phi_2$, i.e.~setting $t=-1$,
$A(t)$ descends to a presentation matrix $C = A(-1)$ of $\lkh$ 
(see \cref{prop:inducedpresmatrix} or \cite[Lemma~3.3]{BorodzikFriedl_15_TheUnknottingnumberAndClassInv1}).
It remains to check the condition on $\det C$.

We have $\det A(t) = \pm\Delta_K(t)$,
where the Alexander polynomial $\Delta_K(t)$ is normalized such that
$\Delta_K(t) = \overline{\Delta_K(t)}$ and $\Delta_K(1) = 1$.
It follows that $\det A(1) = \pm 1$.
So $A(1)$ is a symmetric integral unimodular matrix with signature $0$ (in particular, indefinite).
We now make use of the classification of such forms; see \cref{thm:classificationindefinite}.
To sidestep the case of even matrices, let us consider the form $A(1) \oplus (1)$, which is clearly odd,
and thus congruent over $\Z$ to a diagonal matrix 
with $\gst(K) + 1$ many diagonal entries equal to 1, and $\gst(K)$ many diagonal entries equal to $-1$.
So $\det(A(1)) = \det(A(1) \oplus (1)) = (-1)^{\gst(K)}$,
and thus 
$\det A(t) = (-1)^{\gst(K)}\Delta_K(t)$.

Because $\Delta(1) = 1$, there exists $f(t)\in\Lambda$ such that $\Delta(t) = (t-1)f(t) + 1$.
Since $\Delta(t^{-1}) = \Delta(t)$, we have
$(t^{-1} - 1)f(t^{-1}) + 1 = (t-1)f(t) + 1 \Rightarrow f(t^{-1}) = -tf(t) \Rightarrow f(1) = -f(1) \Rightarrow f(1) = 0$.
So there exists $g(t)\in\Lambda$ such that $f(t) = (t-1)g(t) + 1$, and hence $\Delta(t) = (t-1)^2g(t) + 1$.
Thus we find
\begin{align*}
[\det A(t)] & = [(-1)^{\gst(K)}] \in \Lambda / (t-1)^2\\
\Rightarrow\quad
[\det C] & = [(-1)^{\gst(K)}] \in \Lambda / (\Phi_2, (t-1)^2)\\
\Rightarrow\quad
\det C & \equiv (-1)^{\gst(K)} \pmod{4}.\pushQED{\qed}\qedhere\popQED
\end{align*}%
\renewcommand{\qedsymbol}{}
\end{proof}%
\vspace{-\baselineskip}
\begin{rmk}\label{rmk:ko2}
In the previous proof,
provided that $\gst(K) = g(K)$,
the presentation matrix $A(t)$ of the Blanchfield pairing 
may be obtained from a minimal Seifert matrix $M$ of $K$ using Ko's formula (compare \cref{rmk:ko}).
Note that $M + M^{\top}$ is a presentation matrix for the pairing $\lke$,
whereas $A(-1)$ is a presentation matrix for the pairing $\lkh$.
\end{rmk}

In light of this lower bound for $\gst$ it would be useful---and of somewhat independent interest---to determine for every isometry class
of symmetric pairings $\ell$ on finite abelian groups $A$ of odd order (as classified by Wall~\cite{MR0156890})
the minimal size of a presentation matrix with prescribed determinant modulo 4.
We make a start by giving a necessary and sufficient condition for such a symmetric pairing to admit a $2\times 2$ matrix
in \cref{prop:jacobi} below.
Let us briefly state Wall's classification.
\begin{thm}[\cite{MR0156890}]\label{thm:wall}
For $p$ an odd prime and $k\geq 1$, there are exactly two isometry types of non-degenerate symmetric bilinear pairings $\ell\colon \Z/p^k\Z \times \Z/p^k\Z \to \Q/\Z$. Namely, writing $\ell(1,1) = a / p^k$,
all pairings $\ell$ for which $a$ is a quadratic residue modulo $p$ form an isometry type, which we denote by $A_{p^k}$,
and all pairings $\ell$ for which $a$ is not a square residue modulo $p$ form another isometry type, which we denote by $B_{p^k}$.

The isometry types of non-degenerate symmetric bilinear pairings on finite abelian groups of odd order
constitute a commutative monoid, with the orthogonal sum $\oplus$ as operation, and the trivial group as neutral element.
This monoid may be presented by the generators $A_{p^k}, B_{p^k}$ and the relations $A_{p^k}\oplus A_{p^k} = B_{p^k}\oplus B_{p^k}$, where $p$ ranges over the odd primes and $k$ over the positive integers.
\end{thm}
Wall's classification implies in particular that $(A,\ell)$ decomposes as orthogonal sum of cyclic groups with prime power orders.
By combining summands with coprime orders, one may obtain a decomposition of $(A,\ell)$ as orthogonal sum of pairings on cyclic groups of orders $q_1, \ldots, q_n$ with $q_i | q_{i + 1}$, where the $q_i$ are positive odd integers (but not necessarily prime powers).
It turns out that this decomposition is more practical to phrase \cref{prop:jacobi} than the decomposition given by Wall's classification.

For each $i$, let $g_i$ be a generator of the $i$--th summand, which has order $q_i$. Then $\ell(g_i, g_i) \in \Q/\Z$
can be written as $a_i / q_i$ for some $a_i\in\Z$, which is coprime to $q_i$ since $\ell$ is non-degenerate.
Note that while the $q_i$ are uniquely determined by the isomorphism type of $A$ (this follows from the classification of finitely generated abelian groups), the $a_i$ are not uniquely determined by the isometry type of $\ell$.
Nevertheless, the fraction $a_i/q_i \in \Q/\Z$ determines the isometry type of $\ell$ restricted to the $i$--th summand,
and so the tuple
\[
\left(\frac{a_1}{q_1}, \ldots, \frac{a_n}{q_n}\right)
\]
determines the isometry class of $\ell$.
Note that $n$ is a lower bound for the size of a presentation matrix of $\ell$.
As a warm-up let us solve the case $n = 1$.
\begin{prop}
A non-degenerate symmetric bilinear pairing $\ell\colon A\times A \to \Q/\Z$ on a finite abelian group $A$ of odd order $q > 1$
admits a $1\times 1$ presentation matrix
if and only if
$A$ is cyclically generated by some $g \in A$, such that $\ell(g,g) = a/q$
with $a$ or $-a$ being a quadratic residue modulo $q$.
Moreover, if this condition holds for one choice of $a$,
then it holds for all.
\end{prop}
\begin{proof}
Let us first prove the `only if' direction.
If the $1\times 1$ matrix $(y)$ presents a pairing isometric to $\ell$, then $\Z / y$ must be isomorphic
to $\Z / q$, so $y = \pm q$. Now, for $g = [1] \in \Z/q$, 
the matrix $(q)$ presents the pairing with tuple $(1/q)$, and 
the matrix $(-q)$ presents the pairing with tuple $(-1/q)$. In the first case, $a = 1$
is a quadratic residue, and in the second case, $-a = 1$ is a quadratic residue.

For the `if' direction, assume $a \equiv \pm x^2 \pmod{q}$. Since $a$ is coprime to $q$, so is $x$.
So there exists an integer $y$ coprime to $q$ such that $xy \equiv 1 \pmod{q}$. Sending $g$ to the generator
$yg$ gives an isometry to a pairing with tuple $(\pm 1 / q)$,
since 
\[
\ell(yg,yg) = y^2\ell(g,g) = y^2a/q = \pm y^2x^2/q = \pm 1/q\in \Q/\Z.
\]
This pairing is presented by the $1\times 1$ matrix $(\pm q)$.

To check that the condition that one of $a$ and $-a$ is a quadratic residue does not depend on the choice of $g$
and $a$, observe that $a\in \Z$ is well-defined up to adding multiples of $q$, and multiplying with squares.
Neither of those operations changes whether $\pm a$ is a quadratic residue.
\end{proof}

To analyze the existence of $2\times 2$ presentation matrices,
we will need the \emph{Jacobi symbol} $\legendre{x}{y}\in\{\pm 1,0\}$,
which may be defined as follows for all $x,y \in\Z$ with $y$ odd and positive:
\begin{itemize}
\item If $\gcd(x,y) \neq 1$, then $\legendre{x}{y} = 0$.
\item If $y$ is prime, then $\legendre{x}{y}$ equals the Legendre symbol, i.e.\ it is $1$ if $x$ is a quadratic residue modulo $y$, and $-1$ otherwise.
\item If $y = p_1 \cdot\ldots\cdot p_k$ is a product of prime numbers, then $\legendre{x}{y}$ is the product of the $\legendre{x}{p_i}$.
\end{itemize}
The Jacobi symbol has a number of useful properties.
For instance, it is multiplicative in both $x$ and $y$.
Moreover, if $x,y$ are positive and odd, it satisfies quadratic reciprocity:
\[
\legendre{x}{y} = (-1)^{\frac{x-1}{2}\cdot\frac{y-1}{2}}\cdot\legendre{y}{x}.
\]
Finally, we have $\legendre{-1}{y} = (-1)^{\frac{y - 1}{2}}$. Note that for non-prime $y$, $\legendre{x}{y} = 1$
is a necessary, but not a sufficient condition for $x$ being a quadratic residue modulo~$y$---this is the reason
that Dirichlet's prime number theorem is needed in the proof of the following (which consists, aside from that,
of elementary manipulations).
Recall from \cref{thm:classificationindefinite} that a matrix with integer entries
is called odd if one of its diagonal entries is odd, and even otherwise.
\propjacobi*
\begin{proof}
As discussed above, the existence of such an orthogonal decomposition follows from \cref{thm:wall}. Furthermore, $a_i$ and $q_i$ are coprime since $\ell$ is non-degenerate.
The proof now proceeds by showing that (A) and (B) are both equivalent to the following intermediate statement:
\begin{enumerate}[leftmargin=2em]
\item[(C)]
There exist integers $\alpha, \beta, \gamma, \lambda_1, \lambda_2$ satisfying the following conditions:
\begin{enumerate}[leftmargin=2em]
\item[(C1)] $\alpha$ is odd,
\item[(C2)] $\alpha\gamma - \beta^2 = (-1)^{(q_1q_2 - u)/2} q_2/q_1$,
\item[(C3)] $\lambda_1^2 \alpha \equiv a_1  \pmod{q_1}$,
\item[(C4)] $\lambda_2^2 \alpha \equiv (-1)^{(q_1q_2 - u)/2} a_2 \pmod{q_2}$.\\
\end{enumerate}
\end{enumerate}
More precisely, we will show that if there exists a choice of $a_1, a_2$ such that (C) holds, then (A) follows;
and we will show that if (A) holds, then (C) follows for all possible choice of $a_1, a_2$.
Hence (B) is indeed independent from the choice of $a_1, a_2$, as claimed.
\medskip

\noindent\emph{`(C) $\Rightarrow$ (A)':\quad}
We claim that the matrix
\[
M = q_1 \cdot \begin{pmatrix} \alpha & \beta \\ \beta & \gamma \end{pmatrix}
\]
has the desired properties. Indeed, by (C1) and (C2), this is an odd symmetric $2\times 2$ integer matrix
with determinant $q_1^2(\alpha\gamma - \beta^2) = (-1)^{(q_1q_2 - u)/2} q_1q_2$.
Going through the four possible cases that $q_1q_2 \equiv \pm 1\pmod{4}$ and $u\equiv\pm 1\pmod{4}$,
one sees that $\det M \equiv u\pmod{4}$, as desired.
It remains to check that $M$ presents $\ell$. One computes
\[
M^{-1} =
\frac{(-1)^{\frac{q_1q_2 - u}{2}}}{q_2}\begin{pmatrix} \gamma & -\beta \\ -\beta & \alpha \end{pmatrix}.
\]
Let $v_1 = \lambda_1(\alpha, \beta)^\top$ and $v_2 = \lambda_2(0,1)^{\top}$. Then one computes
\begin{align*}
M^{-1}v_1 & = (\lambda_1,0)^{\top} / q_1 \in \Q^2 \\
M^{-1}v_2 & = (-1)^{(q_1q_2 - u)/2} \lambda_2 (-\beta, \alpha)^{\top} / q_2 \in \Q^2,
\end{align*}
and
\begin{align*}
\ell_M([v_1], [v_2]) & = v_1^{\top} M^{-1} v_2 = 0\in \Q/\Z, \\
\ell_M([v_1], [v_1]) & = v_1^{\top} M^{-1} v_1 = \lambda_1^2\alpha / q_1 \in \Q/\Z, \\
\ell_M([v_2], [v_2]) & = v_2^{\top} M^{-2} v_2 = (-1)^{(q_1q_2 - u)/2}\lambda_2^2\alpha / q_2 \in \Q/\Z.
\end{align*}
Now consider $w = k_1v_1 + k_2v_2$ for $k_1, k_2\in\Z$.
We have $\ell_M([w], [v_1]) = k_1\lambda_1^2\alpha/q_1$
and $\ell_M([w], [v_2]) = k_2(-1)^{(q_1q_2 - u)/2}\lambda_2^2\alpha / q_2$.
Since $a_i$ is coprime with $q_i$, (C3) and (C4) respectively imply that
$\lambda_i^2\alpha$ is also coprime with $q_i$ for $i = 1,2$.
Thus if $[w] = 0 \in \coker M$, then $\ell_M([w], [v_i]) = 0$, and thus $q_i | k_i$ for $i = 1,2$.
It follows that $[v_1], [v_2]$ span a subgroup isomorphic to $\Z/q_1 \oplus \Z/q_2$ of $\coker M$.
But since $|\det M| = q_1q_2$, this subgroup is equal to the whole group $\coker M$.

Above, we have computed that $[v_1]$ and $[v_2]$ are orthogonal with respect to $\ell_M$,
and by (C3) and (C4), $\ell_M([v_i], [v_i]) = a_i/q_i \in \Q/\Z$.
Hence $\ell_M$ is isometric to $\ell$, as desired.
\medskip

\noindent\emph{`(A) $\Rightarrow$ (C)`:\quad}
Since $\coker M \cong \Z/q_1 \oplus \Z/q_2$ with $q_1 | q_2$, row and column operations (which preserve the gcd of entries and $\pm\det$) turn $M$ into its Smith normal form: the diagonal matrix with entries $q_1$ and $q_2$. Hence,
the gcd of the entries of $M$ is $q_1$
and $\det M = \pm q_1q_2$.
The condition $\det M \equiv u \pmod{4}$ implies $\det M = (-1)^{(q_1q_2 - u)/2} q_1q_2$.

Now pick a non-trivial vector $x' \in \Z^2$ with $[x'] = g_2 \in \coker M = A$.
Let $\lambda_2 = \gcd(x')$ and $x = x' / \lambda_2$.
Then, $x$ may be extended to a basis $(y, x)$ of $\Z^2$, i.e.\ $T = (y\mid x) \in \text{GL}_2(\mathbb{Z})$.
Let $\alpha, \beta, \gamma \in \Z$ such that
\[
N \coloneqq T^{-1} M (T^{-1})^{\top} = q_1\cdot \begin{pmatrix}\alpha & \beta \\ \beta & \gamma\end{pmatrix}.
\]
Here, we used that every entry of $M$, and thus also of $T^{-1} M (T^{-1})^{\top}$, is divisible by $q_1$.
We claim that the integers $\alpha,\beta,\gamma,\lambda_1$ satisfy (C2) and (C4).
Indeed, (C2) follows from $q_1^2(\alpha\gamma-\beta^2)\det N = \det M = (-1)^{(q_1q_2 - u)/2} q_1q_2$.
Next, one computes
\[
N^{-1} =
\frac{(-1)^{\frac{q_1q_2 - u}{2}}}{q_2}\begin{pmatrix} \gamma & -\beta \\ -\beta & \alpha \end{pmatrix}
\]
and
\begin{align*}
\frac{a_2}{q_2} & = \ell(g_2, g_2) = x'^{\top} M^{-1} x' = \lambda_2^2 x^{\top} (T^{-1})^{\top} N^{-1} T^{-1} x
= \lambda_2^2\begin{pmatrix} 0 \\ 1 \end{pmatrix}^{\top} N^{-1} \begin{pmatrix} 0 \\ 1 \end{pmatrix}\\
 & = \frac{(-1)^{\frac{q_1q_2 - u}{2}}\lambda_2^2\alpha}{q_2} \in \Q/\Z
\quad\Rightarrow\quad \lambda_2^2\alpha  \equiv (-1)^{\frac{q_1q_2 - u}{2}} a_2\pmod{q_2},
\end{align*}
so (C4) is satisfied.
Now, one checks that
\[
v = \begin{pmatrix} \alpha \\ \beta \end{pmatrix} \quad\Rightarrow\quad
N^{-1}v  = \begin{pmatrix} 1 / q_1 \\ 0 \end{pmatrix},
\]
which implies that $[v]$ is of order $q_1$ in $\coker N$ and orthogonal to $[(0,1)^{\top}]$ with respect to $N^{-1}$.
So, $[Tv]$ is of order $q_1$ in $\coker M$ and orthogonal to $[x]$ with respect to $M^{-1}$.
Therefore, $[Tv]$ and $g_1$ generate the same subgroup, and thus $\lambda_1 [Tv] = g_1$
for some $\lambda_1 \in \Z$, which implies (C3) by a calculation similar to the one for (C4) above.

So if $\alpha$ is odd, (C1) is satisfied and we are done.
If $\alpha$ is even, then $\gamma$ must be odd because $N$ is an odd matrix.
Replacing $\alpha$ and $\beta$ by $\alpha + 2q_2\beta + q_2^2\gamma$ and $\beta + q_2\gamma$,
respectively, preserves (C2)--(C4), while also satisfying (C1).
\medskip

\noindent\emph{`(C) $\Rightarrow$ (B)`:\quad}
Since $q_1|q_2$, (C4) implies that
$\lambda_2^2 \alpha \equiv (-1)^{(q_1q_2 - u)/2} a_2$ also holds mod $q_1$.
Multiplying this with the congruence mod $q_1$ given in (C3) yields
\[
(\lambda_1\lambda_2\alpha)^2 \equiv (-1)^{(q_1q_2 - u)/2} a_1a_2 \pmod{q_1};
\]
hence, the right-hand side is a quadratic residue mod $q_1$. This proves condition~(B1).

To prove (B2), we assume $u = -1$ and $q_1q_2 \equiv 1 \pmod{4}$ and aim to prove $\legendre{a_2}{q_2/q_1} = 1$.
By (C2), we have $\alpha \gamma - \beta^2 = -q_2/q_1$, so $q_2/q_1$ is a square modulo $\alpha$.
Since $a_2$ and $q_2$ are coprime, (C4) implies that $\alpha$ and $q_2$ are coprime,
and thus so are $|\alpha|$ and $q_2/q_1$.
It follows that $\legendre{q_2/q_1}{|\alpha|} = 1$. This concludes the proof since
\begin{align*}
1 = \legendre{q_2/q_1}{|\alpha|} & = \legendre{|\alpha|}{q_2/q_1} \qquad\text{(by quadratic reciprocity and} \\[-3ex]
 & \phantom{= \legendre{|\alpha|}{q_2/q_1}} \qquad \text{ \ $q_2/q_1 \equiv 1\pmod{4}$)} \\
 &  = \legendre{\pm a_2}{q_2/q_1} \qquad\text{(by (C4))} \\
 &  = \legendre{a_2}{q_2/q_1} \qquad\text{(because $q_2/q_1 \equiv 1\pmod{4}$)}.
\end{align*}

\noindent\emph{`(B) $\Rightarrow$ (C)`:\quad}
Dirichlet's prime number theorem states that for any positive integers $x, y$,
there exists a positive prime number equivalent to $x$ modulo~$y$.
By the Chinese remainder theorem, it follows that for positive integers $x_1, x_2, y_1, y_2$
with $y_1$ and $y_2$ coprime,
there exists a positive prime number equivalent to $x_i$ modulo~$y_i$ for~$i = 1,2$.
So in our situation, using that $y_1 = q_2$ and $y_2 = 4$ are coprime,
there exists a positive prime number $p$ satisfying
$p \equiv \sigma_1\cdot (-1)^{(q_1q_2 - u)/2}\cdot a_2 \pmod{q_2}$
and $p \equiv \sigma_2 \pmod{4}$
for any given $\sigma_1, \sigma_2 \in \{-1,1\}$.
Set $\alpha = \sigma_1\cdot p$. By definition, it satisfies (C1) and (C4) with $\lambda_2 = 1$.
By (B1), there exists $\lambda_1\in\Z$ such that
\[
\lambda_1^2 \equiv (-1)^{(q_1q_2 - u)/2} a_1a_2 \pmod{q_1}.
\]
Since $(-1)^{(q_1q_2 - u)/2} a_2$ and $q_1$ are coprime, it follows that
\[
\lambda_1^2 (-1)^{(q_1q_2 - u)/2} a_2 \equiv a_1 \pmod{q_1}.
\]
We have $\alpha = \sigma_1 p \equiv (-1)^{(q_1q_2 - u)/2} a_2 \pmod{q_2}$.
Since $q_1 | q_2$, this congruence also holds mod $q_1$. So (C3) follows.

Finally, note that the existence of some $\beta,\gamma\in\Z$
solving (C2) is equivalent to $- (-1)^{(q_1q_2 - u)/2}\cdot q_2/q_1$
being a quadratic residue modulo~$p$.
Because $p$ is prime, this is equivalent to
\[
\legendre{(-1)^{1 + \frac{q_1q_2 - u}{2}} q_2/q_1}{p} = 1.
\]
One computes
\begin{align*}
\legendre{(-1)^{1 + \frac{q_1q_2 - u}{2}} q_2/q_1}{p} & = \legendre{(-1)^{1 + \frac{q_1q_2 - u}{2}}}{p} \cdot \legendre{q_2/q_1}{p} \\
\intertext{(by multiplicativity of the Jacobi symbol)}
& = (-1)^{\frac{p-1}{2}\cdot \frac{q_1q_2-u+2}{2}} \cdot (-1)^{\frac{p - 1}{2}\cdot\frac{q_2/q_1 - 1}{2}} \cdot \legendre{p}{q_2/q_1} \\
\intertext{(by the formula for $\legendre{-1}{y}$ and by quadratic reciprocity)}
& = (-1)^{\frac{\sigma_2-1}{2}\cdot \frac{q_1q_2-u+2}{2}} \cdot (-1)^{\frac{\sigma_2 - 1}{2}\cdot\frac{q_1q_2 - 1}{2}} \cdot \legendre{p}{q_2/q_1} \\
\intertext{(since $p \equiv \sigma_2$ and $q_2/q_1 \equiv q_1q_2 \pmod{4}$)}
& = (-1)^{\frac{\sigma_2-1}{2}\cdot \frac{u + 1}{2}} \cdot \legendre{\sigma_1(-1)^{(q_1q_2 - u)/2}a_2}{q_2/q_1} \\
\intertext{(since $q_1q_2$ is odd, and $p \equiv \sigma_1 (-1)^{(q_1q_2 - u)/2} a_2 \pmod{q_2}$, and thus also mod $q_2/q_1$)}
& = (-1)^{\frac{\sigma_2-1}{2}\cdot \frac{u + 1}{2}}\cdot (-1)^{\frac{\sigma_1 -1 + q_1q_2 - u}{2}\cdot \frac{q_1q_2 - 1}{2}} \cdot \legendre{a_2}{q_2/q_1}
\intertext{(by multiplicativity of the Jacobi symbol and the formula for $\legendre{-1}{y}$).}
\end{align*}\vspace{-2\baselineskip}

Now, if $u = 1$, switching $\sigma_2$ changes the sign of the first factor,
and so $\sigma_2$ may be chosen to make the whole product $1$.
Similarly, if $q_1q_2 \equiv 3 \pmod{4}$,
then $\sigma_1$ may be chosen to make the whole product $1$.
Else we have $u = -1$ and $q_1q_2 \equiv 1 \pmod{4}$, so the first two
factors are $1$; moreover, the third factor is $1$ by (B2).
\end{proof}
Now let us spell out the obstructions for $u_a$ and $\gst$
obtained by combining the two previous propositions.
In concrete cases, the obstructions can be checked by hand.
\begin{corollary}\label{cor:doubleobstruct}
If the symmetric pairing $\ell$ as in \cref{prop:jacobi} is the linking pairing $\lkh$ of the double branched covering
of a knot $K$, then we have the following obstructions.
\begin{enumerate}[label=(\roman*)]
\item If $q_1 \equiv 1 \pmod{4}$ and $a_1a_2$ is not a quadratic residue modulo $q_1$, then $u_a(K) \geq 3$ and $\gst(K) \geq 2$.
\item If $q_1 \equiv q_2 \equiv 3 \pmod{4}$, $a_1a_2$ is not a quadratic residue modulo $q_1$ and $\legendre{a_2}{q_2/q_1} = -1$, then $u_a(K) \geq 3$ and $\gst(K) \geq 2$.
\item If $q_1 \equiv 3$ and $q_2 \equiv 1\pmod{4}$ and $a_1a_2$ is not a quadratic residue modulo~$q_1$, then $\gst(K) \geq 2$.
\item If $q_1 \equiv q_2 \equiv 3\pmod{4}$ and $-a_1a_2$ is not a quadratic residue modulo $q_1$, then $\gst(K) \geq 2$.
\item If $q_1q_2 \equiv 1\pmod{4}$ and $\legendre{a_2}{q_2/q_1} = -1$, then $\gst(K) \geq 2$.\qed
\end{enumerate}
\end{corollary}

\subsection{Calculations of the $\Z$--slice genus for small knots}\label{subsec:comp}

\begin{table}[ht]
\begin{tabular}{|p{.36\textwidth}|l|p{.51\textwidth}|}\hline
\emph{Obstruction to $\gst = 1$}                     & \multicolumn{2}{p{.52\textwidth}|}{\emph{Knots to which the obstruction applies}} \\\hline
Levine-Tristram-Signatures            &    1 &
12n749 \\\hline
Hasse-Taylor \cite{mccoylewark}  &    6 &
{12a787, 12a1142, 12n\{269, 505, 598, 602, 756\}}                                     \\\hline
$H_1(\Sigma_2(K); \Z)$ needs more than two generators &    7 &
12a554, 12a750, 12n\{553, 554, 555, 556,
 642\}                                    \\\hline
$u_a \geq 3$ \cite{BorodzikFriedl_15_TheUnknottingnumberAndClassInv1},
or equivalently \Cref{cor:doubleobstruct}(i) &    6 &
{10$_{103}$, 11n148, 12a327, 12a921,
  12a1194, 12n147 }                \\\hline
\Cref{cor:doubleobstruct}(iii)         &  26 &
$9_{48}$,
$10_{74}$,
11a\{155, 173, 352\}, 11n\{71, 75, 167\},
12a\{164, 166, 177, 244, 298, 413, 493, 503, 810, 895\},
12n\{334, 379, 460, 495, 549, 583, 869\}   \\\hline
\Cref{cor:doubleobstruct}(iv)         &  12 &
$9_{37}$, 11a135,
12a\{265, 396, 769, 873, 905\},
12n\{388, 480, 737, 813, 846\}
 \\\hline
\end{tabular}
\bigskip

\caption{$\mathbb{Z}$--slice genus calculations for small knots.}
\label{table1}
\end{table}
In \cite{mccoylewark}, $\gst$ was determined for all but 52 of the prime knots with crossing number 12 or less.
Let us summarize those calculations, and see how much further we can go using \Cref{cor:doubleobstruct}
and \cref{thm:main}.

\begin{figure}[p]
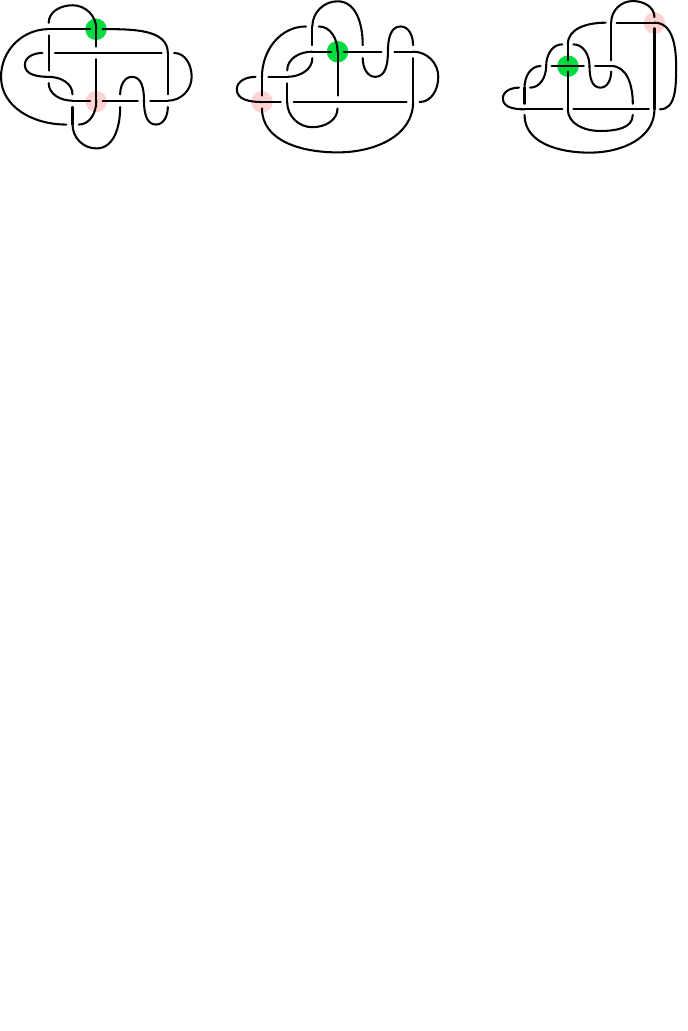
\caption{Using crossing changes, it can be shown that these 15 knots have $\gst = 1$ (see \cref{subsec:comp} for details). Dark green and light red disks mark positive and negative crossings to be changed, respectively. Diagrams drawn with PLink~\cite{plink}.}
\label{fig:unknottings}
\end{figure}
Four of those 2,977 knots have Alexander polynomial $1$ and thus $\gst = 0$.
For all others $\gst \geq 1$ holds. Upper bounds for $\gst$ are given by half of the span of the Alexander polynomial, by the algebraic unknotting number, and may be obtained
bounds from a randomized Seifert matrix algorithm \cite{BaaderFellerLewarkLiechti_15,mccoylewark}.
For 1,999 knots, one of these upper bounds is~$1$ and $\gst = 1$ follows.
For another 901 knots, one of these upper bound equals $|\sigma|/2$ and $\gst = |\sigma|/2$ follows.
For all of the remaining 73 knots, the minimum of these upper bounds is $2$, and it just remains to decide whether $\gst$ equals 1 or~2.
\Cref{table1} lists those knots for which $\gst = 2$ can be proven, and the used obstruction.
The $\Z$--slice genus was previously unknown for the 37 knots listed in the last two rows,
for which \Cref{cor:doubleobstruct} proves to be an effective tool.
There remain 15 knots.
For all of those knots we prove $\gst = 1$, using the characterization of $\gst$ as balanced algebraic unknotting number, as follows (the relevant crossing changes are shown in \cref{fig:unknottings}):
\begin{itemize}
\item
The following 13 knots may be transformed into the unknot by two crossing changes,
one of each sign: 12a1013, 12a1047, 12a1168, 12a1203, 12a1211, 12a1221, 12a1222, 12a1225, 12a1229, 12a1230, 12a1248, 12a1260, 12a1283.
\item
The knot $K = $ 12a1226 may be transformed into $J = 10_{79}$ by changing a negative crossing.
Since $J$ is amphicheiral and has algebraic unknotting number 1~\cite{knotorious},
$J$ may be transformed into a knot with Alexander polynomial 1 by changing a positive crossing.
Overall, it follows that $K$ has balanced algebraic unknotting number 1.
\item
The knot $K = $ 12a1288 may be transformed into $J = 10_{64}$ by changing a positive crossing, or by changing a negative crossing. The knot $J$ has algebraic unknotting number 1~\cite{knotorious}.
Overall, it follows that $K$ has balanced algebraic unknotting number 1.
\end{itemize}
This concludes the calculation of the algebraic genus for all prime knots with crossing number 12 or less. The values are also listed in an online table~\cite{galg-table}.
\subsection{Branched coverings of odd prime order}
\label{subsec:higher}
This section is devoted to the proof of the following.
\thmalliso*
Let us start with an example, which already contains one crucial ingredient of the proof.
\begin{example}
Let $K$ be the $(6,7)$--torus knot. %
Torus knots have Nakanishi index~$1$, in particular $N_p(K)$ is a cyclic module for all primes $p$.
For example, $N_2(K) \cong \Z[t]/(7,t+1)$, which is simply $\Z/7$ on which $t$ acts as $-1$.
The sesquilinear linking pairing $\lkh\colon N_2(K)\times N_2(K) \to \Q/\Z$ is completely characterized by 
its value $\lkh(x,x)$ on a generator $x$ of $N_2(K)$. We have $\lkh(x,x) = a/7$ with $a\in \Z$.
Now, for $K'$ the mirror image of $K$, one has $N_2(K') \cong N_2(K)$ and
$\lkh(x',x') = -a/7$ for a suitably chosen generator $x'$ of $N_2(K')$.
Suppose there exists an isometry between the linking pairings on $N_2(K)$ and $N_2(K')$. It would send $x$ to $\lambda x'$,
which would imply
\[a/7 = \lambda^2\cdot (-a/7) \in \Q/\Z \quad\Leftrightarrow\quad \lambda^2 \equiv -1 \pmod{7}.
\]
But $-1$ is not a square modulo $7$.
So, $K$ and $K'$ are not distinguished by their $N_2$--modules, which are isomorphic; but they are distinguished by their linking
pairings on those $N_2$--modules, which are not isometric.
\cref{thm:allisometric} tells us that $p = 2$ is the only prime for which this may happen.
Indeed, let us consider the same pair of knots $K, K'$ for $p = 3$.
One finds $N_3(K)\cong N_3(K')$ to be of order $7$ as modules over $\Z[t]/(t^2 + t + 1)$, i.e.\ 
$N_3(K) \cong \Z[t]/(7,t^2+t+1)$. Note that the underlying abelian group of this module is isomorphic to $\Z/7 \times \Z/7$.
Again, the linking pairing on $N_3$ is determined by the value $\lkh(x,x)\in (\Q(t)/(t^2+t+1)) / (\Z[t] / (t^2+t+1))$
for $x$ a generator of $N_3(K)$. Because $\lkh$ is Hermitian, $\overline{\lkh(x,x)} = \lkh(x,x)$, where $\overline{\,\cdot\,}$
is the conjugation that linearly extends $t\mapsto t^{-1}$. Thus $\lkh(x,x) = b/7$ for some $b\in\Z$.
Once more, $\lkh(x', x') = -b/7$ for some generator $x'$ of $N_3(K')$.
But, in contrast to the case $p = 2$, for $p = 3$ the linking pairings on the modules $N_p(K)$ and $N_p(K')$ are isometric.
An isometry is given by sending $x$ to $(3+6t)x'$. Indeed, one checks that
\begin{multline*}
\lkh((3+6t)x',(3+6t)x') = (3+6t)(3+6t^{-1})\lkh(x',x') \\
= (-1 + 18t^{-1}(t^2+t+1) + 28)\cdot(-b/7) = b/7 \\
 = \lkh(x,x)\in (\Q(t)/(t^2+t+1)) / (\Z[t] / (t^2+t+1)).
\end{multline*}
The key is that modulo $(7,1+t+t^2)$, one can write $-1 \in \Z[t]$ as $\lambda\overline{\lambda}$ for some $\lambda\in\Z[t]$.
In fact, every $y\in\Z[t]$ with $y=\overline{y}$ can be written in that way.

\end{example}

The theorem is an instance of a more general statement about Hermitian pairings on Dedekind rings,
which we formulate as the following proposition.

\begin{prop}\label{prop:allisometric}
Let $R$ be a Dedekind ring with an involution $\overline{\,\cdot\,}$,
such that the quotient of $R$ by any maximal ideal is a finite field.
Let $R^-$ be the subring of elements fixed by the involution and assume that $R = R^-[\xi]$ for some $\xi\in R\setminus R^-$.
Let $A$ be a finitely presented $R$--torsion module with $A = \overline{A}$
whose order is coprime with~$\xi-\xi^{-1}$.
Then any two non-degenerate $R$--Hermitian pairings $A\times A \to Q(R) / R$ are isometric.
\end{prop}
For the reader's convenience, let us briefly state some essential properties of Dedekind rings
that will come into play. 
Dedekind rings $R$ may be defined as integral domains satisfying any one of the following, equivalent conditions:
\begin{itemize}
\item $R$ is integrally closed, Noetherian, and has Krull dimension at most one.
\item $R$ is integrally closed, Noetherian, and every non-zero prime ideal of $R$ is maximal.
\item Every ideal $I\subset R$ with $I\neq (0), (1)$ equals a product of prime ideals.
\end{itemize}
In particular, every PID is a Dedekind ring.
The reverse is not true. For example, $R = \Lambda / \Phi_{23}$ is Dedekind, but not a PID.
More generally, the ring of algebraic integers of any number field is Dedekind, but such rings need not be PIDs.
This is one of the reasons that, in this context,
Dedekind rings are the more natural class of rings to work with than PIDs.
Statements that hold for PIDs often remain true for Dedekind rings when one replaces `prime number' by `prime ideal'.
In particular, over a Dedekind ring $R$, one can classify finitely generated modules:
all such modules are the sum of a free module of finite rank, and finitely many cyclic modules, each of which is the quotient of $R$ by the positive power of a non-zero prime ideal.

For these statements, as well as the further algebraic number theory needed in this section,
we refer the reader to any introductory text, such as Neukirch's~\cite{neukirch}.

Let us now deduce \cref{thm:allisometric} from \cref{prop:allisometric}.

\begin{proof}[Proof of \cref{thm:allisometric}]
First note that as discussed in \cref{subsec:bisesqui}, the 
isometry classes of the linking pairings $\lkh$ and $\lkh'$ determine each other; so it suffices to prove
that any two non-degenerate Hermitian pairings on
the $R$--module
$H_1(\Sigma_p(K); \Z)$ are isometric
(where $R \coloneqq \Lambda/\Phi_p$).
Now $R$ is the ring of algebraic integers
of the $p$--th cyclotomic field (cf.~\cite[Ch.~I, (10.2)]{neukirch}),
and as such it is a Dedekind ring whose quotient by any maximal ideal is a finite field
(cf.~\cite[Ch.~I, Theorem~(3.1) and its proof]{neukirch}).
The ring $R$ inherits the involution $t\mapsto t^{-1}$
from $\Lambda$. Denote the fixed point ring by~$R^-$.
Since $R^-[t]$ is the smallest subring of $R$ containing $R^-$ and $t$, and every element in $R$ is a polynomial in $t$ with integer coefficients, and $\Z \subset R^-$,
we have $R^-[t] = R$. With $t$ and $H_1(\Sigma_p(K); \Z)$ taking the roles of $\xi$ and $A$ in 
\cref{prop:allisometric}, respectively, it just remains to show that the order of $H_1(\Sigma_p(K); \Z)$ over $R$
is coprime with $t - t^{-1}$, i.e.~that $1$ is an $R$-linear combination of those two elements.

Note that the order of $H_1(\Sigma_p(K); \Z)$ over $R$ equals $\Delta_K \in R$,
since a presentation matrix of the Blanchfield pairing of $K$,
which has determinant $\Delta_K$, descends to a presentation matrix of $H_1(\Sigma_p(K); \Z)$ (see \cref{prop:inducedpresmatrix}).
Firstly, $\Delta_K$ and $(t - 1)$ are coprime in $\Lambda$, since $\Delta_K(1) = 1$,
and so $\Delta_K = (t-1)\cdot f + 1$ for some $f \in \Lambda$,
and thus $\Delta_K - (t-1)\cdot f = 1$. It follows that $\Delta_K$ and $(t - 1)$ are also coprime in $R$.
Secondly, we have $\Phi_p(-1) = 1$ (here, we need $p\neq 2$), and thus $\Phi_p = (t + 1) \cdot g + 1$
for some $g\in \Lambda$, and hence $t+1$ is a unit in $R$.
Therefore, $\Delta_K$ and $(t-1)\cdot (t+1)\cdot t^{-1} = t - t^{-1}$ are coprime in $R$.
This concludes the proof.
\end{proof}
The remainder of this section contains the proof of \cref{prop:allisometric}.
In a nutshell, that proof goes as follows.
We will first follow Wall's classification of symmetric pairings of
finite abelian groups of odd order~\cite{MR0156890}, showing that pairings may
be diagonalized, i.e.\ decomposed as orthogonal sum of pairings on cyclic modules
(for this, we will of course need that $R$ is Dedekind).
On a cyclic module, there is only one isometry class of pairings, because every element is a norm modulo a fixed prime ideal
(for this, we will need that $R$ is not fixed by conjugation).

\begin{lemma}\label{lemma:dedekind}
With the hypotheses of \cref{prop:allisometric}, $R^-$ is a Dedekind ring.
\end{lemma}
\begin{proof}
The element $\xi$ is a root of the polynomial
$p(x) = (x-\xi)(x-\overline{\xi}) \in R[x]$. Since the coefficients of $p(x) = x^2 + (-\xi-\overline{\xi})x + \xi\overline{\xi}$ are invariant under the involution, we actually have $p(x) \in R^-[x]$.
It follows that there is a ring isomorphism $R \cong R^-[x] / (p(x))$.
Because $p(x)$ is monic, this yields an $R^-$-module isomorphism $R \cong R^- \times R^-$.
Now the Eakin-Nagata theorem implies that $R^-$ is Noetherian.
Moreover, the ring extension $R:R^-$ is integral, whence the Krull-dimension of $R^-$ is the same as that of $R$.
Finally, if $u\in Q(R^-)$ is integral over $R^-$, then $u$ is also integral over~$R$,
so $u\in R$. Since $R^- = Q(R^-) \cap R$, this implies $u\in R^-$, and thus $R^-$ is integrally closed.
\end{proof}

So we find ourselves in the usual situation of an extension of Dedekind rings,
where $R^-$ is a Dedekind ring, $K:Q(R^-)$ is a finite field extension (where $K = Q(R)$),
and $R$ is the integral closure of $R^-$ in $K$. The relationship of prime ideals
of $R^-$ and $R$ in this situation is well-understood (cf.~\cite[Ch.~I, Sec.~8]{neukirch}):
let $\mathfrak{p} \subset R^-$ be a prime ideal, and
$\mathfrak{P} \subset R$ the ideal generated by~$\mathfrak{p}$.
Then there are three scenarios:
$\mathfrak{p}$ may be \emph{inert}, i.e.\ $\mathfrak{P}$ is prime;
or $\mathfrak{p}$ may \emph{split}, i.e.\ $\mathfrak{P}$ is the product of two distinct prime ideals of $R$ that are interchanged by the involution;
or $\mathfrak{p}$ may \emph{ramify}, i.e.\ $\mathfrak{P}$ is the square of a prime ideal of $R$.

\begin{lemma}\label{lem:NT}
Let $\mathfrak{p}, \mathfrak{P}$ be non-zero ideals as above with non-ramifying~$\mathfrak{p}$.
Write $E = R / \mathfrak{P}$ and $F = R^- / \mathfrak{p}$.
Define the \emph{trace} $T\colon E \to F$ and the
\emph{norm} $N\colon E^{\times} \to F^{\times}$ as
$T(x) = x + \overline{x}$ and $N(x) = x\cdot\overline{x}$, respectively.
Then $T$ and $N$ are surjective.
\end{lemma}
\begin{proof}
Let us first show the surjectivity of $T$.
Since $R^-$ is Dedekind, $\mathfrak{p}$ is maximal, and so $F = R^- / \mathfrak{p}$ is a field.
The ring $E$ is a two-dimensional $F$--vector space.
Assume that $T$ is not surjective; then it is the zero map.
Since $T(1) = 1 + 1 = 0 \in F$, $F$ must have characteristic $2$.
So for any $x\in E$, $T(x) = x + \overline{x} = 0$ implies that $\overline{x} = x$.
However $R\neq R^-$, so this is not the case.

To show surjectivity of $N$, let us consider the two possible cases that $\mathfrak{p}$ is inert, or it splits.
If $\mathfrak{p}$ is inert, then $\mathfrak{P}$ is a maximal ideal of $R$,
and thus by the assumptions of \cref{prop:allisometric}, $E = R / \mathfrak{P}$ is a finite field.
Since $F$ is a subfield of $E$, $F$ is also finite.
The field extension $E : F$ is of degree two, since $E = F(\xi)$, and $\xi$ is the root of a degree two polynomial with coefficients in $F$ (see the proof of \cref{lemma:dedekind}).
Hence $|E| = |F|^2$, and the Galois group of the field extension is isomorphic to $\Z/2$.
In fact, this Galois group is well-understood: its unique non-trivial element is the map $x \mapsto x^{|F|}$,
which is a power of the Frobenius automorphism \cite[Ch.~V~Thm~5.5]{zbMATH01703931}.
Since the involution $\overline{\,\cdot\,}$ is an automorphism of $F$ that fixes $E$, but not all of $F$, the involution must equal the unique non-trivial element of that Galois group, i.e.~$\overline{x} = x^{|F|}$.
So $x\in \ker N$ holds if and only if $x^{|F|+1} = 1$.
Thus $\ker N(x)$ consists of the roots of the polynomial $x^{|F|+1} - 1$ in $E$.
Since over a field, the number of roots of a polynomial is bounded from above by its degree,
it follows that $\ker N$ contains at most $|F| + 1$ elements.
On the other hand, $|E^{\times}| = |E| - 1 = |F|^2 - 1$ and $|F^{\times}| = |F| - 1$, so the kernel contains at least $|F| + 1$ elements.
It follows that $|\ker N| = |F| + 1$ and $|\im N| = |F| - 1$, so $N$ is surjective.
If $\mathfrak{p}$ splits, then the ring $E$ is isomorphic to $F\times F$,
with conjugation interchanging the two components, and $F\subset E$ identified with $\{(x,x)\mid x\in F\}$.
Clearly, $(x,x) = N((x,1))$, so $N$ is surjective.
\end{proof}
Now, using the freshly established non-triviality of the trace,
let us prove a base change lemma for homogeneous modules.
\begin{lemma}\label{lem:homchange}
Let $\mathfrak{p}, \mathfrak{P}$ be non-zero ideals as above with non-ramifying $\mathfrak{p}$.
Let $m, k \geq 1$, let $A = \bigoplus_{i=1}^m R / \mathfrak{P}^k$
and $\ell\colon A\times A \to Q(R) / R$ a non-degenerate Hermitian pairing.
Write $g_i$ for $[1]$ in the $i$--th summand.
Then there is an automorphism $\phi\colon A\to A$ such that
\[\mathfrak{P}^{k-1}\ell(\phi(g_m), \phi(g_m)) \neq (0).\]
\end{lemma}
\begin{proof}
If $\mathfrak{P}^{k-1}\ell(g_i, g_i) \neq (0)$ for any $i$, let $\phi$ simply exchange the $i$--th and $m$--th summand.
So consider the case that $\mathfrak{P}^{k-1}\ell(g_i, g_i) = (0)$ for all $i$.
Pick ${x \in \mathfrak{P}^{k-1} \setminus \mathfrak{P}^k}$.
Since $xg_m$ is a non-zero element of the $m$-th summand,
and $\ell$ is non-degenerate, we have $\ell(g_j, xg_m) \neq 0$ for some~$j$.
Thus $x \ell(g_j, g_m) \neq 0$, and hence $\mathfrak{P}^{k-1}\ell(g_j, g_m) \neq (0)$.
By \Cref{lem:NT}, there exists a
$\lambda \in R$ such that $\lambda + \overline{\lambda} \not\in \mathfrak{p} = R^- \cap \mathfrak{P}$.
Set $\phi(g_m) = g_m + \lambda g_j$ and $\phi(g_i) = g_i$ for $i < m$.
One computes $\mathfrak{P}^{k-1}\ell(\phi(g_m), \phi(g_m))$ to be equal to
\[
\mathfrak{P}^{k-1}( \ell(g_m, g_m)  + \lambda\overline{\lambda} \ell(g_j,g_j) + (\lambda + \overline{\lambda}) \ell(g_j, g_m))  \\
 = \mathfrak{P}^{k-1}\ell(g_j, g_m) \neq (0).
\pushQED{\qed}\qedhere\popQED
\]%
\renewcommand{\qedsymbol}{}
\end{proof}%
\vspace{-\baselineskip}
We are now ready to prove diagonalizability for general modules.
\begin{lemma}\label{lem:diag}
Let $A$ be an $R$--torsion module as in the \cref{prop:allisometric}.
Let $\ell$ be a non-degenerate Hermitian pairing on $A$.
Then there is a decomposition
\[
A\ \cong\ \bigoplus_{i=1}^m R / \mathfrak{P}_i^{k_i},
\]
where $k_i \geq 1$ and the $\mathfrak{P}_i \subset R$ are ideals
generated by maximal ideals $\mathfrak{p}_i\subset R^-$.
Moreover, if $x$ and $y$ are respective elements of the $i$--th and $j$--th summand with $i \neq j$,
then $\ell(x, y) = 0$.
\end{lemma}
\begin{proof}
Since $R$ is a Dedekind domain, $A$ is isomorphic to a unique sum of
terms of the form $R/\mathfrak{Q}^r$ with $\mathfrak{Q} \subset R$ maximal and $r\geq 1$.
Let $\mathfrak{p} \subset R^-$ be the maximal ideal $\mathfrak{Q} \cap R^-$,
and let $\mathfrak{P} \subset R$ be the ideal generated by $\mathfrak{p}$.
Depending on whether $\mathfrak{p}$ is inert, split, or ramified,
$\mathfrak{P}$ is equal to $\mathfrak{Q}$, $\mathfrak{Q}\overline{\mathfrak{Q}}$, or $\mathfrak{Q}^2$.

Let us check that $\mathfrak{p}$ does not ramify.
Ramification is controlled by the different ideal $\mathfrak{D}_{R:R^-} \subset R$ (cf.~\cite[Ch.~III, Sec.~2]{neukirch}):
$\mathfrak{p}$ ramifies if and only if $\mathfrak{Q} | \mathfrak{D}_{R:R^-}$.
Since $R = R^-[\xi]$, the different is the principal ideal generated by $f'(\xi)$,
where $f(x) = x^2 - (\xi + \xi^{-1})x + \xi\overline{\xi}\in R^-[x]$ is the minimal polynomial of $\xi$.
Thus the different is $(\xi - \xi^{-1})$. Note $\mathfrak{Q}$ divides the order of $A$,
which is by assumption coprime with $(\xi - \xi^{-1})$.
Thus $\mathfrak{Q}$ does not divide $\mathfrak{D}_{R:R^-}$, and so $\mathfrak{p}$ does not ramify.

Since $A = \overline{A}$, for every term $R/\mathfrak{Q}^r$ either it holds that $\overline{\mathfrak{Q}} = \mathfrak{Q}$,
or the term $R / \overline{\mathfrak{Q}}^r$ also appears.
So for each $\mathfrak{Q}$ with  $\overline{\mathfrak{Q}} = \mathfrak{Q}$,
take $\mathfrak{P}_i = \mathfrak{Q}$; for each pair $\mathfrak{Q}, \overline{\mathfrak{Q}}$ with $\overline{\mathfrak{Q}} \neq \mathfrak{Q}$, take $\mathfrak{P}_i = \mathfrak{Q}\overline{\mathfrak{Q}}$.
In both cases, let $k_i = r$ and $\mathfrak{p}_i = \mathfrak{Q} \cap R^-$, so that $\mathfrak{p}_i$
is a prime ideal generating $\mathfrak{P}_i$ over $R$ (here, we use that $\mathfrak{p}_i$ does not ramify).
This gives a decomposition of $A$ as desired, which it remains to diagonalize.

If $x, y$ are given as above, and $\mathfrak{P}_i \neq \mathfrak{P}_j$, then
$\ell(x,y) \in Q(R) / R$ is annihilated by
$S = \mathfrak{P}_i^{k_i} + \mathfrak{P}_j^{k_j}$.
Let us show that $S = R$.
In the case $k_i = k_j = 1$, we have $\mathfrak{p}_i + \mathfrak{p}_j \subset S$.
Because $\mathfrak{p}_i$ and $\mathfrak{p}_j$ are distinct maximal ideals,
we have $\mathfrak{p}_i + \mathfrak{p}_j = R^-$. Since $1 \in R^-$, we have $1\in S$ and thus $S = R$.
In the case $k_i, k_j \geq 1$, one deduces $S = R$
by inductively applying that $I_1I_2+J=R$ if $I_1+J=I_2+J=R$ for ideals $I_1$, $I_2$ and $J$ (a general property of commutative rings). Since $\ell(x,y)$ is annihilated by $R$, it follows that 
$\ell(x,y) = 0$.
So, it suffices to solve the case that all $\mathfrak{P_i}$ are equal.
Let us write $\mathfrak{P} = \mathfrak{P}_i$, order the $k_i$ ascendingly, let $k$ be their maximum,
and let $r \in \{1,\ldots,m\}$ such that
\[
k_1 \leq k_2 \leq \ldots \leq k_{r-1} < k_r = k_{r+1} = \ldots = k_m = k.
\]

We note that $\ell$ restricts to a non-degenerate Hermitian pairing on the submodule $M\coloneqq\bigoplus_{i=r}^mR/\mathfrak{P}^k$.
To establish this, assume towards a contradiction that we have $x\in M\setminus \{0\}$ with $\ell(x,y)=0$ for all $y\in M$.
Let $n$ be the maximal integer such that $x=(x_r,\ldots,x_m)$ is an element of $\bigoplus_{i=r}^m\mathfrak{P}^n/\mathfrak{P}^k$.
By the maximality of $n$, $\mathfrak{P}^{k-1-n} x$ is a nontrivial submodule of $M$.
Let $x'$ be any non-trivial element of it. We can write $x'=px$ for some $p\in\mathfrak{P}^{k-1-n}$.
We note that $\ell(x',y)=\ell(x,py)=0$ for all $y$ in~$M$.
Furthermore, since $px_i\in \mathfrak{P}^{k-1}\subset\mathfrak{P}^{k_i}$ for all $i\leq r-1$, we find
\begin{align*}
\ell(x',y) &= \sum_{i=r}^m \ell((\ldots,0,px_i,0,\ldots),y) && = \sum_{i=r}^m \ell((\ldots,0,1,0,\ldots),px_iy) \\
 &= \sum_{i=r}^m \ell((\ldots,0,1,0,\ldots),0) && =\sum_{i=r}^m 0
\end{align*}
for all $y$ in $\bigoplus_{i=1}^{r-1}R/\mathfrak{P}^{k_i}$.
Thus, $x'$ is in the radical of $\ell$ contradicting the assumption of $\ell$ being non-degenerate.

Let $g_i$ be a generator of the $i$--th summand. By \Cref{lem:homchange}, there is a base change on summands number $r$ to $m$ after which
$\mathfrak{P}^{k_r - 1}\ell(g_m, g_m) \neq (0)$.
This implies that $\lambda_i = \ell(g_m, g_i) / \ell(g_m, g_m)$ lies in $R$ for all $i\in\{1,\ldots,m-1\}$.
Now one may replace $g_i$ with
$g_i - \lambda_i g_m$ (as in the Gram-Schmidt algorithm). After this base change, one has
$\ell(g_i, g_m) = 0$  and may proceed by induction over $m$.
\end{proof}
We now conclude the proof of the proposition.
\begin{proof}[Proof of \Cref{prop:allisometric}]
By \Cref{lem:diag}, it suffices to prove that any two pairings $\ell_1, \ell_2$ on a module of the form
$A = R / \mathfrak{P}^k$ are isometric. A pairing on a cyclic module such as $A$ is completely determined
by its value $\ell_i(1,1)$, which is annihilated by~$\mathfrak{P}^k$, but not by~$\mathfrak{P}^{k-1}$ .
Moreover, $\overline{\ell_i(1,1)} = \ell_i(1,1)$, and thus
we can represent $\ell_i(1,1)$ as $x_i / y_i$ with $x_i\in R^- \setminus \mathfrak{p}$ and $y_i \in \mathfrak{p}^k$.
The quotient of $x_1/y_1$ by $x_2/y_2$ is $\mu = \frac{x_1y_2}{x_2y_1}$, which lies in $R^- \setminus \mathfrak{p}$.
By \Cref{lem:NT}, there exists a $\lambda \in R$ such that
$[\lambda\cdot \overline{\lambda} \cdot \mu] = [1] \in R^- / \mathfrak{p}$.
Let $\phi$ be the ring endomorphism of $A$ given by multiplication with~$\lambda$.
This endomorphism is an automorphism since $\lambda$ is a unit in $A$. Indeed, the latter follows since $(\lambda)+\mathfrak{P}=R$ implies $(\lambda)+\mathfrak{P}^k=R$.
 Then $\ell_1(\phi(1), \phi(1)) = \ell_2(1,1)$, so $\phi$ is an isometry
between $\ell_1$ and $\ell_2$ as desired.
\end{proof}

\begin{appendix}
\section{Base change for Hermitian pairings}\label{app}

One purpose of this appendix is to establish~\cref{prop:BC},
which allows base changes for Hermitian pairings of $R$--modules
over a different (in the applications usually bigger) ring $R'$.
As mentioned
in \Cref{subsec:(4)=>(3)},
this method is already in implicit use by experts for the Blanchfield pairing;
in \Cref{subsec:a2}, we make the method more explicit and generalize it to arbitrary rings $R$.
Along the way, in \Cref{subsec:a1}, we discuss how one may replace
$Q(R)/R$ as target of Hermitian pairings by $R/I$ for a suitably chosen ideal $I$.
Once again, this method has been used before;
we provide a general setup and discuss how the ideal~$I$ can be chosen,
particularly if $R$ is not a PID.
Finally, using (among other things) \Cref{subsec:a2}, we discuss how the linking form on the $n$--fold branched covering can be understood as a specialization of the Blanchfield paring, and, in particular, prove \cref{prop:inducedpresmatrix}; see \Cref{sec:ApP:BlandLk}.

Let us start with a simple, motivating example to illustrate the results of both
\Cref{subsec:a1}
and
\Cref{subsec:a2}.
Consider the pairing $\ell_A$ presented by a matrix $A$:
\[
\ell_A\colon \Z^2/A\Z^2 \times \Z^2/A\Z^2 \to \Q/\Z, \
\ell_A(x,y) = x^{\top} A^{-1} y,
\ \text{where }
A = \begin{pmatrix} 12 & 3 \\ 3 & 24 \end{pmatrix}.
\]

Firstly, note that $\im\ell_A = (\frac{1}{93}\Z)/\Z \subset \Q/\Z$, where $93$ generates the annihilator of
the cokernel of $A$.
Since the abelian groups $(\frac{1}{93}\Z)/\Z$ and $\Z/93\Z$ are isomorphic, one may instead of $\ell_A$ equivalently
consider the map $\Z^2/A\Z^2 \times \Z^2/A\Z^2 \to \Z/93$ given by $\ell_A(x,y) = 93\cdot x^{\top} A^{-1} y$.

Secondly, a matrix $B$ congruent to $A$ over $\Z$
clearly yields an isometric pairing~$\ell_B$. But now, consider the matrix
\[
B = \begin{pmatrix} 3 & 3 \\ 3 & 96\end{pmatrix},
\]
which is congruent to $A$ over $\Z[\frac{1}{2}]$.
One observes that $\ell_B$ is isometric to $\ell_A$, even though $A$ and $B$ are not congruent over $\Z$
(because $A$ is even and $B$ is odd).

Let us now formulate general principles based on these two observations.
\subsection{Change of perspective: Hermitian pairings as maps to $R/I$}\label{subsec:a1}
Let $R$ be a unital commutative rings with involution.
Usually, one considers Hermitian pairings of torsion modules $M$ with target $Q(R)/R$
(as we did in \cref{subsec:pairings}):
\[
\ell\colon M \times M \to Q(R)/R.
\]
However, it will turn out to be more convenient
to consider Hermitian pairings on $M$ with target $R/I$, where $I$ is an ideal contained in $\Ann(M)$:
\[
\ell'\colon M\times M\to R/I.
\]
Note that this makes sense even when $M$ is not torsion.

If $I$ is a principal ideal generated by a non-zero-divisor $s$,
then there is the following 1-to-1-correspondence
between Hermitian pairings $\ell$ of $M$ with target $Q(R)/R$
and Hermitian pairings $\ell'$ of $M$ with target $R / I$.
Denote by $\iota_s$ the injection given by `dividing by $s$':
\[
\iota_s\colon
R/I=R/(s) \hookrightarrow Q(R)/R, \quad r+(s)\mapsto \frac{r}{s} + R.
\]
Let a pairing $\ell'$ as above correspond to $\iota_s\circ\ell'$.
This mapping from pairings with target $R/I$ to pairings with target $Q(R) / R$
is clearly injective; it is also surjective, because $I\subset \Ann(M)$.

If $M$ is torsion, but $I$ is not principal, the Hermitian pairings with targets $Q(R)/R$ and $R/I$
may not be in such a natural 1-to-1 correspondence; but since we do not encounter such $M$ in this text,
we refrain from pursuing this further.

Note that the 1-to-1-correspondence  depends on the choice of $s$.
However, when applying this change of perspective to the Blanchfield pairing of a knot~$K$,
where $M = H_1(S^3\setminus K;\Lambda)$ is the Alexander module,
there are two natural choices of $I$ as a principal ideal with a canonical generator.
One may choose $I$ to be the order ideal, canonically generated by the Alexander polynomial
$\Delta_K(t)$ with $\Delta_K(t) = \Delta_K(t^{-1})$ and $\Delta_K(1) = 1$.
Thus the Blanchfield pairing of a knot $K$ may be considered in a canonical fashion as a Hermitian pairing
\[
\Bl'(K)\colon H_1(S^3\setminus K; \Lambda) \times H_1(S^3\setminus K; \Lambda) \to \Lambda / (\Delta(t)).
\]

One may also choose $I$ to be the annihilator of $H_1(S^3\setminus K; \Lambda)$. It follows from \cref{lem:ufd} below
that $I$ is principal. Since $\Delta_K \in \Ann(H_1(S^3\setminus K; \Lambda))$, any generator of $\Ann(H_1(S^3\setminus K; \Lambda))$ divides $\Delta_K$.
Let us choose as canonical generator of the annihilator of the Alexander module the
unique generator $a(t)$ with $a(t) = a(t^{-1})$ and $a(1) = 1$.
So one may equally well consider the Blanchfield pairing of a knot $K$ in a canonical fashion as a Hermitian pairing
\[
\Bl''(K)\colon H_1(S^3\setminus K; \Lambda) \times H_1(S^3\setminus K; \Lambda) \to \Lambda / (a(t)).
\]
Note that composing $\Bl''$ with multiplication by $\frac{\Delta(t)}{a(t)}$ gives $\Bl'$.

\begin{lemma}\label{lem:ufd}
Let $R$ be a unital commutative unique factorization domain with involution,
and $A$ an $n\times n$ matrix over $R$ with non-zero-divisor determinant.
Then the annihilator ideal of the cokernel of $A$ is principal.
\end{lemma}
\begin{proof}
Write $M$ for the cokernel of $A$. We have $r\in \Ann(M)$ if and only if
for all $v\in R^n$ it holds that $rv \in A R^n$, or equivalently $rA^{-1}v \in R^n$
(where $A^{-1}$ is a matrix over~$Q(R)$). Now, $A^{-1}v$ is a vector in $Q(R)^n$ with entries
of the form $p / q$. So $\Ann(M)$ consists of the intersection of all the principal ideals $(q)$ for the
$q\in R$ that appear in this way for some~$v$. However, over a unique factorization domain, an intersection of
principal ideals is again principal.
\end{proof}

Let us generalize from the Blanchfield pairing to
Hermitian pairings $\ell_A$ on an $R$--torsion module $M$ presented by a square matrix $A$ with non-zero-divisor determinant.
Taking $I$ as the order ideal $\Ord(M)$, which is principal
and generated by $\det(A)$, gives the following formula for $\ell_A$ in terms of $A$:
\[
\ell_A\colon R^n/AR^n\times R^n/AR^n \to R/(\det(A)),\quad
(x,y) \mapsto \overline{x}^{\top}\Adj(A)y+(\det(A)),
\]
where $\Adj(A)$ denotes the \emph{adjoint} of $A$. Recall that $A\Adj(A)$ equals $\det A$ times the identity matrix.
This explains how this formula and the `old' formula (cf.\ \cref{subsec:pairings})
\[\ell_A\colon R^n/AR^n\times R^n/AR^n\to Q(R)/R,\quad (x,y)\mapsto \overline{x}^{\top}A^{-1}y + R\]
translate into each other via the above 1-to-1 correspondence
coming from the generator $\det A$ of $\Ord(M)$.

\subsection{Changing bases over a different ring}\label{subsec:a2}
The advantage of the perspective on Hermitian pairings given in the previous subsection
is that a change of the base ring of the module naturally carries over to the target of Hermitian pairings,
even when the order is no longer a non-zero-divisor.

Let $\phi\colon R\to R'$ be a homomorphism of unital commutative rings $R$ and $R'$ with involution
and let $\ell\colon M\times M\to R/I$ be a Hermitian pairing on an $R$--module~$M$.
Then there is an induced Hermitian pairing $\ell^{\phi}$ on the $R'$--module $M' = M \otimes_R R'$ given by
\begin{align*}
\ell^{\phi} \colon M' \times M' &\to R' / \phi(I)R',\\
(x\otimes r, y\otimes s) &\mapsto \overline{r}\cdot s\cdot\phi(\ell(x,y)) + \phi(I)R'.
\end{align*}
For well-definedness, observe that $\phi(I)R' \subset \Ann_{R'}(M')$.

\begin{prop}\label{prop:BC}
Let $\phi\colon R\to R'$ be a homomorphism of unital commutative rings with involution $R$ and $R'$. %
For $i \in \{1,2\}$, let $\ell_i\colon  M_i \times M_i \to R / I_i$ be Hermitian pairings of $R$--modules $M_i$,
where $I_i$ is an ideal contained in $\Ann(M_i)$.
Assume that $\phi$ induces an isomorphism of $R$--modules between $R/I_i$ and $R'/\phi(I_i)R'$ for $i = 1,2$.
Then $\ell_1^{\phi}$ and $\ell_2^{\phi}$ are isometric over $R'$
if and only if $\ell_1$ and $\ell_2$ are isometric over~$R$.
\end{prop}
\begin{proof}
It is straightforward that an isometry between $\ell_1$ and $\ell_2$ induces an isometry
between $\ell_1^{\phi}$ and $\ell_2^{\phi}$.
For the other direction, note that the map $M_i \to M_i \otimes_R R'$ given by $x \mapsto x\otimes 1$
is an $R$--module isomorphism, since it can be written as the composition of the following $R$--isomorphisms:
\[
M_i\cong M_i\otimes_R R/I_i \cong M_i\otimes_R R'/\phi(I_i)R'\cong M_i\otimes_R R'.
\]
An isometry between $\ell_1^{\phi}$ and $\ell_2^{\phi}$ is in particular an $R'$--isomorphism $M_1 \otimes_R R' \to M_2 \otimes_R R'$.
Composing with the isomorphisms $M_1 \to M_1 \otimes_R R'$ and $M_2 \otimes_R R' \to M_2$ gives an
$R$--isomorphism $M_1 \to M_2$. It is straight-forward that this isomorphism behaves well with respect
to $\ell_1, \ell_2$, and is thus an isometry.
\end{proof}

Let us consider two corollaries needed in the paper.
The first allows one to think of the sesquilinear linking pairing of $p$--fold branched coverings of the knot as module over the ring $\Lambda/(g)$ rather than $\Lambda/(t^p-1)$ where $p$ is a prime and $g = 1 + \ldots + t^{p-1}$ is the $p$--th cyclotomic polynomial
(see \cref{subsec:bisesqui}). For this purpose, take $f = \Delta_K$:

\begin{corollary}\label{cor:pfold}
Let $g\in\Lambda$ and $\phi\colon \Lambda / ((t-1)g) \to \Lambda/ (g)$ be the canonical projection.
For $i\in\{1,2\}$, let $\ell_i\colon M_i\times M_i \to\Lambda/(f, (t-1)g)$
be Hermitian pairings of $\Lambda/((t-1)g)$--modules for an $f\in\Lambda$ with $f(1) = 1$ and $f\in\Ann(M_i)$.
Then $\ell_1^{\phi}$ and $\ell_2^{\phi}$ are isometric over $\Lambda/(g)$
if and only if $\ell_1$ and $\ell_2$ are isometric over~$\Lambda/((t-1)g)$.
\end{corollary}
\begin{proof}
The statement is a direct application of \Cref{prop:BC}
with $R = \Lambda/((t-1)g)$ and $R' = \Lambda/(g)$;
one merely needs to check that the map induced by $\phi$ between
$R / fR$ and $R' / fR'$ is an isomorphism---this is the canonical projection
\[
\Lambda/((t-1)g,f) \to \Lambda/(g, f).
\]
Note that $f(1) = 1$ implies $f - 1 = (t-1)r$ for some $r\in\Lambda$, so $g = -(t-1)gr + fg$.
This implies the equality of the ideals $((t-1)g,f) = (g,f)$.
\end{proof}

The following is a version of \Cref{prop:BC} for Hermitian pairings presented by square matrices, which we use in our proof of $(4)\Rightarrow(3)$ of \cref{thm:main}; see \cref{lem:lambda0}.

\begin{corollary}\label{cor:BC}
Let $\phi\colon R\to R'$ be a homomorphism of unital commutative rings $R$ and $R'$ with involution.
Let $A$ and $B$ be Hermitian $n\times n$ $R$--matrices and denote by $A_\phi$ and $B_\phi$ the $R'$--matrices obtained from applying $\phi$ entry-wise to $A$ and $B$, respectively.
Assume that $\phi$ induces isomorphisms $R/(\det(A))\cong R'/(\phi(\det(A)))$ and $R/(\det(B))\cong R'/(\phi(\det(B)))$.
Then:\smallskip

(i) $\ell_A$ is isometric to $\ell_B$ if and only if $\ell_{A_\phi}$ is isometric to $\ell_{B_\phi}$.

(ii) $\ell_A$ is isometric to $\ell_B$ if
there exists $T\in GL_n(R')$ with $B_\phi=\overline{T}^{\top} A_\phi T$.
\end{corollary}
\begin{proof}
The existence of $T$ implies that $\ell_{A_\phi}$ is isometric to $\ell_{B_\phi}$, so (ii) follows from~(i).
Let us show (i). Denote by $M$ be the cokernel of $A$.
The Hermitian pairing~$(\ell_A)^\phi$ is isometric to
$\ell_{A_\phi}$ via the canonical isomorphism from $M\otimes R'$ to the cokernel of $A_\phi$
(using right-exactness of the tensor product with $R'$).
Similarly, $(\ell_B)^\phi$ is isometric to~$\ell_{B_\phi}$.
Since $\Ord(M)=(\det(A))$, and similarly for the cokernel of $B$,
the statement follows from \Cref{prop:BC}.
\end{proof}
\subsection{The linking form on a branched covering as specialization of the Blanchfield pairing}\label{sec:ApP:BlandLk}
In this subsection, we will examine how the Blanchfield form of a knot induces the sesquilinear linking pairing $\lkh$ of the $n$--fold branched covering, and give a proof of \cref{prop:inducedpresmatrix}.

First, we explain how the Blanchfield pairing $\Bl$ induces a pairing $\Bl_n$ on $T_{\Delta}H_1(X_K;\Lambda_n)$.
Here, we have fixed a knot $K$ and a surjection $\pi_1(S^3\setminus K)\to \Z/n\Z$ for $n$ a prime power, and we have identified the group ring $\Z[\Z/n\Z]$ with $\Lambda_n = \Lambda / (t^n - 1)$.
For the sake of brevity we write $\Delta \in \Lambda$ for the Alexander polynomial of $K$, and $X_K = S^3\setminus K$.
For $M$ a module over $\Lambda$ or $\Lambda_n$, $T_{\Delta}M \subset M$ denotes the submodule of elements annihilated by $\Delta$.
Note that $H_1(X_K;\Lambda_n)$ contains a $\Lambda_n / (t-1)$ summand generated by a meridian of the boundary torus of the $n$--fold covering $X_K^n$ of $X_K$. We work with $T_{\Delta}H_1(X_K;\Lambda_n)$ in order to exclude this summand.

As explained in~\Cref{subsec:a1}, one may see the Blanchfield pairing as a Hermitian pairing
\[
\Bl'(K)\colon H_1(X_K; \Lambda) \times H_1(X_K; \Lambda) \to \Lambda / (\Delta).
\]
Indeed, we find $\Bl'$ from the definition of $\Bl$ in \cref{subsec:defBl} by identifying a submodule of $\Q(t)/\Lambda$ with $\Lambda / (\Delta)$. This amounts to $\Bl'(K)(x,y)=(\Psi'(x))(y)$,
where $\Psi'$ is the composition of the following maps
\begin{multline*}
H_1(X_K;\Lambda)\to H_1(X_K,\partial X_K;\Lambda)\xrightarrow{\cong}
H^2(X_K;\Lambda)\\\xrightarrow{\cong}
H^1(X_K;\Lambda / (\Delta))
\xrightarrow{\mathrm{ev}}
\overline{\Hom_\Lambda(H_1(X_K;\Lambda),\Lambda / (\Delta))},
\end{multline*}
with all the maps as described in \Cref{subsec:defBl}, except the third, which is taken to be the inverse of the connecting homomorphism of the long exact sequence associated with the short exact sequence
\[
\begin{tikzcd}[column sep=scriptsize]
0 \ar[r] & \Lambda \ar[r,"\cdot\Delta"] & \Lambda \ar[r] & \Lambda / (\Delta) \ar[r] & 0.
\end{tikzcd}
\]
Similarly, one may define $\Psi_n^{'\mathrm{twist}}$ as the composition of the maps
\begin{multline*}
T_{\Delta}H_1(X_K;\Lambda_n)\to H_1(X_K,\partial X_K;\Lambda_n)
\xrightarrow{\cong}
H^2(X_K;\Lambda_n)\\\xrightarrow{\cong}
H^1(X_K;\Lambda_n/(\Delta))
\xrightarrow{\mathrm{ev}}
\overline{\Hom_{\Lambda_n}(T_{\Delta}H_1(X_K;\Lambda_n),\Lambda_n/(\Delta))}.
\end{multline*}
where, again, all the maps are as described in \Cref{subsec:defBl}, except the third, which is taken to be a corresponding connecting homomorphism of the long exact sequence associated with the short exact sequence
$\begin{tikzcd}[column sep=scriptsize]
0 \ar[r] & \Lambda_n \ar[r,"\cdot\Delta"] & \Lambda_n \ar[r] & \Lambda_n / (\Delta) \ar[r] & 0.
\end{tikzcd}$
This gives rise to a pairing
\[
\Bl'_n(K)\colon T_{\Delta}H_1(X_K; \Lambda_n) \times T_{\Delta}H_1(X_K;\Lambda_n) \to  \Lambda_n/(\Delta)\] defined by $\Bl'_n(K)(x,y)\coloneqq(\Psi_n^{'\mathrm{twist}}(x))(y)$.
In other words, using $\phi\colon\Lambda\to\Lambda_n, r\mapsto [r]$ as ring homomorphism for the setup in~\Cref{subsec:a2}, we have $\Bl'_n(K)=\Bl'(K)^\phi$.
Now we have the following.
\begin{lemma}\label{lemma:blnprime}
There is a commutative diagram
\[
\begin{tikzcd}
H_1(X_K;\Lambda) \ar[r,"\Psi'"] \ar[d,"\alpha"] & \overline{\Hom_\Lambda(H_1(X_K;\Lambda),\Lambda / (\Delta))} \ar[d,"\beta"] \\
T_{\Delta}H_1(X_K;\Lambda_n) \ar[r,"\Psi_n^{'\mathrm{twist}}"] & \overline{\Hom_{\Lambda_n}(T_{\Delta}H_1(X_K;\Lambda_n),\Lambda_n/(\Delta))},
\end{tikzcd}
\]
where the map $\alpha$, induced by the projection $\phi\colon \Lambda \to \Lambda_n$, is surjective,
and the map $\beta$ is given by sending $f\colon H_1(X_K; \Lambda) \to \Lambda/(\Delta)$
to 
$T_{\Delta}H_1(X_K;\Lambda_n) \to \Lambda_n/(\Delta)$, $x \mapsto \phi(f(\alpha^{-1}(x)))$.
In particular, for all $x,y\in H_1(X_K;\Lambda)$ it holds that
\[
\Bl'_n(\alpha(x), \alpha(y)) = \overline{\phi}(\Bl'(x,y)) \in \Lambda_n/(\Delta),
\]
where $\overline{\phi}\colon \Lambda/(\Delta) \to \Lambda_n/(\Delta)$ is induced by the quotient map $\phi\colon\Lambda\to\Lambda_n$.
\end{lemma}
For a detailed proof of this lemma, we refer the reader to Proposition~A.1 and Lemma~A.4 in \cite{homrib},
where a pairing on $H_1(X_K; R)$ is constructed for certain $(t-1)$-torsion free $\Lambda$--algebras $R$.
After sprinkling them with $T_{\Delta}$, the proofs in \cite{homrib} may also be applied to $R = \Lambda_n$ (which has $(t-1)$--torsion).

Again by~\Cref{subsec:a1}, $\Bl'_n(K)$ can be equivalently seen (essentially by dividing by $\Delta\in \Lambda_n$) as a pairing
\[
\Bl_n(K)\colon T_{\Delta}H_1(X_K; \Lambda_n) \times T_{\Delta}H_1(X_K;\Lambda_n) \to  Q(\Lambda_n)/\Lambda_n,\]
where $\Bl_n(K)(x,y)=(\Psi_n^\mathrm{twist}(x))(y)$ for $\Psi_n^\mathrm{twist}$ a corresponding composition of maps
\begin{multline*}
T_{\Delta}H_1(X_K;\Lambda_n)\to H_1(X_K,\partial X_K;\Lambda_n)
\xrightarrow{\cong}
H^2(X_K;\Lambda_n)\\\xrightarrow{\cong}
H^1(X_K;Q(\Lambda_n)/\Lambda_n)
\xrightarrow{\mathrm{ev}}
\overline{\Hom_{\Lambda_n}(T_{\Delta}H_1(X_K;\Lambda_n),Q(\Lambda_n)/\Lambda_n)}.
\end{multline*}
As a consequence of \cref{lemma:blnprime}, $\Bl_n$ is induced by $\Bl$ in the following sense:
\begin{corollary}\label{cor:BlnandBl}
For all $x,y\in H_1(X_K;\Lambda)$ we have $\Bl(x,y) \in \Delta^{-1}\Lambda/\Lambda$ and
\[
\Bl_n(\alpha(x), \alpha(y)) = \overline{\phi}(\Bl(x,y)),
\]
where $\overline{\phi}\colon \underset{\subset Q(\Lambda)/\Lambda}{\Delta^{-1}\Lambda/\Lambda} \to \underset{\subset Q(\Lambda_n)/\Lambda_n}{\Delta^{-1}\Lambda_n/\Lambda_n}$
is given by $[r / \Delta] \mapsto [\phi(r) / \Delta]$ for $r\in \Lambda$.\qed
\end{corollary}

Now that we have constructed the pairing $\Bl_n$ and have shown how it is induced by $\Bl$,
let us next make the connection between $\Bl_n$ and $\lkh$.
Note that we have defined $H_1(X_K; \Lambda_n)$ as $H_1(X_K^n; \Z)$ with its $\Lambda_n=\Z[\Z/n\Z]$--module structure, hence we can consider the map $i_*\colon H_1(X_K; \Lambda_n)\to
H_1(\Sigma_n(K);\Z)$, induced by the inclusion $i\colon  X_K^n\to \Sigma_n(K)$. The map $i_*$ induces a $\Lambda_n$--module isomorphism from $T_{\Delta}H_1(X_K; \Lambda_n)=TH_1(X_K^n; \Z)$ to $H_1(\Sigma_n(K);\Z)$.
\begin{lemma}\label{lemma:Blnlkh'}
The $\Lambda_n$--module isomorphism $i_*\colon 
T_{\Delta}H_1(X_K; \Lambda_n)\to H_1(\Sigma_n(K);\Z)$ is an isometry between
\[
(T_{\Delta}H_1(X_K; \Lambda_n), \Bl_n)\qquad\text{ and }\qquad (H_1(\Sigma_n(K);\Z),\lkh').
\]
\end{lemma}
We make use of the following description of twisted Poincar\'{e} duality.
\begin{lemma}\label{lemma:twistPD}
For a compact connected oriented $d$--manifold $X$ and a surjective group homomorphism $\varphi\colon \pi_1(X)\to G$ to a finite group $G$, twisted Poincar\'{e} duality
$\PD_{\varphi}\colon H^{k}(X;\Z[G])\to H_{d-k}(X,\partial X;\Z[G])$ satisfies
$\PD_{X^G}([f])=\PD_\varphi([\eta(f)])$. Here
\[
\PD_{X^G}\colon H^{k}(X^G;\Z)\to H_{d-k}(X^G,\partial X^G;\Z)
\]
denotes untwisted Poincar\'{e} duality on the $\ker(\varphi)$--covering $X^G$ of $X$ and
\[ \eta\colon C^k(X^G;\Z)\to C^k(X;\Z[G]) \]
is the map sending
$(f\colon C_k(X^G;\Z)\to \Z)\in C^k(X^G; \Z)$ to
\[
  \eta(f)\colon \overline{C_k(X;\Z[G])}\to \Z[G],\quad c\mapsto \sum_{g\in G}f(gc)\cdot g.
\]
\end{lemma}
\begin{proof}[Sketch of proof]
Denote by
\[
[X]\in H_{d}(X,\partial X;\Z)\quad\text{ and }\quad [X^G]\in H_{d}(X^G, \partial X^G;\Z)
\]
the fundamental class of $X$ and $X^G$, respectively.
Let $[X]=[S]\in H_{d}(X;\Z)$, where $S$ is a $\Z$--linear combination of singular simplexes of~$X$.
Let $S'$ be a lift of $S$ to the covering $X^G$.
The key observation is that $[X^G] = \Bigl[\sum_{g\in G}g S'\Bigr]$.

Recall that the untwisted Poincar\'{e} duality of $X^G$ is defined as
\[
\PD_{X^G}([f]) = [f] \frown [X^G],
\]
where $\frown$ denotes the cap product.
Likewise, twisted Poincar\'{e} duality may be defined via the twisted cap product (see e.g.~\cite{FNOP}).
In particular, $\PD_\varphi\colon H^{k}(X;\Z[G])\to H_{d-k}(X,\partial X;\Z[G])$
is given by sending $[h]$ to the twisted cap product of $[h]$ with $[X]$, which may be seen to equal $[h \frown S']$.
We then calculate
\begin{multline*}\PD_{X^G}([f])=[f]\frown [X^G]=\Bigl[\sum_{g\in G} f\frown gS'\Bigr]\\
=\Bigl[\sum_{g\in G}f((gS')|_{[v_0,\ldots,v_{k}]})\cdot(gS')|_{[v_{k},\ldots,v_{d}]}\Bigr]
=[\eta(f)\frown S']=\PD_\varphi([\eta(f)]),
\end{multline*}
as desired.
\end{proof}
\begin{proof}[Proof of \cref{lemma:Blnlkh'}]
Consider the following diagram, where the first and last column compose to $\Psi_n$ (as used to define $\lke$) and $\Psi_n^{\mathrm{twist}}$ (as used to define $\Bl_n$), respectively.
\[\begin{tikzcd}[column sep=small]
H_1(\Sigma_n(K); \Z)
\arrow[swap]{dd} {\PD^{-1}}&TH_1(X_K^n; \Z) \arrow{l}{i_*}\arrow{r}{\id} \arrow[swap]{d} & T_{\Delta}H_1(X_K; \Lambda_n) \arrow{d} \\
 & H_1(X_K^n, \partial X_K^n; \Z)\arrow{r}{\id}\arrow[swap]{d}{\PD^{-1}} & H_1(X_K,\partial X_K^n; \Lambda_n)\arrow[swap]{d}{\PD^{-1}_\varphi}\\
H^2(\Sigma_n(K); \Z)\arrow[swap]{d}\arrow{r}{i^*} &H^2(X_K^n; \Z) \arrow[swap]{d}\arrow{r}{\eta} & H^2(X_K;\Lambda_n)\arrow[swap]{d} \\
H^1(\Sigma_n(K);\Q/\Z)\arrow[swap]{d}\arrow{r}{i^*} &H^1(X_K^n;\Q/\Z)\arrow[swap]{d}\arrow{r}{\eta'}& H^1(X_K;Q(\Lambda_n)/\Lambda_n))\arrow[swap]{d}\\
\parbox{6.3em}{\hspace*{-3em}$\Hom_{\Z}(H_1(\Sigma_n(K);\Z),\\\hspace*{4em}\Q/\Z)$}\arrow{r}{(i_*)^*} & \parbox{8.7em}{$\Hom_{\Z}(TH_1(X_K^n;\Z),\\\hspace*{6.4em}\Q/\Z)$}\arrow{r}{\eta''}& \parbox{7.3em}{\hspace{2.5em}$\overline{\Hom_{\Lambda_n}(T_{\Delta}H_1(X_K;\Lambda_n),}$\\\hspace*{5.5em}$\overline{Q(\Lambda_n)/\Lambda_n)}$}
  \end{tikzcd}\]
Here, $\eta$ is as defined in \cref{lemma:twistPD} and $\eta'$, $\eta''$ are given by
  \begin{multline*} \eta'([f]) =[C_1\to Q(\Lambda_n)/\Lambda_n\colon c_1\mapsto \sum_{i=1}^nt^if(t^{i}c_1)] \\ \text{ for }[f\colon C_1(X^n_K;\Z)\to \Q/\Z]\in H^1(X_K^n;\Q/\Z),\end{multline*}
  \[\eta''(f)=[c_1]\mapsto \sum_{i=1}^nt^if(t^{-i}[c_1])\text{ for }f\in \Hom_{\Z}(TH_1(X_K^n;\Z),\Q/\Z).\]
   Note the different sign of the exponent of $t$ in the formulas for $\eta'$ to $\eta''$, which is due to the appearance of the ring involution in the formula that defines the evaluation map used in the definition of the Blanchfield paring.
   
  This diagram commutes. The second square on the right-hand side reads $\eta\circ\PD^{-1}=\PD^{-1}_\varphi\circ\id$ which follows from \cref{lemma:twistPD}. The bottom square on the right-hand side can be checked to commute by unpacking the definition of the evaluation map as discussed in~\cite{FriedlPowell_15}. All the other squares commute due to naturality under continuous maps and change of coefficients.
  
We now check $\lkh'(i_*(x),i_*(y))=\Bl_n(x,y)$ for all $x,y\in T_{\Delta}H_1(X_K; \Lambda_n)$.
We have
\begin{align*}\lkh'(i_*(x),i_*(y))&=\sum_{k=1}^n t^k\lke(t^ki_*(x),i_*(y))=\sum_{k=1}^n t^k\lke(i_*(x),t^{-k}i_*(y))\\
\intertext{by  the definition $\lkh'$ and equivariance of $\lke$,}
&=\sum_{k=1}^n t^k\Psi_n(i_*(x))(t^{-k}i_*(y))=\sum_{k=1}^n t^k\Psi_n(i_*(x))(i_*(t^{-k}y))\\
\intertext{by the definition of $\lke$ and linearity of $i_*$,}
&=\sum_{k=1}^n t^k\left((i_*)^*\left(\Psi_n(i_*(x))\right)(t^{-k}y)\right)\\
\intertext{by the definition of $(i_*)^*$ (i.e.~the $i_*$-induced map on $\Hom$),}
&=\left(\eta''\left((i_*)^*\left(\Psi_n(i_*(x))\right)\right)\right)(y)\\
\intertext{by the definition of $\eta''$,}
&=\Psi_n^\mathrm{twist}(x)(y),
\end{align*}
{by the commutativity of the outermost rectangle of the diagram.}
\end{proof}
\begin{proof}[Proof of \cref{prop:inducedpresmatrix}]
Let $A(t) \in \Lambda^{m\times m}$ be a presentation matrix for $\Bl$,
and let $B(t) \in \Lambda_n^m$ and $C(t) \in (\Lambda/\rho_n)^{m\times m}$ be the matrices obtained from $A(t)$
by applying, respectively, the quotient maps $\Lambda\to\Lambda_n$ and $\Lambda\to\Lambda/(\rho_n)$ to each matrix entry.
We wish to show that $C(t)$ presents $\lkh$.

First, let us check that 
$\coker C(t)$ and $H_1(\Sigma_n; \Z)$ are isomorphic $\Lambda/(\rho_n)$--modules.
We claim that the following diagram of $\Lambda$--modules commutes and each row is a short exact sequence.
\[
\begin{tikzcd}
0 \ar[r] & \Lambda^m \ar[r,"A(t)"]\ar[d] & \Lambda^m \ar[r,"q"] \ar[d]& H_1(X_K; \Lambda)\ar[r]\ar[d,"\alpha"] & 0 \\
0 \ar[r] & \Lambda_n^m \ar[r,"B(t)"]\ar[d,"\id"] & \Lambda_n^m \ar[r] \ar[d,"\id"]& T_{\Delta}H_1(X_K; \Lambda_n)\ar[r]\ar[d,"i_*"] & 0 \\
0 \ar[r] & \Lambda_n^m \ar[r,"B(t)"]\ar[d] & \Lambda_n^m \ar[r] \ar[d]& H_1(\Sigma_n; \Z)\ar[r]\ar[d,"\id"] & 0 \\
0 \ar[r] & (\Lambda/(\rho_n))^m \ar[r,"C(t)"] & (\Lambda/(\rho_n))^m \ar[r,"q_n"] & H_1(\Sigma_n; \Z) \ar[r] & 0
\end{tikzcd}
\]
The first row is exact by the assumption that $A(t)$ presents $\Bl$. The second row and the vertical maps from first to second row arise by tensoring the first row with $\Lambda_n$ over $\Lambda$. Tensoring preserves exactness because $\det A(t) = \Delta$ is a non-zero divisor over $\Lambda_n$.
The exactness of the third row and the commutativity between second and third row is a consequence of \cref{lemma:Blnlkh'}.
Finally, the fourth row and the vertical maps from third to fourth row arise by tensoring the third row with $\Lambda_n/(\rho_n)$ over~$\Lambda_n$. Tensoring preserves exactness because $\Delta$ is a non-zero divisor over $\Lambda_n/(\rho_n)$, since $n$ is a prime power. We have $H_1(\Sigma_n;\Z) \otimes_{\Lambda_n} \Lambda_n/(\rho_n) \cong H_1(\Sigma_n;\Z)$
since $\rho_n$ acts trivially on $H_1(\Sigma_n; \Z)$, as discussed in the paragraphs preceding \cref{prop:inducedpresmatrix}.
So, it follows that $\coker C(t)$ and $H_1(\Sigma_n; \Z)$ are isomorphic $\Lambda/(\rho_n)$--modules, as desired.

It just remains to show that for $x, y\in (\Lambda/(\rho_n))^m$ we have
\[
\lkh(q_n(x), q_n(y)) = \overline{x}^{\top}C(t)^{-1} y \quad\in Q(\Lambda/(\rho_n)) / \Lambda/(\rho_n).
\]
Observe that
\begin{align*}
\lkh(q_n(x), q_n(y)) & = p(\lkh'(q_n(x), q_n(y))) \\
\intertext{by definition of $\lkh$, for $p\colon Q(\Lambda_n)/\Lambda_n \to Q(\Lambda / \rho_n) / (\Lambda / \rho_n)$ the canonical projection,}
                 & = p(\Bl_n(i_*^{-1}(q_n(x)), i_*^{-1}(q_n(y)))) \\
\intertext{by \cref{lemma:Blnlkh'},}
                 & = p(\overline{\phi}(\Bl(\tilde x, \tilde y))).
\intertext{by \cref{cor:BlnandBl}, with $\tilde x, \tilde y \in \Lambda^m$ such that $\alpha(q(\tilde x)) = i_*^{-1}(q_n(x))$ and similarly for~$\tilde y$,}
                 & = p(\overline{\phi}(\overline{\tilde x}^{\top}A(t)^{-1} \tilde y))
\intertext{since $A(t)$ presents $\Bl$,}
                 & = \overline{x}^{\top}C(t)^{-1} y
\end{align*}
because of the commutativity of the above diagram. This concludes the proof.
\end{proof}

\end{appendix}
\bibliographystyle{myamsalpha}
\bibliography{peterbib}
\end{document}